\documentclass{amsart}

\usepackage{geometry, amssymb, esint, verbatim, mathrsfs}

\newtheorem{lemma}{Lemma}[section]
\newtheorem{theorem}[lemma]{Theorem}
\newtheorem{proposition}[lemma]{Proposition}
\newtheorem{corollary}[lemma]{Corollary}
\theoremstyle{definition}
\newtheorem{definition}[lemma]{Definition}
\newtheorem{remark}[lemma]{Remark}
\newtheorem*{ack}{Acknowledgments}

\numberwithin{equation}{section}

\title[Fractional Morrey's inequality]{On Morrey's inequality\\ in Sobolev-Slobodecki\u{\i} spaces}

\author[Brasco]{Lorenzo Brasco}
\address[L.\ Brasco]{Dipartimento di Matematica e Informatica
	\newline\indent
	Universit\`a degli Studi di Ferrara
	\newline\indent
	Via Machiavelli 35, 44121 Ferrara, Italy}
\email{lorenzo.brasco@unife.it}

\author[Prinari]{Francesca Prinari}

\address[F. Prinari]{Dipartimento di Scienze Agrarie, Alimentari e Agro-ambientali
\newline\indent 
Universit\`a di Pisa
\newline\indent
Via del Borghetto 80, 56124 Pisa, Italy}
\email{francesca.prinari@unipi.it}

\author[Sk]{Firoj Sk}
\address[F. Sk]{Analysis And Partial Differential Equations Unit
\newline\indent 
Okinawa Institute Of Science And Technology
\newline\indent 
1919-1 Tancha, Onna-Son, Okinawa 904-0495, Japan}
\email{firojmaciitk7@gmail.com}

\usepackage[colorlinks=true,urlcolor=blue, citecolor=red,linkcolor=blue,
linktocpage,pdfpagelabels, bookmarksnumbered,bookmarksopen]{hyperref}
\usepackage[hyperpageref]{backref}

\date{\today}


\begin{document}

\begin{abstract}
We study the sharp constant in the Morrey inequality for fractional Sobolev-Slobodecki\u{\i} spaces on the whole $\mathbb{R}^N$. By generalizing a recent work by Hynd and Seuffert, we prove existence of extremals, together with some regularity estimates. We also analyze the sharp asymptotic behaviour of this constant as we reach the borderline case $s\,p=N$, where the inequality fails. This can be done by means of a new elementary proof of the Morrey inequality, which combines: a local fractional Poincar\'e inequality for punctured balls, the definition of capacity of a point and Hardy's inequality for the punctured space. Finally, we compute the limit of the sharp Morrey constant for $s\nearrow 1$, as well as its limit for $p\nearrow \infty$. We obtain convergence of extremals, as well.
\end{abstract}

\maketitle

\begin{center}
\begin{minipage}{10cm}
\small
\tableofcontents
\end{minipage}
\end{center}

\section{Introduction}

\subsection{Foreword}
Let $N\in\mathbb{N}\setminus\{0\}$ be the dimension, we fix an exponent $1\le p<\infty$ and a parameter $0<s\le 1$. We define the following norm\footnote{At risk of being pedantic, we point out that this is actually a norm on compactly supported functions.} for every $\varphi\in C^\infty_0(\mathbb{R}^N)$
\[
[\varphi]_{W^{s,p}(\mathbb{R}^N)}=\left\{\begin{array}{ll}
\displaystyle \left(\iint_{\mathbb{R}^N\times \mathbb{R}^N} \frac{|\varphi(x)-\varphi(y)|^p}{|x-y|^{N+s\,p}}\,dx\,dy\right)^\frac{1}{p},& \mbox{ if } 0<s<1,\\
&\\
\displaystyle\left(\int_{\mathbb{R}^N} |\nabla \varphi|^p\,dx\right)^\frac{1}{p},& \mbox{ if } s=1. 
\end{array}
\right.
\]
For $0<\alpha\le 1$, we also set
\[
[\varphi]_{C^{0,\alpha}(\mathbb{R}^N)}=\sup_{x,y\in \mathbb{R}^N, x\not=y} \frac{|\varphi(x)-\varphi(y)|}{|x-y|^\alpha}.
\]
When $s$ and $p$ are such that $s\,p>N$, we set 
\begin{equation}
\label{exponent}
\alpha_{s,p}=s-\frac{N}{p}.
\end{equation}
We observe that this is the unique exponent $\alpha$ such that 
\[
\varphi\mapsto[\varphi]_{W^{s,p}(\mathbb{R}^N)} \qquad \mbox{ and }\qquad \varphi\mapsto[\varphi]_{C^{0,\alpha}(\mathbb{R}^N)},
\]
are dimensionally equivalent, i.e. they scale the same way. Thus, for every $\varphi\in C^\infty_0(\mathbb{R}^N)$ not identically vanishing, the ratio
\[
\varphi\mapsto \left(\frac{[\varphi]_{W^{s,p}(\mathbb{R}^N)}}{[\varphi]_{C^{0,\alpha_{s,p}}(\mathbb{R}^N)}}\right)^p,
\]
is scale invariant. In this paper, we are interested in the lowest possible value of this ratio, which can be written as
\begin{equation}
\label{morreysharp}
\mathfrak{m}_{s,p}(\mathbb{R}^N)=\inf_{\varphi\in C^\infty_0(\mathbb{R}^N)} \Big\{[\varphi]_{W^{s,p}(\mathbb{R}^N)}^p\, : \, [\varphi]_{C^{0,\alpha_{s,p}}(\mathbb{R}^N)}=1\Big\},
\end{equation}
thus generalizing to the fractional case $0<s<1$ some recent studies by Hynd and Seuffert (see the series of papers \cite{HS0, HS, HS2} and \cite{HS3}).
\vskip.2cm\noindent
By definition, the quantity $\mathfrak{m}_{s,p}(\mathbb{R}^N)$ is nothing but the sharp constant in the so-called {\it Morrey inequality}
\begin{equation}
\label{morrey_intro}
C_{N,s,p}\,[\varphi]^p_{C^{0,\alpha_{s,p}}(\mathbb{R}^N)}\le [\varphi]^p_{W^{s,p}(\mathbb{R}^N)},\qquad \mbox{ for every } \varphi\in C^\infty_0(\mathbb{R}^N).
\end{equation}
The local case $s=1$ is well-known and can be found for example in \cite[Theorem 3.8]{Gi}. For the fractional case $0<s<1$, we refer for example to \cite[Theorem 7.23]{Leoni}.
This inequality can be seen as a surrogate for $s\,p>N$ of the {\it Sobolev inequality} (see for example \cite[Theorem 10.2.1]{Maz} for the case $0<s<1$)
\[
T_{N,s,p}\,\left(\int_{\mathbb{R}^N} |\varphi|^\frac{N\,p}{N-s\,p}\,dx\right)^\frac{N-s\,p}{N}\le [\varphi]^p_{W^{s,p}(\mathbb{R}^N)},\qquad \mbox{ for every } \varphi\in C^\infty_0(\mathbb{R}^N),
\]
which is valid for $s\,p<N$. Observe that both inequalities are {\it homogeneous}, in the sense that they only contain norms having a scaling property. The main difference between the two inequalities is that Morrey's inequality does not provide ``size'' (i.e. integrability) estimates on functions, but only estimates on their oscillations. Actually, for $s\,p>N$ it is not possible to get size estimates only in terms of $[\,\cdot\,]_{W^{s,p}(\mathbb{R}^N)}$ for functions in $C^\infty_0(\mathbb{R}^N)$.
This is the ultimate reason why, for example, the completion of $C^\infty_0(\mathbb{R}^N)$ in the norm $[\,\cdot\,]_{W^{s,p}(\mathbb{R}^N)}$
{\it fails to be a functional space} for $s\,p>N$ (see \cite{BGCV, MPS} and the references therein on this point).

\subsection{Main results}

Our first result is an elementary proof of the fractional Morrey inequality (see Theorem \ref{thm:morrey} below). This is simply based on showing that
\[
\mathfrak{m}_{s,p}(\mathbb{R}^N)\gtrsim \Lambda_{s,p}:=\inf_{\varphi\in C^1(\overline{B_1(0)})} \Big\{[\varphi]_{W^{s,p}(B_1(0))}^p\, :\, \|\varphi\|_{L^p(B_1(0))}=1, \, \varphi(0)=0\Big\}.
\]
In other words, we show that it is possible to bound $\mathfrak{m}_{s,p}(\mathbb{R}^N)$ from below by means of the ``local'' fractional Poincar\'e constant, in a punctured disk. In turn, the latter is ``quantitatively'' positive: indeed, by an argument {\it \`a la} Maz'ya, it is (lenghty but) not difficult to prove that
\[
\Lambda_{s,p}\gtrsim \widetilde{\mathrm{cap}}_{s,p}\big(\{0\};B_1(0)\big),
\]
see Proposition \ref{prop:poincare_cap}. Here $\widetilde{\mathrm{cap}}_{s,p}(\{0\};B_1(0))$ is the {\it fractional $p-$capacity of order $s$ of $\{0\}$ relative to $B_1(0)$}, see Definition \ref{defi:capacity}. This quantity is positive precisely for the range $s\,p>N$.
\par
In turn, by relying on a symmetrization argument, we show that 
\[
\widetilde{\mathrm{cap}}_{s,p}\big(\{0\};B_1(0)\big)\gtrsim \mathfrak{h}_{s,p}(\mathbb{R}^N\setminus\{0\}):=\inf_{\varphi\in C^\infty_0(\mathbb{R}^N\setminus\{0\})}\left\{[\varphi]^p_{W^{s,p}(\mathbb{R}^N)}\, :\, \int_{\mathbb{R}^N} \frac{|\varphi|^p}{|x|^{s\,p}}\,dx=1\right\}.
\]
Thus, by joining the last three estimates, we get a lower bound of the type
\begin{equation}
\label{morrey-hardy}
\mathfrak{m}_{s,p}(\mathbb{R}^N)\gtrsim \mathfrak{h}_{s,p}(\mathbb{R}^N\setminus\{0\}),
\end{equation}
in terms of the Hardy constant $\mathfrak{h}_{s,p}(\mathbb{R}^N\setminus\{0\})$ for the punctured space (see Corollary \ref{coro:mistery}). This can be seen as the analogue for the case $s\,p>N$ of the well-known lower bound on the sharp Sobolev constant, in terms on the sharp Hardy constant for the whole space, see for example \cite[proof of Theorem 1]{MS} and \cite[Section 4]{FS}.
\vskip.2cm\noindent
Other results obtained in this paper concern the existence and regularity of minimizers for the minimization problem \eqref{morreysharp}. Indeed, by relying on the same argument as in \cite{HS}, we will show that the infimum in \eqref{morreysharp} is attained on the larger homogeneous space $\mathcal{W}^{s,p}(\mathbb{R}^N)$ (see \eqref{goodspace} below for its definition). Any such a minimizer, called {\it extremal} in the sequel, will be shown to verify weakly the equation
\begin{equation}
\label{ELintro}
(-\Delta_p)^s u=C\,(\delta_{x_0}-\delta_{y_0}),\qquad \mbox{ in }\mathbb{R}^N,
\end{equation}
for a suitable $C\not=0$ and $x_0,y_0$ distinct points. Here $(-\Delta_p)^s$ is the {\it fractional $p-$Laplacian of order $s$}, defined in its weak form by 
\[
\langle(-\Delta_p)^s u,\varphi\rangle:=\iint_{\mathbb{R}^N\times \mathbb{R}^N} \frac{J_p(u(x)-u(y))\,(\varphi(x)-\varphi(y))}{|x-y|^{N+s\,p}}\,dx\,dy, \qquad \mbox{ for every } \varphi\in C^\infty_0(\mathbb{R}^N),
\]
and $J_p(t)=|t|^{p-2}\,t$, for every $t\in\mathbb{R}$. By relying on some recent a priori estimates contained in \cite{BLS} and \cite{GL}, we will prove that any extremal $u$ enjoys the following additional regularity
\[
u\in C^{0,\delta}_{\rm loc}(\mathbb{R}^N)\cap L^\infty(\mathbb{R}^N),\qquad \mbox{ for every } \alpha_{s,p}\le \delta<\frac{s\,p-N}{p-1},
\]
and
\[
u\in C^{0,\delta}_{\rm loc}(\mathbb{R}^N\setminus\{x_0,y_0\}),\qquad \mbox{ for every } \alpha_{s,p}\le \delta<\min\left\{\frac{s\,p}{p-1},1\right\},
\]
see Theorems \ref{thm:bound} and \ref{thm:holderreg}. In \cite{HS}, for the local case $s=1$, the global boundedness of extremals is proved by exploiting their geometric properties. Here on the contrary, we simply rely on the optimality condition \eqref{ELintro}, in conjunction with a {\it Moser--type iterative argument} based on Hardy's inequality. The result is then readily obtained, actually for more general right-hand sides (see Proposition \ref{prop:limita}).
\vskip.2cm\noindent
The explicit value of the sharp constant $\mathfrak{m}_{s,p}(\mathbb{R}^N)$ is still unknown and its determination seems quite a challenging problem, already in the local case $s=1$. Thus, it is interesting to get at least some upper and lower bounds, which are ``quantitatively'' sharp, in a suitable sense. At this aim, we devote a part of our paper to study some meaningful asymptotics for $\mathfrak{m}_{s,p}(\mathbb{R}^N)$. In particular, we will analyze its asymptotic behaviour as:
\begin{itemize}
\item[(i)] $1\le N<p$ are fixed and $s\searrow N/p$;
\vskip.2cm
\item[(ii)] $0<s<1\le N$ are fixed and $p\nearrow \infty$;
\vskip.2cm
\item[(iii)] $1\le N<p$ are fixed and $s\nearrow 1$.
\end{itemize}
As we will explain in the next subsection, the item (i) is the most difficult and original one. Here, we will crucially exploit our proof of Morrey's inequality, that we exposed above.
Indeed, even if our lower bound on $\mathfrak{m}_{s,p}(\mathbb{R}^N)$ is very likely not optimal, it permits to get quite easily the sharp decay rate to $0$ of $\mathfrak{m}_{s,p}(\mathbb{R}^N)$, as $p>N$ is fixed and $s\searrow N/p$. More precisely, by using the estimate \eqref{morrey-hardy}
and the explicit expression of the Hardy constant $\mathfrak{h}_{s,p}(\mathbb{R}^N\setminus\{0\})$, we can prove that
\[
0<\liminf_{s\searrow \frac{N}{p}} \frac{\mathfrak{m}_{s,p}(\mathbb{R}^N)}{(s\,p-N)^{p-1}}.
\]
This has to be coupled with the information
\[
\limsup_{s\searrow \frac{N}{p}} \frac{\mathfrak{m}_{s,p}(\mathbb{R}^N)}{(s\,p-N)^{p-1}}<+\infty,
\]
which is here obtained by means of a careful choice of a test function. Both results are contained in Proposition \ref{prop:spN} below. 
\par
The lower bound provided our proof is good enough to handle item (ii) above, as well. 
We show in Proposition \ref{prop:limitinfty} that
\[
\lim_{p\to\infty} \Big(\mathfrak{m}_{s,p}(\mathbb{R}^N)\Big)^\frac{1}{p}=1,
\]
thus generalizing to the case $0<s<1$ a result recently obtained by the first two authors and Zagati in \cite{BPZ}. However, here we use a slightly different proof, with respect to that of \cite{BPZ}.
Finally, for the case (iii), we show in Theorem \ref{thm:limits1} that
\[
\lim_{s\nearrow 1}(1-s)\, \mathfrak{m}_{s,p}(\mathbb{R}^N)= K_{p,N}\, \mathfrak{m}_{1,p}(\mathbb{R}^N),
\]
where $K_{p,N}$ is an explicit constant. The renormalization factor $(1-s)$ is not surprising, in light of the celebrated {\it Bourgain-Brezis-Mironescu formula} (see \cite[Corollary 3.20]{EE}), which reads as follows
\[
\lim_{s\nearrow 1} (1-s)\,[\varphi]^p_{W^{s,p}(\mathbb{R}^N)}=K_{N,p}\,[\varphi]^p_{W^{1,p}(\mathbb{R}^N)},\qquad \mbox{ for every } \varphi\in C^\infty_0(\mathbb{R}^N).
\] 
Our result comes with a convergence result of fractional extremals to those of the local case $s=1$, in the sense of local uniform convergence. It can be regarded as a $\Gamma-$convergence result, we refer to \cite{Dal} for the general theory and to \cite{ADM, BPS, DR, ponce} for some results of this type in the context of fractional Sobolev spaces. 
\subsection{A note on the borderline case $s\,p=N$}
We want to expand a bit the discussion on the asymptotic behaviour of $\mathfrak{m}_{s,p}(\mathbb{R}^N)$, as $s\,p\searrow N$.
\par
Let us start with the local case $s=1$. Some of the classical proofs that can be found in the textbooks provide the following lower bound
\[
\mathfrak{m}_{1,p}(\mathbb{R}^N)\ge C_{p,N}\sim (p-N)^p,\qquad \mbox{ as } p\searrow N,
\]
see for example \cite[Theorem 9.12]{Bre} and \cite[Theorem 3.8]{Gi}. However, this behaviour is not sharp. By resorting to some finer estimates relying on the following {\it local potential estimate}
\[
\int_{B_r(x)} |\varphi(y)-\varphi(z)|\,dy\le C_N\,\int_{B_r(x)} \frac{|\nabla \varphi(y)|}{|x-y|^{N-1}}\,dy,\qquad \mbox{ for } x\in\mathbb{R}^N, z\in B_r(x),
\] 
one can get
\[
\mathfrak{m}_{1,p}(\mathbb{R}^N)\ge \widetilde C_{p,N}\sim (p-N)^{p-1},\qquad \mbox{ as } p\searrow N.
\]
See among others the proofs of \cite[Theorem 3]{EG} and \cite[Theorem 2.4.4]{Zi}.
Asymptotic optimality of this lower bound can be easily proved by using a truncation of the {\it fundamental solution for the $p-$Laplacian}, i.e.
\[
\varphi(x)=\Big(1-|x|^\frac{p-N}{p-1}\Big)_+,\qquad \mbox{ for } x\in\mathbb{R}^N.
\]
The choice of this test function can be easily guessed, as we now explain. Indeed, it has been shown in \cite{HS} that each extremal $u$ for $\mathfrak{m}_{1,p}(\mathbb{R}^N)$ solves the following local version of \eqref{ELintro}
\[
-\Delta_p u=C\,(\delta_{x_0}-\delta_{y_0}),
\]
for a suitable pair of distinct points $x_0,y_0\in\mathbb{R}^N$.  Moreover, $u$ is smooth in the doubly punctured space $\mathbb{R}^N\setminus\{x_0,y_0\}$, while around these two points we have
\[
u(x)-u(x_0)\sim |x-x_0|^\frac{p-N}{p-1} \mbox{ as } x\to x_0\qquad \mbox{ and }\qquad u(x)-u(y_0)\sim |x-y_0|^\frac{p-N}{p-1} \mbox{ as }x\to y_0,
\]
see \cite[Section 4]{HS}.
\begin{remark}[The one-dimensional case, $s=1$]
This is quite peculiar, in this case extremals are known, together with the value of the sharp constant (see \cite[Section 2.3]{HS}). We have 
\[
\mathfrak{m}_{1,p}(\mathbb{R})=1,
\]
and observe that this does not vanish, as $p$ goes to $1$. This is consistent with the fact that for $p=1$ the homogeneous Sobolev space still embeds continuously into a space of continuous function and we have
\[
2\,\|\varphi\|_{L^\infty(\mathbb{R})}\le [\varphi]_{W^{1,1}(\mathbb{R})},\qquad \mbox{ for every } \varphi\in C^\infty_0(\mathbb{R}),
\]
see \cite[Theorem 5.5]{BGCV}.
\end{remark}
Let us now come to the fractional case $0<s<1$, which is our primary target. A proof of \eqref{morrey_intro} in its {\it nonhomogeneous} version, i.e.
\[
C_{N,s,p}\,\left(\|\varphi\|_{L^\infty(\mathbb{R}^N)}+[\varphi]_{C^{0,\alpha_{s,p}}(\mathbb{R}^N)}\right)^p\le \Big(\|\varphi\|_{L^p(\mathbb{R}^N)}+[\varphi]_{W^{s,p}(\mathbb{R}^N)}\Big)^p,\quad \mbox{ for every } \varphi\in C^\infty_0(\mathbb{R}^N),
\] 
can be found for example in \cite[Theorem 2.7.1, case (i)]{Tri}. By a simple scaling argument, we see that if we call $\mathfrak{M}_{s,p}(\mathbb{R}^N)$ the sharp constant for this inequality, then we have
\[
\mathfrak{M}_{s,p}(\mathbb{R}^N)\le \mathfrak{m}_{s,p}(\mathbb{R}^N).
\]
Thus, lower bounds on the nonhomogeneous constant $\mathfrak{M}_{s,p}(\mathbb{R}^N)$ would provide lower bounds on our constant, as well. Unfortunataley, the proof in \cite{Tri} is done in the wider context of Triebel-Lizorkin spaces, defined in terms of Fourier analysis and Paley--Littlewood--type decompositions. In particular, the result in \cite{Tri} concerns some {\it equivalent norms} and not directly the ones considered here.
Tracing back a possible explicit lower bound on $\mathfrak{M}_{s,p}(\mathbb{R}^N)$ in this proof seems quite a gravesome task, nearly prohibitive.
\par
An older proof due to Peetre directly deals with the homogeneous inequality \eqref{morrey_intro}, see \cite[Th\'eor\`eme 8.2]{Pe}. However, here as well, it seems quite complicate to extrapolate a quantitative lower bound on $\mathfrak{m}_{s,p}(\mathbb{R}^N)$. Indeed, the use of abstract real interpolation techniques certainly simplifies the task of getting \eqref{morrey_intro} for equivalent norms, but it complicates a lot the task of producing some effective lower bounds. 
\par
It is possible to give a handier proof by relying on {\it Campanato spaces}  and the {\it Campanato-Meyers Theorem}, see for example \cite[Chapter 2, Section 3]{Gi} for these classical tools. 
This proof can be found in \cite[Theorem 8.2]{DNPV}, for the case of the non-homogeneous Morrey inequality in open bounded sets having some smoothness condition. The homogeneous case in the whole space can be found in \cite[Corollary 2.7]{BGCV}, where \eqref{morrey_intro} is obtained with a constant incorporating the factor $(1-s)$. 
\par
In this proof, the first step is to get the fractional Poincar\'e-Wirtinger inequality
\begin{equation}
\label{PWs}
\int_{B_R(x_0)} \left|\varphi-\fint_{B_R(x_0)} \varphi\,dy\right|^p\,dx\le (1-s)\,\mu_{N,p}\,R^{s\,p}\,[\varphi]_{W^{s,p}(\mathbb{R}^N)}^p,\qquad \mbox{ for every ball } B_R(x_0).
\end{equation}
This entails an estimate in the homogeneous Campanato space $\mathscr{L}^{p,\lambda}(\mathbb{R}^N)$, with $\lambda=s\,p$. Here
\[
\mathscr{L}^{p,\lambda}(\mathbb{R}^N)=\Big\{\varphi\in L^p_{\rm loc}(\mathbb{R}^N)\, :\, [u]_{\mathscr{L}^{p,\lambda}(\mathbb{R}^n)}<+\infty\Big\},
\]
and
\[
[\varphi]_{\mathscr{L}^{p,\lambda}(\mathbb{R}^N)}:=\left(\sup_{x_0\in\mathbb{R}^N,\,R>0} \frac{1}{R^\lambda}\,\int_{B_R(x_0)} \left|\varphi-\fint_{B_R(x_0)} \varphi\,dy\right|^p\,dx\right)^\frac{1}{p}.
\]
Then we use that if $\lambda=s\,p$ is larger than the dimension $N$, this space embeds into a homogeneous space of H\"older continuous functions, with exponent $\alpha_{s,p}$ given by \eqref{exponent}. More precisely, we have
\begin{equation}
\label{campanatoholder}
L_{N,s,p}\,[u]_{C^{0,\alpha_{s,p}}(\mathbb{R}^N)}\le [u]_{\mathscr{L}^{p,s\,p}(\mathbb{R}^N)}.
\end{equation}
This is the Campanato-Meyers Theorem (see \cite[Teorema 5.I]{Cam} and \cite[Theorem]{Mey}). By joining \eqref{campanatoholder} and \eqref{PWs}, we thus get \eqref{morrey_intro}, with the lower bound
\[
\mathfrak{m}_{s,p}(\mathbb{R}^N)\ge \frac{(L_{N,s,p})^p}{(1-s)\,\mu_{N,p}}.
\]
However, there is a vicious detail here: the presence of the Campanato-Meyers constant $L_{N,s,p}$. When $\lambda=s\,p=N$ the Campanato space $\mathscr{L}^{p,N}(\mathbb{R}^N)$ is isomorphic to a space of BMO functions, which are neither continuous nor bounded in general.
Thus, we must expect that the constant $L_{N,s,p}$ in \eqref{campanatoholder}, and accordingly the lower bound on $\mathfrak{m}_{s,p}(\mathbb{R}^N)$, deteriorates as $s\,p$ goes to $N$. A careful inspection of Campanato's proof of \eqref{campanatoholder} (see the proof of the key estimate \cite[equation (2.21), Lemma 2.2]{Gi} or directly \cite[Proposition 1]{KM}), reveals that
\[
L_{N,s,p}\sim s\,p-N,\qquad \mbox{ as } s\,p\searrow N.
\] 
Thus, by this method we get a lower bound which, apart for not being optimal, {\it does not even display the sharp decay rate} as we approach the borderline case $s\,p=N$. 
\par
As explained in the previous subsection, our method of proof of Theorem \ref{thm:morrey} permits to circumvent this problem and rectifies this non-sharp behaviour.
\begin{remark}[The one-dimensional case, $0<s<1$]
Even in this case $\mathfrak{m}_{s,p}(\mathbb{R})$ is not known.
Obtaining the sharp decay rate of $\mathfrak{m}_{s,p}(\mathbb{R}^N)$ would be easier in the case $p=2$ and $N=1$, for $s\searrow 1/2$.
In this case, we can exploit the following characterization 
\[
[\varphi]_{W^{s,2}(\mathbb{R})}=\frac{1}{2\,\pi\,C_s}\,\int_{\mathbb{R}} |\xi|^{2\,s}\,\Big|\mathcal{F}[\varphi](\xi)\Big|^2\,d\xi,\qquad \mbox{ for every } \varphi\in C^\infty_0(\mathbb{R}),
\] 
in terms of the Fourier transform $\mathcal{F}$,
see for example \cite[Chapter VII, Section 9]{Ho}.
The sharp decay rate in this particular situation has been obtained in \cite[Theorem 3.6]{BB} by Bianchi and the first author.
\par
Still for $N=1$, but this time for the full range $s\,p>1$, an explicit estimate on the Morrey constant can be found in a paper by Simon, see \cite[Corollary 26]{Si}. However, here as well, the resulting lower bound has a suboptimal asymptotic behaviour in the regime $s\,p\searrow 1$ (the same as in the proof using the Campanato-Meyers Theorem). 
\end{remark}

\subsection{Some open issues}

Left aside for the explicit determination of $\mathfrak{m}_{s,p}(\mathbb{R}^N)$ and its extremals (an unsolved problem already for the case $s=1$), we want to list some
questions left open. We believe that some of them are quite interesting and deserve a future investigation. 
\par
At first, it would be interesting to improve our Proposition \ref{prop:spN} and determine $C_{N,p}>0$ such that
\[
\lim_{s\searrow \frac{N}{p}}\frac{\mathfrak{m}_{s,p}(\mathbb{R}^N)}{(s\,p-N)^{p-1}}=C_{N,p}.
\]
Still on the same problem, it seems intriguing to know what happens to $\mathfrak{m}_{s,p}(\mathbb{R}^N)$ as $s\,p\searrow N$, {\it without fixing} $p>N$. In other words, we allow both $s$ and $p$ to vary. The most interesting case is when
\[
s\searrow 0\qquad \mbox{ and }\qquad p\nearrow \infty,
\]
but their product is always larger than $N$, eventually converging to $N$. In this case some arcane phenomena could appear.
\par
It would be important to deepen the study of the regularity properties of extremals for $\mathfrak{m}_{s,p}(\mathbb{R}^N)$: for example, we expect extremals to be in $C^{0,\beta_{s,p}}_{\rm loc}(\mathbb{R}^N)$, with
\[
\beta_{s,p}=\frac{s\,p-N}{p-1}.
\]
Observe that this is the natural expected regularity, in light of the function
\[
x\mapsto |x|^\frac{s\,p-N}{p-1},
\]
which is supposed to be the fundamental solution\footnote{We try to shed some light on this point: it is a well-known fact that for $s\,p\not=N$ we have
\[
(-\Delta_p)^s|x|^\frac{s\,p-N}{p-1}=0,\qquad \mbox{ in } \mathbb{R}^N\setminus\{0\},
\]
in weak sense, see for example \cite[Theorem A.4]{BMS}, \cite[Theorem 1.1]{DQ} or \cite[Proposition 6.6]{Ta}. However, to the best of our knowledge, a proof of the fact that
\[
(-\Delta_p)^s|x|^\frac{s\,p-N}{p-1}=C_{N,p,s}\,\delta_0,\qquad \mbox{ in } \mathbb{R}^N,
\]
is still missing, except for the Hilbertian case $p=2$ (see \cite[Theorem 2.3]{Bu} and \cite[Theorem 8.4]{Ga}).} for $(-\Delta_p)^s$.
This would require a boosting of the regularity techniques presently available for fractional elliptic equations. Without any attempt of completeness, we mention \cite{Co, DKP, DKP2, IMS, KMS} and \cite{KMS2} for some important regularity results for equations driven by the fractional $p-$Laplacian.
Also, it would be interesting to prove regularity estimates for extremals which are {\it robust}, i.e. stable for $s\nearrow 1$ (in the vein of \cite[Theorem 1.5]{Brolin}). This could permit to improve our convergence result in Theorem \ref{thm:limits1} and to retrieve some informations on extremals for $s\sim 1$, from those of the local case.
\par
Finally, it would be interesting to prove some {\it global} regularity result for extremals. In this respect, we point out the recent concurrent paper \cite{Tav} by Tavakoli, studying the behaviour at infinity of extremals.

\subsection{Plan of the paper}

Section \ref{sec:2} contains the main notation used throughout the whole paper, together with a small ``bestiary'' listing all the relevant functional spaces needed in the sequel. We also collect some technical facts on fractional Sobolev and H\"older seminorms, which are useful for our needing. Section \ref{sec:3} still has a preparatory character: here we discuss, from a quantitative point of view, the fractional capacity of a point.
In Section \ref{sec:4} we introduce the constant $\Lambda_{s,p}$ and prove the cornerstone inequality (see Lemma \ref{lm:1stjuly}), which eventually leads to the proof of Morrey's inequality on the whole space. Section \ref{sec:5} is devoted to generalize to the fractional case the results of \cite[Section 6]{HS}. Here, the delicate point is to get the ``localized'' fractional Morrey inequality of Lemma \ref{lm:pratignano}. This is done by means of a different argument with respect to the case $s=1$, still based on the constant $\Lambda_{s,p}$. 
We then obtain in Section \ref{sec:6} existence of extremals for $\mathfrak{m}_{s,p}(\mathbb{R}^N)$. The subsequent Section \ref{sec:7} explores some of their regularity properties.
At last, Section \ref{sec:8} contains the results about the asymptotic behaviours of $\mathfrak{m}_{s,p}(\mathbb{R}^N)$. For ease of readability, in turn this is divided in three subsections, according to the relevant asymptotics considered. 
\par
The paper ends with two appendices: Appendix \ref{app:A} contains some (lengthy) computations for the Sobolev-Slobodecki\u{\i} seminorm of a suitable test function. These are expedient to prove the optimality of the asymptotic behaviour for $s\searrow N/p$. Finally, Appendix \ref{app:B} contains the characterization of a homogeneous Sobolev-Slobodecki\u{\i} space built on the multiply punctured space. 

\begin{ack}
We thank Erik Lindgren and Alireza Tavakoli for sharing a draft version of their papers \cite{GL} and \cite{Tav}, respectively. We wish to thank Francesca Bianchi and Giuseppe Mingione for some discussions on the Campanato-Meyers Theorem.
Part of this work has been done during a visit of L.\,B. to the University of Pisa in May 2023. We gratefully acknowledge the Department of Mathematics and its facilities. F.\,P. is a member of the {\it Gruppo Nazionale per l'Analisi Matematica, la Probabilit\`a
e le loro Applicazioni} (GNAMPA) of the Istituto Nazionale di Alta Matematica (INdAM).
\end{ack}

\section{Preliminaries}

\label{sec:2}

\subsection{Notation}

For $x_0\in\mathbb{R}^N$ and $R>0$, we indicate by $B_R(x_0)$ the $N-$dimensional open ball centered at $x_0$, with radius $R$. We will use the standard notation $\omega_N$ for the $N-$dimensional Lebesgue measure of $B_1(0)$. Occasionally, we will use the symbol $\mathcal{H}^k$ for the $k-$dimensional Hausdorff measure, with $k\in\mathbb{N}$.
\par
For every bounded measurable set $E\subseteq\mathbb{R}^N$ with positive measure, we set
\[
\mathrm{av}(\varphi;E)=\fint_E \varphi\,dx, \qquad \mbox{ for } \varphi\in L^1_{\rm loc}(\mathbb{R}^N).
\]
Fo a pair of open sets $E\subseteq \Omega\subseteq\mathbb{R}^N$, the symbol $E\Subset \Omega$ means that the closure $\overline{E}$ is a compact subset of $\Omega$. 

\subsection{Bestiary of functional spaces}

Let $\Omega\subseteq\mathbb{R}^N$ be an open set. We indicate by $C(\overline\Omega)$ the vector space of real valued continuous functions on $\overline\Omega$. In what follows, for every $0<\alpha\le 1$ and every open set $\Omega\subseteq\mathbb{R}^N$, we will need the {\it homogeneous H\"older space} defined by
\[
\dot C^{0,\alpha}(\overline\Omega)=\Big\{\varphi\in C(\overline\Omega)\, :\, [\varphi]_{C^{0,\alpha}(\overline\Omega)}<+\infty\Big\}.
\]
Observe that the elements of this space are {\it not} bounded functions, in general. We will reserve the more familiar symbol $C^{0,\alpha}(\overline\Omega)$ for the space
\[
C^{0,\alpha}(\overline\Omega)=\dot C^{0,\alpha}(\overline\Omega)\cap C_{\rm bound}(\overline\Omega),
\]
where $C_{\rm bound}(\overline\Omega)$ is the Banach space of continuous and bounded functions on $\overline\Omega$. 
The space $C^{0,\alpha}(\overline\Omega)$ will be endowed with the usual norm
\[
\|\varphi\|_{C^{0,\alpha}(\overline\Omega)}=\|\varphi\|_{L^\infty(\Omega)}+[\varphi]_{C^{0,\alpha}(\overline\Omega)},\qquad \mbox{ for every } \varphi\in C^{0,\alpha}(\overline\Omega).
\]
We observe that when $\Omega$ is bounded, the two spaces $\dot C^{0,\alpha}(\overline\Omega)$ and $C^{0,\alpha}(\overline\Omega)$ coincide as sets. In some places, we will also need the following symbol
\[
C_{\rm loc}^{0,\alpha}(\Omega)=\Big\{u:\Omega\to\mathbb{R} \mbox{ continuous }\, :\, [u]_{C^{0,\alpha}(\overline{B_r(x_0)})}<+\infty \mbox{ for every } B_r(x_0)\Subset \Omega \Big\}.
\]
Whenever $\Omega\not=\mathbb{R}^N$, we will also indicate by $C_0(\Omega)$ the completion of $C^\infty_0(\Omega)$ with respect to the sup norm. We recall that we have
\[
C_0(\Omega)\subseteq \{u\in C_{\rm bound}(\overline\Omega)\, :\, u=0 \mbox{ on }\partial\Omega\},
\]
see \cite[Lemma 2.1]{BPZ}.
For $1\le p<\infty$ and $0<s\le 1$, we define
\[
W^{s,p}(\Omega)=\Big\{\varphi\in L^p(\Omega)\, :\, [\varphi]_{W^{s,p}(\Omega)}<+\infty\Big\},
\]
endowed with the norm
\[
\|\varphi\|_{W^{s,p}(\Omega)}=\|\varphi\|_{L^p(\Omega)}+[\varphi]_{W^{s,p}(\Omega)},\qquad \mbox{ for every } \varphi\in W^{s,p}(\Omega).
\]
We also indicate by $\widetilde{W}^{s,p}_0(\Omega)$ the closure of $C^\infty_0(\Omega)$ in $W^{s,p}(\mathbb{R}^N)$. We will denote by $\mathscr{D}^{s,p}_0(\Omega)$ the completion of $C^\infty_0(\Omega)$ with respect to the norm
\[
\varphi\mapsto [\varphi]_{W^{s,p}(\mathbb{R}^N)}.
\]
The following space
\[
\mathcal{D}^{s,p}(\mathbb{R}^N)=\Big\{\varphi\in L^1_{\rm loc}(\mathbb{R}^N)\, :\, [\varphi]_{W^{s,p}(\mathbb{R}^N)}<+\infty\Big\},
\]
will be useful, as well. At last, one more fractional Sobolev space.
For every $N\ge 1$, $0<s\le 1$ and $1<p<\infty$ such that $s\,p>N$, we will frequently work with the following space
\begin{equation}
\label{goodspace}
\mathcal{W}^{s,p}(\mathbb{R}^N)=\Big\{\varphi\in \dot C^{0,\alpha_{s,p}}(\mathbb{R}^N)\, :\, [\varphi]_{W^{s,p}(\mathbb{R}^N)}<+\infty\Big\},
\end{equation}
where the exponent $\alpha_{s,p}$ is defined in \eqref{exponent}.
\begin{remark}[A family tree]
We will show that
\[
\mathcal{D}^{s,p}(\mathbb{R}^N)=\mathcal{W}^{s,p}(\mathbb{R}^N),\qquad \mbox{ for } s\,p>N,
\]
see Proposition \ref{prop:uguaspazi} below. In Proposition \ref{prop:homo}, we will also prove that
\[
\mathscr{D}^{s,p}_0(\mathbb{R}^N\setminus\{x_0\})=\Big\{u\in \mathcal{W}^{s,p}(\mathbb{R}^N)\, :\, u(x_0)=0\Big\},\qquad \mbox{ for }s\,p>N.
\]
We also recall that, still for $s\,p>N$, we have that the quotient space
\[
\frac{\mathcal{W}^{s,p}(\mathbb{R}^N)}{\sim_C}=:\dot W^{s,p}(\mathbb{R}^N),
\]
is isomorphic to the completion space $\mathscr{D}^{s,p}_0(\mathbb{R}^N)$ (see \cite[Theorem 4.4]{BGCV}). Here $\sim_C$ is the equivalence relation given by 
\[
u\sim_C v \qquad \Longleftrightarrow\qquad u-v \mbox{ is constant}.
\]
Finally, we recall that in general we have
\[
\widetilde{W}^{s,p}_0(\Omega)\not= \mathscr{D}^{s,p}_0(\Omega).
\]
\end{remark}

\subsection{Some technical facts}
The next technical result is a sort of H\"older inequality for fractional seminorms. Observe that the estimate below blows-up as $t\nearrow s$ and $q<p$ are fixed. This is in accordance with the well-known fact that $W^{s,q}(\Omega)$ {\it is not} embedded in $W^{s,p}(\Omega)$, in general (see \cite{MiSi}). The proof can be found in \cite[Lemma 4.6]{Co}.
\begin{lemma}
\label{lemma:embeddinglow}
Let $0<t<s<1$ and $1\le q<p<\infty$. Let $\Omega\subseteq\mathbb{R}^N$ be an open bounded set. Then for every $u\in W^{s,p}(\Omega)$ we have 
\[
[u]_{W^{t,q}(\Omega)}\le \left(\frac{N\,\omega_N\,|\Omega|}{s-t}\,\frac{p-q}{p\,q}\right)^\frac{p-q}{p\,q}\,\Big(\mathrm{diam}(\Omega)\Big)^{s-t}\,[u]_{W^{s,p}(\Omega)}.
\]
\end{lemma}
\begin{lemma}
\label{lemma:2}
Let $0<t<1$ and let $\Omega\subseteq\mathbb{R}^N$ be an open set. Let $u\in W^{t,q}(\Omega)$ for every $1\le q<\infty$, such that
\[
L:=\liminf_{q\to \infty} [u]_{W^{t,q}(\Omega)}<+\infty.
\]
Then $u\in \dot C^{0,t}(\overline\Omega)$ and $[u]_{C^{0,t}(\overline\Omega)}\le L$, in the sense that there exists $\widetilde{u}\in \dot C^{0,t}(\overline\Omega)$ such that $u=\widetilde{u}$ almost everywhere.
\end{lemma}
\begin{proof}
Let us fix $0<\varepsilon\ll 1$ and set 
\[
\mathcal{O}_\varepsilon:=\Big\{(x,y)\in\Omega\times\Omega\, :\, |x-y|\ge \varepsilon\Big\}\cap \left(B_\frac{1}{\varepsilon}(0)\times B_\frac{1}{\varepsilon}(0)\right).
\]
We then have 
\[
\begin{split}
L=\liminf_{q\to \infty} [u]_{W^{t,q}(\Omega)}&\ge \liminf_{q\to\infty} \left(\iint_{\mathcal{O}_\varepsilon}\frac{|u(x)-u(y)|^q}{|x-y|^{t\,q}}\,\frac{dx\,dy}{|x-y|^{N}}\right)^\frac{1}{q}\\
&=\lim_{q\to\infty} \left\|\frac{u(x)-u(y)}{|x-y|^t}\right\|_{L^q(\mathcal{O}_\varepsilon;|x-y|^{-N})}=\sup_{(x,y)\in \mathcal{O}_\varepsilon} \frac{|u(x)-u(y)|}{|x-y|^t}.
\end{split}
\]
By arbitrariness of $\varepsilon>0$, we get the conclusion.
\end{proof}
\begin{lemma} 
Let $\Omega\subseteq \mathbb{R}^N$ be an open set. Then, 
if $0<\beta<\beta'\leq 1$  it holds
\begin{equation}
\label{normeholder2}
\lim_{\alpha \searrow \beta} [u]_{C^{0,\alpha}(\overline\Omega)}=[u]_{C^{0,\beta}(\overline\Omega)}, \qquad \hbox{ for  every }  u\in \dot C^{0,\beta'}(\overline\Omega)\cap \dot C^{0,\beta}(\overline\Omega). 
\end{equation}
Moreover, for every $u\in \dot C^{0,\beta}(\overline\Omega)$ such that 
\[
[u]_{W^{\beta,q}(\overline\Omega)}<+\infty,\qquad \mbox{ for some } 1\le q<\infty,
\]
we have that
\begin{equation}
\label{normegagliardo1}
\lim_{p\to \infty }[u]_{W^{\beta,p}(\Omega)} =[u]_{C^{0,\beta}(\overline\Omega)}.
\end{equation}
\end{lemma}
\begin{proof} 
In order to show \eqref{normeholder2}, we first observe that for every pair $x,y\in\overline\Omega$ of distinct points, we have
\[
\frac{|u(x)-u(y)|}{|x-y|^\beta}=\lim_{\alpha\to \beta }\frac{|u(x)-u(y)|}{|x-y|^\alpha}\le \liminf_{\alpha\to \beta} [u]_{C^{0,\alpha}(\overline\Omega)}.
\] 
This in turn gives
\begin{equation}
\label{liminf_holder}
[u]_{C^{0,\beta}(\overline\Omega)}\le \liminf_{\alpha\to \beta} [u]_{C^{0,\alpha}(\overline\Omega)}.
\end{equation}
This holds in both cases $\alpha\nearrow \beta$ and $\alpha\searrow \beta$.
For the $\limsup$ inequality, we can use a simple interpolation estimate: observe that for every $\beta<\alpha<\beta'$ and every pair $x,y\in\overline\Omega$ of distinct points, we have
\[
\frac{|u(x)-u(y)|}{|x-y|^\alpha}\le \left(\frac{|u(x)-u(y)|}{|x-y|^\beta}\right)^\frac{\beta'-\alpha}{\beta'-\beta}\,\left(\frac{|u(x)-u(y)|}{|x-y|^{\beta'}}\right)^\frac{\alpha-\beta}{\beta'-\beta}.
\]
This yields 
\[
[u]_{C^{0,\alpha}(\overline\Omega)}\le \left([u]_{C^{0,\beta}(\overline\Omega)}\right)^\frac{\beta'-\alpha}{\beta'-\beta}\,\left([u]_{C^{0,\beta'}(\overline\Omega)}\right)^\frac{\alpha-\beta}{\beta'-\beta},
\]
from which the desired $\limsup$ inequality easily follows.
\vskip.2cm\noindent
In order to show \eqref{normegagliardo1}, let $u\in \dot C^{0,\beta}(\overline{\Omega})$ be such that $[u]_{W^{\beta,q}(\Omega)}<+\infty$ for some finite $q\ge 1$. We can resort to an interpolation inequality, here as well. For every $p> q$, we have
\[
\begin{split}
[u]_{W^{\beta,p}(\Omega)}^p=\iint_{\Omega\times\Omega} \frac{|u(x)-u(y)|^p}{|x-y|^{N+\beta\,p}}\,dx\,dy&=\iint_{\Omega\times\Omega} \left(\frac{|u(x)-u(y)|}{|x-y|^\beta}\right)^{p-q}\,\frac{|u(x)-u(y)|^q}{|x-y|^{N+\beta\,q}}\,dx\,dy\\
&\le \left([u]_{C^{0,\beta}(\overline\Omega)}\right)^{p-q} \, [u]_{W^{\beta,q}(\Omega)}^{q}.
\end{split}
\]
By raising to the power $1/p$ and taking the limit as $p$ goes to $\infty$, we get
\[
\limsup_{p\to\infty} [u]_{W^{\beta,p}(\Omega)}\le [u]_{C^{0,\beta}(\overline\Omega)}.
\]
The $\liminf$ inequality follows from Lemma \ref{lemma:2}.
\end{proof}
\begin{remark}
\label{below}
If we assume $u\in C^{0,\beta}(\overline\Omega)$, then one can get that 
\[
\lim_{\alpha \nearrow \beta} [u]_{C^{0,\alpha}(\overline\Omega)}=[u]_{C^{0,\beta}(\overline\Omega)},
\]
for every open set $\Omega\subseteq\mathbb{R}^N$. In light of \eqref{liminf_holder}, it sufficient to prove
\[
\limsup_{\alpha \nearrow \beta} [u]_{C^{0,\alpha}(\overline\Omega)}\le [u]_{C^{0,\beta}(\overline\Omega)}.
\]
For this, it is sufficient to note that for every $x\not=y$ we have
\[
\frac{|u(x)-u(y)|}{|x-y|^\alpha}=\left(\frac{|u(x)-u(y)|}{|x-y|^\beta}\right)^\frac{\alpha}{\beta}\,|u(x)-u(y)|^\frac{\beta-\alpha}{\beta}\leq \left(\frac{|u(x)-u(y)|}{|x-y|^\beta}\right)^\frac{\alpha}{\beta}\, \left(2\, \|u\|_{L^\infty(\Omega)}\right)^\frac{\beta-\alpha}{\beta},
\]
which entails that for every $0<\alpha<\beta\le 1$
\begin{equation}
\label{needed}
[u]_{C^{0,\alpha}(\overline\Omega)}\le \left(2\, \|u\|_{L^\infty(\Omega)}\right)^\frac{\beta-\alpha}{\beta}\,[u]_{C^{0,\beta}(\overline\Omega)}.
\end{equation}
Observe that we are now assuming $u\in C^{0,\beta}(\overline\Omega)$, thus $u$ is bounded.
\end{remark}
Finally, we need the following simple $\Gamma-$convergence--type result.
\begin{lemma}
\label{lemma:gamma-conv}
Let $0<s<1$ and let $\Omega\subseteq\mathbb{R}^N$ be an open bounded set. Let $\{u_n\}_{n\in\mathbb{N}}\subseteq C(\overline\Omega)$ be such that 
\[
u_n\in W^{s,p_n}(\Omega)\qquad \mbox{ and }\qquad \liminf_{n\to\infty} [u_n]_{W^{s,p_n}(\Omega)}<+\infty,
\]
where $\{p_n\}_{n\in\mathbb{N}}\subseteq (1,+\infty)$ is an increasingly diverging sequence.
Let us suppose that $\{u_n\}_{n\in\mathbb{N}}$ converges pointwise on $\overline\Omega$ to a function $u\in C(\overline{\Omega})$. Then $u\in C^{0,s}(\overline\Omega)$ and 
\[
[u]_{C^{0,s}(\overline\Omega)}\le \liminf_{n\to\infty} [u_n]_{W^{s,p_n}(\Omega)}.
\]
\end{lemma}
\begin{proof}
Let us fix $0<t<s$ and $1\le q<\infty$. From Lemma \ref{lemma:embeddinglow}, we have 
\[
\liminf_{n\to\infty}[u_n]_{W^{t,q}(\Omega)}\le \left(\frac{N\,\omega_N\,|\Omega|}{s-t}\,\frac{1}{q}\right)^\frac{1}{q}\,\Big(\mathrm{diam}(\Omega)\Big)^{s-t}\,\liminf_{n\to\infty}[u_n]_{W^{s,p_n}(\Omega)}.
\]
By Fatou's Lemma, we also have 
\[
\liminf_{n\to\infty}[u_n]_{W^{t,q}(\Omega)}\ge [u]_{W^{t,q}(\Omega)}.
\]
This shows that $u\in W^{t,q}(\Omega)$ for every $0<t<s$ and every $1\le q<\infty$. Moreover, we have the estimate
\[
\limsup_{q\to\infty}[u]_{W^{t,q}(\Omega)}\le \Big(\mathrm{diam}(\Omega)\Big)^{s-t}\,\liminf_{n\to\infty}[u_n]_{W^{s,p_n}(\Omega)}.
\]
By Lemma \ref{lemma:2}, this implies that $u\in C^{0,t}(\overline\Omega)$, with
\[
[u]_{C^{0,t}(\overline\Omega)}\le \Big(\mathrm{diam}(\Omega)\Big)^{s-t}\,\liminf_{n\to\infty}[u_n]_{W^{s,p_n}(\Omega)}.
\]
Since this is valid for every $t<s$, we now easily get the desired conclusion.
\end{proof}

\subsection{Hardy's inequality on the punctured space}
According to \cite[Theorem 1.1]{FS}, we have the fractional Hardy inequality for functions on $\mathbb{R}^N\setminus\{0\}$. More precisely, for $N\ge 1$, $0<s<1$ and $1\le p<\infty$ such that $s\,p\not=N$, we have
\[
\mathfrak{h}_{s,p}(\mathbb{R}^N\setminus\{0\})\,\int_{\mathbb{R}^N} \frac{|\varphi|^p}{|x|^{s\,p}}\,dx\le [\varphi]_{W^{s,p}(\mathbb{R}^N)}^p,\qquad \mbox{ for every } \varphi\in \mathscr{D}^{s,p}_0(\mathbb{R}^N\setminus\{0\}).
\]
The sharp constant $\mathfrak{h}_{s,p}(\mathbb{R}^N\setminus\{0\})>0$ is computed in \cite{FS}. In order to analyze the asymptotic behaviour of the sharp constant in our Morrey's inequality, the following simple result will be useful. An analogous estimate for the case $s\,p<N$ was obtained in \cite[Theorem 2]{MS}, without prior knowledge of the sharp constant.
\begin{lemma}
\label{lemma:fastidio}
Let $N\ge 1$, $0<s<1$ and $1<p<\infty$ be such that $s\,p>N$. Then we have
\[
\mathfrak{h}_{s,p}(\mathbb{R}^N\setminus\{0\})\ge C\, \frac{(s\,p-N)^p}{(1-s)},
\]
for a constant $C=C(N,p)>0$.
\end{lemma}
\begin{proof}
From \cite[Theorem 1.1]{FS}, we know that 
\[
\mathfrak{h}_{s,p}(\mathbb{R}^N\setminus\{0\})=2\,\int_0^1 r^{s\,p-1}\,\left|1-r^\frac{N-s\,p}{p}\right|^p\,\Phi_{N,s\,p}(r)\,dr,
\]
where for $\alpha\ge 0$ and $0<r<1$, the quantity $\Phi_{N,\alpha}(r)$ is given by
\begin{equation}
\label{phistrange}
\Phi_{N,\alpha}(r)=|\mathbb{S}^{N-2}|\,\int_{-1}^1 \frac{(1-t^2)^\frac{N-3}{2}}{(1-2\,t\,r+r^2)^\frac{N+s\,p}{2}}\,dt,\qquad \mbox{ for }N\ge 2,
\end{equation}
and
\[
\Phi_{1,\alpha}(r)=\left(\frac{1}{(1-r)^{1+\alpha}}+\frac{1}{(1+r)^{1+\alpha}}\right).
\]
With simple manipulations, we first rewrite the sharp constant as follows 
\[
\begin{split}
\mathfrak{h}_{s,p}(\mathbb{R}^N\setminus\{0\})
&=2\,\int_0^1 r^{N-1}\,\left(1-r^\frac{s\,p-N}{p}\right)^p\,\Phi_{N,s\,p}(r)\,dr.
\end{split}
\]
By using that $\tau\mapsto \tau^{(s\,p-N)/p}$ is concave, we get
\[
1-r^\frac{s\,p-N}{p}\ge\frac{s\,p-N}{p}\,(1-r),\qquad \mbox{ for }0<r<1.
\]
This yields
\begin{equation}
\label{h1}
\mathfrak{h}_{s,p}(\mathbb{R}^N\setminus\{0\})\ge \left(\frac{s\,p-N}{p}\right)^p\,\int_0^1 r^{N-1}\,(1-r)^p\,\Phi_{N,s\,p}(r)\,dr.
\end{equation}
We are left with estimating $\Phi_{N,s\,p}$. Let us start with the case $N\ge 2$. With simple manipulations, we can rewrite it as follows
\[
\Phi_{N,s\,p}(r)=|\mathbb{S}^{N-2}|\,\left[\int_{0}^1 \frac{(1-t^2)^\frac{N-3}{2}}{(1-2\,t\,r+r^2)^\frac{N+s\,p}{2}}\,dt+\int_{0}^1 \frac{(1-t^2)^\frac{N-3}{2}}{(1+2\,t\,r+r^2)^\frac{N+s\,p}{2}}\,dt\right].
\]
By dropping the second integral and observing that for $t\in[0,1]$
\[
1-2\,t\,r+r^2=(1-r)^2+2\,r\,(1-t)\qquad \mbox{ and }\qquad (1-t^2)^\frac{N-3}{2}\ge \sqrt{\frac{1}{2}}\,(1-t)^\frac{N-3}{2},
\]
we get
\[
\Phi_{N,s\,p}(r)\ge \sqrt{\frac{1}{2}}\,|\mathbb{S}^{N-2}|\,\frac{1}{(1-r)^{N+s\,p}}\,\int_{0}^1 \frac{(1-t)^\frac{N-3}{2}}{\left(1+2\,r\,\dfrac{1-t}{(1-r)^2}\right)^\frac{N+s\,p}{2}}\,dt.
\]
For every $0<r<1$, we perform the change of variable $(1-t)/(1-r)^2=\tau$, which gives
\[
\Phi_{N,s\,p}(r)\ge \sqrt{\frac{1}{2}}\,|\mathbb{S}^{N-2}|\,\frac{1}{(1-r)^{1+s\,p}}\,\int_{0}^\frac{1}{(1-r)^2} \frac{\tau^\frac{N-3}{2}}{(1+2\,r\,\tau)^\frac{N+s\,p}{2}}\,d\tau.
\]
The last integral can be further estimated from below by
\[
\int_{0}^\frac{1}{(1-r)^2} \frac{\tau^\frac{N-3}{2}}{(1+2\,r\,\tau)^\frac{N+s\,p}{2}}\,d\tau\ge \frac{1}{3^{N+s\,p}}\,\int_{0}^{1} \tau^\frac{N-3}{2}\,d\tau\ge \frac{1}{3^{N+p}}\,\frac{2}{N-1}.
\]
Thus, from \eqref{h1} we can infer
\[
\mathfrak{h}_{s,p}(\mathbb{R}^N\setminus\{0\})\ge C\,\left(\frac{s\,p-N}{p}\right)^p\,\int_0^1 r^{N-1}\,(1-r)^{p-1-s\,p}\,dr,
\]
where $C=C(N,p)>0$. The last integral can be finally bounded from below as follows
\[
\int_0^1 r^{N-1}\,(1-r)^{p-1-s\,p}\,dr\ge 2^{1-N}\,\int_\frac{1}{2}^1 (1-r)^{p-1-s\,p}\,dr=\frac{2^{1-N}}{p\,(1-s)}\,\left(\frac{1}{2}\right)^{p\,(1-s)}.
\]
This concludes the proof for $N\ge 2$. Finally, the case $N=1$ is simpler and it is left to the reader. 
\end{proof}

\section{Some expedient fractional Poincar\'e--type inequalities}
\label{sec:3}

We first recall the definition of fractional capacity we need to work with. See also \cite{AFN, Ri, SX2, SX} and \cite{Warma} for further studies on other definitions of fractional capacities.
\begin{definition}
\label{defi:capacity}
Let $\Sigma\subseteq\mathbb{R}^N$ be a compact set and let $E\subseteq\mathbb{R}^N$ an open set such that $\Sigma\Subset E$. For every $0<s<1$ and $1<p<\infty$, we define the {\it fractional $p-$capacity of order $s$ of $\Sigma$ relative to $E$} as the quantity
\[
\widetilde{\mathrm{cap}}_{s,p}(\Sigma;E)=\inf_{\varphi\in C^\infty_0(E)} \Big\{[\varphi]^p_{W^{s,p}(\mathbb{R}^N)}\, :\, \varphi\ge 1 \mbox{ on } \Sigma\Big\}.
\]
\end{definition}
If $s\,p>N$, it is well-known that every point has positive fractional $p-$capacity of order $s$, relative to any bounded set. For our scopes, we need a quantitative version of this fact. At this aim, we need an ``endpoint'' Poincar\'e constant. Let $0<s<1$ and $1<p<\infty$ be such that $s\,p>N$. For every $\Omega\subseteq\mathbb{R}^N$ open set, we define
\begin{equation}
\label{lambda_endpoint}
\lambda_{p,\infty}^s(\Omega)=\inf_{\varphi\in C^\infty_0(\Omega)} \Big\{[\varphi]^p_{W^{s,p}(\mathbb{R}^N)}\, :\, \|\varphi\|_{L^\infty(\Omega)}=1\Big\}.
\end{equation}
Ii is not difficult to see that this infimum is unchanged, if we add the further restriction $\varphi\ge 0$.
The following lower bound will be crucial. 
\begin{lemma}
\label{lemma:drive}
Let $N\ge 1$, $0<s<1$ and $1<p<\infty$ be such that $s\,p>N$. For every $\Omega\subseteq\mathbb{R}^N$ open set with finite volume, we have
\[
|\Omega|^\frac{s\,p-N}{N}\,\lambda_{p,\infty}^s(\Omega)\ge N\,(\omega_N)^{\frac{s\,p}{N}}\,\frac{\mathfrak{h}_{s,p}(\mathbb{R}^N\setminus\{0\})}{s\,p-N}.
\]
\end{lemma}
\begin{proof}
Let $\varphi\in C^\infty_0(\Omega)$ be such that $\|\varphi\|_{L^\infty(\Omega)}=1$ and $\varphi\ge 0$. We take $\varphi^*$ the spherically symmetric decreasing rearrangement of $\varphi$. This is in particular a Lipschitz function, compactly supported in $\Omega^*$ (see for example \cite[Lemma 1]{Ta}). The latter is the ball centered at the origin, such that $|\Omega^*|=|\Omega|$. We observe that, by construction, we have
\[
\|\varphi^*\|_{L^\infty(\Omega^*)}=\|\varphi\|_{L^\infty(\Omega)}=1.
\]
We also have 
\[
1-\varphi^*\in \mathcal{W}^{s,p}(\mathbb{R}^N)\qquad \mbox{ and }\qquad 1-\varphi^*(0)=0.
\]
Thus in particular we have $1-\varphi^*\in\mathscr{D}^{s,p}_0(\mathbb{R}^N\setminus\{0\})$ by Proposition \ref{prop:homo}. We can then apply Hardy's inequality for $\mathbb{R}^N\setminus\{0\}$ and obtain
\[
\begin{split}
[1-\varphi^*]^p_{W^{s,p}(\mathbb{R}^N)}&\ge \mathfrak{h}_{s,p}(\mathbb{R}^N\setminus\{0\})\,\int_{\mathbb{R}^N} \frac{(1-\varphi^*)^p}{|x|^{s\,p}}\,dx\ge \mathfrak{h}_{s,p}(\mathbb{R}^N\setminus\{0\})\,\int_{\mathbb{R}^N\setminus\Omega^*} \frac{1}{|x|^{s\,p}}\,dx.
\end{split}
\]
We set $R=(|\Omega|/\omega_N)^{1/N}$, then the last integral can be computed and gives
\[
\int_{\mathbb{R}^N\setminus\Omega^*} \frac{1}{|x|^{s\,p}}\,dx=N\,\omega_N\,\int_R^{+\infty} \varrho^{N-1-s\,p}\,d\varrho=\frac{N\,\omega_N}{s\,p-N}\,R^{N-s\,p}=\frac{N\,\omega_N}{s\,p-N}\,\left(\frac{|\Omega|}{\omega_N}\right)^{1-\frac{s\,p}{N}}.
\]
Up to now, we have obtained
\[
[1-\varphi^*]^p_{W^{s,p}(\mathbb{R}^N)}\ge \frac{N\,\omega_N^{\frac{s\,p}{N}}}{s\,p-N}\,|\Omega|^{1-\frac{s\,p}{N}}\,\mathfrak{h}_{s,p}(\mathbb{R}^N\setminus\{0\}).
\]
It is only left to observe that 
\[
[1-\varphi^*]^p_{W^{s,p}(\mathbb{R}^N)}=[\varphi^*]^p_{W^{s,p}(\mathbb{R}^N)}\le [\varphi]^p_{W^{s,p}(\mathbb{R}^N)},
\]
where in the last inequality we used the fractional {\it P\'olya-Szeg\H{o} principle} (see \cite[Theorem 9.2]{AL}). By arbitrariness of $\varphi$, we get the desired conclusion.
\end{proof}
We can now give a good lower bound for the relative capacity of a point.
\begin{lemma}
\label{lemma:point}
Let $N\ge 1$, $0<s<1$ and $1<p<\infty$ be such that $s\,p>N$. For every $\Omega\subseteq\mathbb{R}^N$ open set with finite volume and $\overline{x}\in \Omega$, we have
\[
\widetilde{\mathrm{cap}}_{s,p}\big(\{\overline{x}\};\Omega\big)\ge N\,|\Omega|\,\left(\frac{\omega_N}{|\Omega|}\right)^\frac{s\,p}{N}\,\frac{\mathfrak{h}_{s,p}(\mathbb{R}^N\setminus\{0\})}{s\,p-N}.
\]
\end{lemma}
\begin{proof}
Let $\varphi\in C^\infty_0(\Omega)$ be such that $\varphi(\overline{x})\ge 1$. By using the very definition of $\lambda_{p,\infty}^s(\Omega)$, we have
\[
[\varphi]^p_{W^{s,p}(\mathbb{R}^N)}\ge \lambda_{p,\infty}^s(\Omega)\,\|\varphi\|_{L^\infty(\Omega)}^p\ge \lambda_{p,\infty}^s(\Omega).
\]
By taking the infimum over $\varphi$, we get in particular
\[
\widetilde{\mathrm{cap}}_{s,p}\big(\{\overline{x}\};\Omega\big)\ge \lambda_{p,\infty}^s(\Omega).
\]
By using Lemma \ref{lemma:drive}, we eventually get the claimed estimate.
\end{proof}
The following capacitary fractional Poincar\'e inequality {\it \`a la Maz'ya} can be obtained with minor modifications of the proof of \cite[Proposition 4.3]{BB}. The latter is in turn a fractional generalization of \cite[Theorem 14.1.2]{Maz}.

\begin{proposition}\label{prop:poincare_cap}
	Let $N\ge 1$, $0<s<1$ and $1<p<\infty$ be such that $s\,p>N$. There exists a constant $C=C(N,p)>0$ such that for every $\varrho\le r$ the following Poincar\'e inequality holds
\[
	[u]^p_{W^{s,p}(B_r(x_0))}\ge \frac{C}{r^N}\,\left(\frac{\varrho}{r}\right)^{s\,p}\,\widetilde{\mathrm{cap}}_{s,p}\big(\{x_0\};B_\varrho(x_0)\big)\,\|u\|^p_{L^p(B_r(x_0))},
\]
	for every $u\in C^1(\overline{B_r(x_0)})$ with $u(x_0)=0$. 
\end{proposition}
\begin{proof}
Let us take $0<\varrho\le r$ and $\eta\in C^\infty_0(B_{\varrho/2}(x_0))$ a cut-off function such that 
\[
\eta\equiv 1 \mbox{ on } B_\frac{\varrho}{4}(x_0),\qquad 0\le \eta\le 1,\qquad |\nabla \eta|\le \frac{C}{\varrho},
\]
for some constant $C=C(N)>0$. For every $u\in C^1(\overline{B_r(x_0)})$ such that $u(x_0)=0$ and 
\[
\fint_{B_r(x_0)}|u|^p\,dx=1,
\]
we have that the function $\varphi=(1-u)\,\eta$ is admissible for the definition of the capacity of the point $x_0$. Thus we have
\[
\begin{split}
\widetilde{\mathrm{cap}}_{s,p}\big(\{x_0\};B_\varrho(x_0)\big)&\le [(1-u)\,\eta]_{W^{s,p}(\mathbb{R}^N)}^p=\big[(1-u)\,\eta\big]^p_{W^{s,p}(B_\varrho(x_0))}\\
&+2\,\int_{B_\varrho(x_0)}\,|1-u(x)|^p\,|\eta(x)|^p \left(\int_{\mathbb{R}^N\setminus B_\varrho(x_0)}\frac{dy}{|x-y|^{N+s\,p}}\right)\,dx\\
&\le \big[(1-u)\,\eta\big]^p_{W^{s,p}(B_\varrho(x_0))}\\
&+2\,\int_{B_\frac{\varrho}{2}(x_0)}\,|1-u(x)|^p\,\left(\int_{\mathbb{R}^N\setminus B_\varrho(x_0)}\frac{dy}{|x-y|^{N+s\,p}}\right)\,dx\\ 
&\le 2^{p-1}\,\iint_{B_\varrho(x_0)\times B_\varrho(x_0)} \frac{|u(x)-u(y)|^p}{|x-y|^{N+s\,p}}\,\eta(x)^p\,dx\,dy\\
&+2^{p-1}\,\iint_{B_\varrho(x_0)\times B_\varrho(x_0)} \frac{|\eta(x)-\eta(y)|^p}{|x-y|^{N+s\,p}}\,|1-u(y)|^p\,dx\,dy\\
&+2\,\int_{B_\frac{\varrho}{2}(x_0)}\,|1-u(x)|^p\,\left(\int_{\mathbb{R}^N\setminus B_\varrho(x_0)}\frac{dy}{|x-y|^{N+s\,p}}\right)\,dx.
\end{split}
\]
By using that $B_\varrho(x_0)\subseteq B_r(x_0)$ and $|\eta|\le 1$, we observe that 
\[
\iint_{B_\varrho(x_0)\times B_\varrho(x_0)} \frac{|u(x)-u(y)|^p}{|x-y|^{N+s\,p}}\,\eta(x)^p\,dx\,dy\le [u]_{W^{s,p}(B_r(x_0))}^p.
\]
As for the second double integral, by using the Lipschitz estimate on $\eta$, we have 
\[
\begin{split}
\iint_{B_\varrho(x_0)\times B_\varrho(x_0)}& \frac{|\eta(x)-\eta(y)|^p}{|x-y|^{N+s\,p}}\,|1-u(y)|^p\,dx\,dy\\
&\le \frac{C}{\varrho^p}\,\iint_{B_\varrho(x_0)\times B_\varrho(x_0)} |x-y|^{p-N-s\,p}\,|1-u(y)|^p\,dx\,dy\\
&\le \frac{C}{\varrho^p}\,\int_{B_\varrho(x_0)}\,\left(\int_{B_{2\varrho}(y)} |x-y|^{p-N-s\,p}\,dx\right)\,|1-u(y)|^p\,dy\\
&=\frac{C\,N\,\omega_N}{2^{p-s\,p}\,p\,(1-s)}\,\frac{1}{\varrho^{s\,p}}\,\int_{B_\varrho(x_0)} |1-u(y)|^p\,dy.
\end{split}
\]
Up to now, we obtained
\[
\begin{split}
\widetilde{\mathrm{cap}}_{s,p}\big(\{x_0\};B_\varrho(x_0)\big)&\le 2^{p-1}\,[u]_{W^{s,p}(B_r(x_0))}^p+\frac{C}{\varrho^{s\,p}}\,\frac{1}{1-s}\,\int_{B_\varrho(x_0)} |1-u(y)|^p\,dy\\
&+2\,\int_{B_\frac{\varrho}{2}(x_0)}\,|1-u(x)|^p\,\left(\int_{\mathbb{R}^N\setminus B_\varrho(x_0)}\frac{dy}{|x-y|^{N+s\,p}}\right)\,dx,
\end{split}
\]
for some $C=C(N,p)>0$. In turn, for the third integral we can observe that for every $x\in B_{\varrho/2}(x_0)$ and $y\not\in B_\varrho(x_0)$, we have
\[
|x-y|\ge |y-x_0|-|x-x_0|\ge |y-x_0|-\frac{\varrho}{2}\ge \frac{|y-x_0|}{2}.
\]
Thus, we get
\[
\begin{split}
\int_{B_\frac{\varrho}{2}(x_0)}&\,|1-u(x)|^p\,\left(\int_{\mathbb{R}^N\setminus B_\varrho(x_0)}\frac{dy}{|x-y|^{N+s\,p}}\right)\,dx\\
&\le 2^{N+s\,p}\,\int_{B_\frac{\varrho}{2}(x_0)}\,|1-u(x)|^p\,\left(\int_{\mathbb{R}^N\setminus B_\varrho(x_0)}\frac{dy}{|y-x_0|^{N+s\,p}}\right)\,dx\\
&=\frac{2^{N+s\,p}}{s\,p} \,\frac{N\,\omega_N}{\varrho^{s\,p}}\,\int_{B_\frac{\varrho}{2}(x_0)}\,|1-u(x)|^p\,dx\\
&\le \frac{2^{N+s\,p}}{s\,p} \,\frac{N\,\omega_N}{\varrho^{s\,p}}\,\int_{B_\varrho(x_0)}\,|1-u(x)|^p\,dx
\end{split}
\]
In conclusion, by also using that $N/sp<1$, we get
\[
\widetilde{\mathrm{cap}}_{s,p}\big(\{x_0\};B_\varrho(x_0)\big)\le 2^{p-1}\,[u]_{W^{s,p}(B_r(x_0))}^p+\frac{C}{\varrho^{s\,p}}\,\frac{1}{1-s}\,\int_{B_\varrho(x_0)} |1-u(y)|^p\,dy,
\]
for some $C=C(N,p)>0$.
We are only left with estimating the $L^p$ norm of $1-u$. At this aim, we observe that
\[
\|1-u\|_{L^p(B_\varrho(x_0))}\le \left\|1-\mathrm{av}\left(u;B_\varrho(x_0)\right)\right\|_{L^p(B_\varrho(x_0))}+\left\|u-\mathrm{av}\left(u;B_\varrho(x_0)\right)\right\|_{L^p(B_\varrho(x_0))}.
\]
The second term can be estimated thanks to the fractional Poincar\'e-Wirtinger inequality, i.e.\footnote{The proof of this inequality can be extrapolated from that of \cite[Corollary 1, page 524]{Maz}. Alternatively, it is sufficient to repeat verbatim the proof of \cite[Lemma 3.5]{BB}, for $p\not=2$. We point out that the presence of the factor $(1-s)$ is important for us. A fractional Poincar\'e-Wirtinger inequality without this factor can be found for example in \cite[equation (4.2)]{Mi2003}.}
\begin{equation}
\label{fracPW}
\left\|u-\mathrm{av}\left(u;B_\varrho(x_0)\right)\right\|^p_{L^p(B_\varrho(x_0))}\le C\,(1-s)\,\varrho^{s\,p}\,[u]^p_{W^{s,p}(B_\varrho(x_0))},
\end{equation}
where $C=C(N,p)>0$. 
For the first term, we recall the normalization on the $L^p$ norm of $u$. We also observe that, without loss of generality, we can suppose that
\[
\mathrm{av}\left(u;B_\varrho(x_0)\right)\ge 0.
\]
Thus, we have
\[
\begin{split}
\left\|1-\mathrm{av}\left(u;B_\varrho(x_0)\right)\right\|_{L^p(B_\varrho(x_0))}&=|B_\varrho(x_0)|^\frac{1}{p} \left|1-\mathrm{av}\left(u;B_\varrho(x_0)\right)\right|\\
&=\left(\frac{|B_\varrho(x_0)|}{|B_r(x_0)|}\right)^\frac{1}{p}\, \left|\|u\|_{L^p(B_r(x_0)}-\left\|\mathrm{av}\left(u;B_\varrho(x_0)\right)\right\|_{L^p(B_r(x_0)}\right|\\
&\le \left(\frac{|B_\varrho(x_0)|}{|B_r(x_0)|}\right)^\frac{1}{p}\,\left(\int_{B_r(x_0)} \left|u-\mathrm{av}\left(u;B_\varrho(x_0)\right)\right|^p\,dx\right)^\frac{1}{p}.
\end{split}
\]
For the last term, we have 
\[
\begin{split}
\left(\int_{B_r(x_0)} \left|u-\mathrm{av}\left(u;B_\varrho(x_0)\right)\right|^p\,dx\right)^\frac{1}{p}&\le \left(\int_{B_r(x_0)} \left|u-\mathrm{av}\left(u;B_r(x_0)\right)\right|^p\,dx\right)^\frac{1}{p}\\
&+\left(\int_{B_r(x_0)} \left|\mathrm{av}\left(u;B_r(x_0)\right)-\mathrm{av}\left(u;B_\varrho(x_0)\right)\right|^p\,dx\right)^\frac{1}{p}.
\end{split}
\]
Moreover, by using Jensen's inequality, we get
\[
\begin{split}
\int_{B_r(x_0)} \left|\mathrm{av}\left(u;B_r(x_0)\right)-\mathrm{av}\left(u;B_\varrho(x_0)\right)\right|^p\,dx&=\int_{B_r(x_0)} \left|\fint_{B_\varrho(x_0)}[u(y)-\mathrm{av}\left(u;B_r(x_0)\right)]\,dy\right|^p\,dx\\
&\le \int_{B_r(x_0)} \fint_{B_\varrho(x_0)}|u(y)-\mathrm{av}\left(u;B_r(x_0)\right)|^p\,dy\,dx\\
&\le \frac{r^N}{\varrho^N}\, \int_{B_r(x_0)}|u(y)-\mathrm{av}\left(u;B_r(x_0)\right)|^p\,dy.
\end{split}
\]
We conclude by using again the fractional Poincar\'e-Wirtinger inequality \eqref{fracPW}, this time on the ball $B_r(x_0)$.
\end{proof}

\section{A proof of Morrey's inequality}

\label{sec:4}

We need to introduce a further Poincar\'e--type constant. For every ball $B_r(x_0)$
and every $0<s<1$, $1<p<\infty$ such that $s\,p>N$, we define
\begin{equation}
\label{lambdone}
\Lambda_{s,p}(B_r(x_0)):=\inf_{\varphi\in C^1(\overline{B_r(x_0)})} \Big\{[\varphi]_{W^{s,p}(B_r(x_0))}^p\, :\, \|\varphi\|_{L^p(B_r(x_0))}=1 \mbox{ and } \varphi(x_0)=0\Big\}.
\end{equation}
We observe that such a constant is positive. Indeed, by using Proposition \ref{prop:poincare_cap} with $\varrho=r$, we get
\begin{equation}
\label{lam_cap}
\Lambda_{s,p}(B_r(x_0))\ge \frac{C\,s}{r^N}\,\widetilde{\mathrm{cap}}_{s,p}\big(\{x_0\};B_r(x_0)\big),
\end{equation}
and the relative capacity of a point is positive (see Lemma \ref{lemma:point}). 
We also observe that $\Lambda_{s,p}$ enjoys the following scaling and translational rule
\begin{equation}
\label{scalingalpha}
\Lambda_{s,p}(B_r(x_0))=r^{-s\,p}\,\Lambda_{s,p}(B_1(0)).
\end{equation}
The main interest of the positive constant $\Lambda_{s,p}$ is that it permits to bound from below the constant $\mathfrak{m}_{s,p}(\mathbb{R}^N)$. The key estimate is the following one. Loosely speaking, this may be seen as a ``pointwise fractional Hardy inequality'', in nuce (see \cite{Ha, Sk} for more details on the subject).

\begin{lemma}
\label{lm:1stjuly}
Let $N\ge 1$, $0<s<1$ and $1<p<\infty$ be such that $s\,p>N$. For every $r>0$, $x_0\in\mathbb{R}^N$ and $\varphi\in C^1(\overline{B_r(x_0)})$ such that $\varphi(x_0)=0$, we have
\[
|\varphi(x)|\le \left(\frac{1}{\omega_N\,\Lambda_{s,p}(B_1(0))}\right)^\frac{1}{p}\,\frac{\big(r-|x-x_0|\big)^s+r^s}{\big(r-|x-x_0|\big)^\frac{N}{p}}\,[\varphi]_{W^{s,p}(B_r(x_0))},\quad \mbox{ for } x\in B_r(x_0).
\] 
\end{lemma}
\begin{proof}
We pick a point $x\in B_r(x_0)$ and we set for brevity $\delta=r-|x-x_0|$. For every $\varphi$ as in the statement, we estimate 
\[
|\varphi(x)|\le|\varphi(x)-\varphi(z)|+|\varphi(z)|,\qquad \mbox{ for every }z\in B_\delta(x).
\]
By taking the average over $B_\delta(x)$ and using Jensen's inequality, we get
\begin{equation}
\label{pointwise}
|\varphi(x)|\le \left(\fint_{B_\delta(x)}|\varphi(x)-\varphi(z)|^p\,dz\right)^\frac{1}{p}+\left(\fint_{B_\delta(x)}|\varphi(z)|^p\,dz\right)^\frac{1}{p}.
\end{equation}
We now observe that the function
\[
v(z)=\frac{\varphi(z)-\varphi(x)}{\|\varphi(\cdot)-\varphi(x)\|_{L^p(B_\delta(x))}},
\]
is admissible for the problem which defines $\Lambda_{s,p}(B_\delta(x))$. Thus we get
\[
\Lambda_{s,p}(B_\delta(x))\le [v]_{W^{s,p}(B_\delta(x))}^p=\frac{[\varphi]_{W^{s,p}(B_\delta(x))}^p}{\displaystyle\int_{B_\delta(x)} |\varphi(z)-\varphi(x)|^p\,dx}.
\]
By taking into account \eqref{scalingalpha}, this can be recast into
\[
\int_{B_\delta(x)} |\varphi(z)-\varphi(x)|^p\,dx\le \frac{\delta^{s\,p}}{\Lambda_{s,p}(B_1(0))}\,[\varphi]_{W^{s,p}(B_\delta(x))}^p\le \frac{\delta^{s\,p}}{\Lambda_{s,p}(B_1(0))}\,[\varphi]_{W^{s,p}(B_r(x_0))}^p.
\]
From \eqref{pointwise}, we thus obtain
\[
|\varphi(x)|\le \frac{1}{(\omega_N\,\delta^N)^\frac{1}{p}}\,\frac{\delta^{s}}{\Big(\Lambda_{s,p}(B_1(0))\Big)^\frac{1}{p}}\,[\varphi]_{W^{s,p}(B_r(x_0))}+\left(\fint_{B_\delta(x)}|\varphi(z)|^p\,dz\right)^\frac{1}{p}.
\]
Finally, for the last term we can simply observe that
\[
\fint_{B_\delta(x)}|\varphi(z)|^p\,dz\le \frac{1}{\omega_N\,\delta^N}\,\int_{B_r(x_0)} |\varphi(z)|^p\,dz\le \frac{1}{\omega_N\,\delta^N}\,\frac{1}{\Lambda_{s,p}(B_r(x_0))}\,[\varphi]_{W^{s,p}(B_r(x_0))}^p.
\]
By using again \eqref{scalingalpha}, we conclude.
\end{proof}
By virtue of the previous result, we can now prove the validity of the fractional homogeneous Morrey inequality on the whole space, together with a quantitative lower bound. We use the notation $\mathfrak{m}_{s,p}(\mathbb{R}^N)$ introduced in \eqref{morreysharp}.
\begin{theorem}
\label{thm:morrey}
Let $N\ge 1$, $0<s<1$ and $1<p<\infty$ be such that $s\,p>N$. Let us define 
the constant
\[
\theta_{N,s,p}:=\omega_N\,\sup_{T>1} \frac{(T-1)^N}{\Big((T-1)^s+T^s\Big)^p}.
\] 
With the notation above, we have 
\[
\theta_{N,s,p}\,\Lambda_{s,p}(B_1(0))\le\mathfrak{m}_{s,p}(\mathbb{R}^N).
\]
\end{theorem}
\begin{proof}
We take $\varphi\in C^\infty_0(\mathbb{R}^N)$, then we fix  two points $x\not= y$ and set $\delta=|x-y|$. We also take $T>1$ and consider the function
\[
v(z)=\varphi(z)-\varphi(y),\qquad \mbox{ for } z\in B_{T\delta}(y).
\]
We observe that, by construction, $x\in B_{T\delta}(y)$.
By applying Lemma \ref{lm:1stjuly} to the function $v$, which vanishes at the center of the ball $B_{T\delta}(y)$, we get
\begin{equation}
\label{mor}
\begin{split}
|\varphi(x)-\varphi(y)|=|v(x)|&\le \left(\frac{1}{\omega_N\,\Lambda_{s,p}(B_1(0))}\right)^\frac{1}{p}\,\frac{(T-1)^s\,\delta^s+T^s\,\delta^s}{\big((T-1)\,\delta\big)^\frac{N}{p}}\,[\varphi]_{W^{s,p}(B_{T\,\delta }(y))}\\
&\leq \left(\frac{1}{\omega_N\,\Lambda_{s,p}(B_1(0))}\right)^\frac{1}{p}\,\frac{(T-1)^s+T^s}{(T-1)^\frac{N}{p}}\,\delta^{s-\frac{N}{p}}\,[\varphi]_{W^{s,p}(\mathbb{R}^N)}
\end{split}
\end{equation}
By recalling that $\delta=|x-y|$, the previous estimate permits to conclude. 
\end{proof}
\begin{remark}
We do not expect the lower bound obtained in Theorem \ref{thm:morrey} to be sharp. However, we will see in Section \ref{sec:8} that it becomes ``qualitatively sharp'' in the limit $s\searrow N/p$, as explained in the Introduction. 
This is the main bonus feature of using the Poincar\'e constant $\Lambda_{s,p}$. 
\par
Moreover, it also becomes ``quantitatively sharp'' in the asymptotic regime $p\nearrow \infty$: in this case, we will see in Section \ref{sec:8} that we have
\[
\lim_{p\nearrow \infty} \Big(\mathfrak{m}_{s,p}(\mathbb{R}^N)\Big)^\frac{1}{p}=\lim_{p\nearrow \infty}\Big(\theta_{N,s,p}\,\Lambda_{s,p}(B_1(0))\Big)^\frac{1}{p}=1.
\]
\end{remark}

\section{The H\"older seminorm of functions in $\mathcal{W}^{s,p}$}

\label{sec:5}

In this section, we will show that for every $\varphi\in \mathcal{W}^{s,p}(\mathbb{R}^N)$, the supremum in the $C^{0,\alpha_{s,p}}$ seminorm is actually a maximum.
We need at first the following ``localized'' fractional Morrey inequality. This is analogous to \cite[equation (6.1)]{HS}.
\begin{lemma}
\label{lm:pratignano}
Let $N\ge 1$, $0<s<1$ and $1<p<\infty$ be such that $s\,p>N$. For every open ball $B_r(x_0)\subseteq \mathbb{R}^N$, every $\varphi\in C(\overline{B_r(x_0)})$ and every $x,y\in \overline{B_r(x_0)}$ such that
\[
B_\frac{\delta}{2}\left(\frac{x+y}{2}\right)\subseteq \overline{B_r(x_0)},\qquad \mbox{ for } \delta=|x-y|,
\] 
we have 
\begin{equation}
\label{morreymodif}
\frac{|\varphi(x)-\varphi(y)|}{|x-y|^{s-\frac{N}{p}}}\le  4^\frac{4\,N+1}{p}\,  \left(\frac{2^{N+2\,p}}{\omega_N\,\Lambda_{s,p}(B_1(0))}\right)^\frac{1}{p}\,[\varphi]_{W^{s,p}(B_\frac{\delta}{2}(\frac{x+y}{2}))}.
 \end{equation}
\end{lemma}
\begin{proof}
We can suppose that 
\[
[\varphi]_{W^{s,p}(B_\frac{\delta}{2}(\frac{x+y}{2}))}<+\infty,
\]
otherwise there is nothing to prove. Then, thanks to the linear extension operator of \cite[Proposition 3.1]{BB},  there exists $\widetilde u \in  W^{s,p}(B_{2\delta}((x+y)/2)$  such that $ \tilde \varphi=\varphi$ on $\overline{B_{\delta/2}((x+y)/2)}$ and 
 \begin{equation}
 \label{eq:extens} [\tilde \varphi ]_{W^{s,p}(B_{2\delta}(\frac {x + y}{2}))}\leq 4^\frac{4\,N+1}{p}\,   [\varphi]_{W^{s,p}(B_\frac{\delta}{2}(\frac {x + y}{2}))} .
 \end{equation}
By the construction of this operator, we also know that $\widetilde{\varphi}\in C(\overline{B_{2\delta}((x+y)/2)})$. Let $\rho\in C^\infty_0(\mathbb{R}^n)$ be the usual Friedrichs mollifier supported on the unit ball $B_1(0)$. We define the sequence of smoothing kernels
\[
\rho_n(x)=n^N\,\rho(n\,x),\qquad \mbox{ for } n\in\mathbb{N}\setminus\{0\}.
\]
By the properties of the convolutions, we have
\begin{equation}
\label{approssima_prat}
\lim_{n\to\infty} [\varphi_n-\widetilde\varphi]_{W^{s,p}(B_{\delta}(\frac {x + y}{2}))}=0\qquad \mbox{ and }\qquad \lim_{n\to\infty} \|\varphi_n-\widetilde\varphi\|_{L^\infty(B_{\delta}(\frac {x + y}{2}))}=0,
\end{equation}
where $\varphi_n=\widetilde\varphi\ast \rho_n$. 
By proceeding as for the estimate \eqref{mor} with $(x+y)/2$ in place of $y$ and choosing $T=2$, for every $n\in\mathbb{N}$ we have
\[
\left|\varphi_n(x)-\varphi_n\left(\frac{x+y}{2}\right)\right|\le \left(\frac{2^{p\,(s+1)}}{\omega_N\,\Lambda_{s,p}(B_1(0))}\right)^\frac{1}{p}\,\left(\frac{\delta}{2}\right)^{s-\frac{N}{p}}\,[\varphi_n]_{W^{s,p}(B_{\delta }(\frac{x+y}{2}))},
\]
and
\[
\left|\varphi_n(y)-\varphi_n\left(\frac{x+y}{2}\right)\right|\le \left(\frac{2^{p\,(s+1)}}{\omega_N\,\Lambda_{s,p}(B_1(0))}\right)^\frac{1}{p}\,\left(\frac{\delta}{2}\right)^{s-\frac{N}{p}}\,[\varphi_n]_{W^{s,p}(B_{\delta }(\frac{x+y}{2}))}.
\]
By summing up and using the triangle inequality, we get
\[
\frac{|\varphi_n(x)-\varphi_n(y)|}{|x-y|^{s-\frac{N}{p}}}\le  \left(\frac{2^{N+2\,p}}{\omega_N\,\Lambda_{s,p}(B_1(0))}\right)^\frac{1}{p}\,[\varphi_n]_{W^{s,p}(B_{\delta }(\frac{x+y}{2}))}.
\]
By taking the limit as $n$ goes to $\infty$ and using \eqref{approssima_prat}, we get
\[
\frac{|\widetilde\varphi(x)-\widetilde\varphi(y)|}{|x-y|^{s-\frac{N}{p}}}\le  \left(\frac{2^{N+2\,p}}{\omega_N\,\Lambda_{s,p}(B_1(0))}\right)^\frac{1}{p}\,[\widetilde\varphi]_{W^{s,p}(B_{\delta }(\frac{x+y}{2}))}.
\]
Finally, by using \eqref{eq:extens} and the fact that $\varphi(x)=\widetilde{\varphi}(x)$, $\varphi(y)=\widetilde{\varphi}(y)$, we conclude.
\end{proof}
\begin{remark}[A subtle detail]
We observe that in one-dimensional case $N=1$, from \eqref{morreymodif} we get for $s\,p>1$ and $x<y$ the following estimate\footnote{Indeed, for $N=1$ and $x<y$ we have
\[
B_\frac{\delta}{2}\left(\frac {x + y}{2}\right)=\left(\frac{x+y}{2}-\frac{\delta}{2},\frac{x+y}{2}+\frac{\delta}{2}\right)=(x,y).
\]}
\begin{equation}
\label{onedimensional}
\frac{|\varphi(x)-\varphi(y)|^p}{|x-y|^{s\,p-1}}\le \frac{C_p}{\Lambda_{s,p}((-1,1))}\,\iint_{(x,y)\times(x,y)} \frac{|\varphi(t)-\varphi(\tau)|^p}{|t-\tau|^{1+s\,p}}\,dt\,d\tau,
\end{equation}
which is the same as \cite[Lemma 4.5]{DF}, the latter in turn taken from \cite[Lemma 1.1]{GRR}. However, we point out that in \cite{DF} this estimate is obtained with a constant $C_{s,p}>0$ such that
\[
0<\lim_{s\searrow \frac{1}{p}} (s\,p-1)^p\,C_{s,p}<+\infty.
\]
On the contrary, in \eqref{onedimensional} the asymptotic behaviour of the constant is dictated by $\Lambda_{s,p}((-1,1))$, which enjoys the following property
\[
\limsup_{s\searrow \frac{1}{p}}\frac{(s\,p-1)^{p-1}}{\Lambda_{s,p}((-1,1))}<+\infty,
\]
see \eqref{loweralpha} below, in conjunction with Lemma \ref{lemma:fastidio}.
\end{remark}
\begin{lemma}
\label{lm:morrey_firoj}
Let $R>0$ and let $x,\, y\in\mathbb{R}^N\setminus B_{2R}(0)$ be two distinct points. Let $N\ge 1$, $0<s<1$ and $1<p<\infty$ be such that $s\,p>N$.
There exists a constant $\widetilde{C}=\widetilde{C}(N,s,p)>0$ such that for every $\varphi\in \mathcal{W}^{s,p}(\mathbb{R}^N)$ 
we have
	\[
	\frac{|\varphi(x)-\varphi(y)|}{|x-y|^{s-\frac{N}{p}}}\leq \widetilde{C}\,[\varphi]_{W^{s,p}(\mathbb{R}^N\setminus B_R(0))}.
	\]
\end{lemma}
\begin{proof}
It is sufficient to repeat the proof of \cite[Lemma 6.2]{HS}, by using this time Lemma \ref{lm:pratignano}, in place of \cite[equation (6.1)]{HS}. We leave the details to the reader. 
\end{proof}
We are ready for the main result of this section. This is an extension of \cite[Theorem 6.1]{HS}.
\begin{theorem}
\label{thm:maximize Holder ratio}
	Let $N\ge 1$, $0<s\le 1$ and $1<p<\infty$ be such that $s\,p>N$. For every $\varphi\in \mathcal{W}^{s,p}(\mathbb{R}^N)$ there exist two distinct points $x_0,\,y_0\in \mathbb{R}^N$ such that
	\[
	[\varphi]_{C^{0,\alpha_{s,p}}(\mathbb{R}^N)}=\frac{|\varphi(x_0)-\varphi(y_0)|}{|x_0-y_0|^{\alpha_{s,p}}}. 
\]
\end{theorem}
\begin{proof}
The case $s=1$ is contained in \cite{HS}, thus we consider $0<s<1$.
We can suppose that $u$ is not constant, otherwise the statement is straightforward.
	For every $n\in\mathbb{N}$, we take $x_n,y_n\in\mathbb{R}^N$ a pair of points such that 
	\begin{equation}
	\label{distanti}
	[u]_{C^{0,\alpha_{s,p}}(\mathbb{R}^N)}-\frac{1}{n+1}\le \frac{|u(x_n)-u(y_n)|}{|x_n-y_n|^{\alpha_{s,p}}}.
	\end{equation}
The goal of the proof is proving that	
	\begin{equation}
	\label{nonsazzera}
		\liminf_{n\to\infty}|x_n-y_n|>0,
	\end{equation}
and
\begin{equation}\label{max sequnce bdd}
		\sup_{n\in\mathbb{N}}(|x_n|+|y_n|)<+\infty.
	\end{equation}
Indeed, from \eqref{max sequnce bdd}, we would get that the maximizing sequence $\{(x_n,y_n)\}_{n\in\mathbb{N}}$ is precompact in $\mathbb{R}^N\times\mathbb{R}^N$. Thus, up to pass to a subsequence, we get that there exists $(\overline{x},\overline{y})\in\mathbb{R}^N\times \mathbb{R}^N$ such that
\[
\lim_{n\to\infty} x_n=\overline{x} \qquad \mbox{ and }\qquad \lim_{n\to\infty} y_n=\overline{y}.
\]
Moreover, by \eqref{nonsazzera}, we would have $\overline{x}\not=\overline{y}$. 
By using this information in \eqref{distanti}, we would obtain
\[
[u]_{C^{0,\alpha_{s,p}}(\mathbb{R}^N)}= \lim_{n\to\infty}\frac{|u(x_n)-u(y_n)|}{|x_n-y_n|^{\alpha_{s,p}}}=\frac{|u(\overline{x})-u(\overline{y})|}{|\overline{x}-\overline{y}|^{\alpha_{s,p}}}.
\]
The property \eqref{nonsazzera} can be proved exactly as in \cite{HS}: it is sufficient to use Lemma \ref{lm:pratignano}, in place of \cite[equation (6.1)]{HS}.
\par
As for \eqref{max sequnce bdd}, the proof is quite lengthy, but here as well it is sufficient to reproduce almost verbatim the arguments in \cite{HS}. We just list the modifications which are needed to cope with the fractional case: at first, since $u\in\mathcal{W}^{s,p}(\mathbb{R}^N)$, the function
\[
G_u(x):=\left(\int_{\mathbb{R}^N}\frac{|u(x)-u(y)|^p}{|x-y|^{N+s\,p}}\, dy\right)^\frac{1}{p},\qquad \mbox{ for every }x\in\mathbb{R}^N,
\]
belongs to $L^p(\mathbb{R}^N)$. Thus, for every $\varepsilon>0$ there exists $R_\varepsilon>0$ such that 
\[
\left(\int_{\mathbb{R}^N\setminus B_{R_\varepsilon}(0)}|G_u(x)|^p\,dx\right)^{\frac{1}{p}}< \varepsilon.
\]
Then in \cite{HS} the two cases $N=1$ and $N\ge 2$ are discussed separately: for dimension $N=1$, one can repeat the proof of \cite{HS}, by using in place of the estimate
\[
|u(x)-u(y)|\le |x-y|^\frac{1}{p}\,\left(\int_x^y |u'|^p\,dt\right)^\frac{1}{p}, \qquad \mbox{ for every } x,y\in\mathbb{R}\ \mbox{ with } x<y,
\]
its fractional counterpart \eqref{onedimensional}.
\par
For dimension $N\ge 2$, the only further modification needed with respect to \cite{HS} is the use of Lemma \ref{lm:morrey_firoj}, in place of \cite[Lemma 6.2]{HS}.
\end{proof}

\section{Existence of extremals}
\label{sec:6}

We still denote by $\mathfrak{m}_{s,p}(\mathbb{R}^N)$ the sharp constant defined in \eqref{morreysharp}.
\begin{theorem}
\label{thm:existence}
Let $N\ge 1$, $0<s\le 1$ and $1<p<\infty$ be such that $s\,p>N$. Then we have 
\[
\mathfrak{m}_{s,p}(\mathbb{R}^N)=\inf_{\varphi\in \mathcal{W}^{s,p}(\mathbb{R}^N)} \Big\{[\varphi]_{W^{s,p}(\mathbb{R}^N)}^p\, : \, [\varphi]_{C^{0,\alpha_{s,p}}(\mathbb{R}^N)}=1\Big\}.
\]
Moreover, the infimum on the right-hand side is attained.
\end{theorem}
\begin{proof}
Since  $C^\infty_0(\mathbb{R}^N)\subseteq \mathcal{W}^{s,p}(\mathbb{R}^N)$, we immediately obtain
\[
\mathfrak{m}_{s,p}(\mathbb{R}^N)\geq \inf_{\varphi\in \mathcal{W}^{s,p}(\mathbb{R}^N)} \Big\{[\varphi]_{W^{s,p}(\mathbb{R}^N)}^p\, : \, [\varphi]_{C^{0,\alpha_{s,p}}(\mathbb{R}^N)}=1\Big\}.
\]
For the reverse inequality, we observe at first that
\[
\begin{split}
\inf_{\varphi\in \mathcal{W}^{s,p}(\mathbb{R}^N)}& \Big\{[\varphi]_{W^{s,p}(\mathbb{R}^N)}^p\, : \, [\varphi]_{C^{0,\alpha_{s,p}}(\mathbb{R}^N)}=1\Big\}\\
&=\inf_{\varphi\in \mathcal{W}^{s,p}(\mathbb{R}^N)} \Big\{[\varphi]_{W^{s,p}(\mathbb{R}^N)}^p\, : \, [\varphi]_{C^{0,\alpha_{s,p}}(\mathbb{R}^N)}=1,\, \varphi(0)=0\Big\},
\end{split}
\]
since both seminorms are invariant with respect to the addition of constants. We then use that 
\[
\mathscr{D}^{s,p}_0(\mathbb{R}^N\setminus\{0\})=\Big\{\varphi\in \mathcal{W}^{s,p}(\mathbb{R}^N)\, :\, \varphi(0)=0\Big\},
\]
thanks to Proposition \ref{prop:homo}. 
In particular, for every $\varphi\in\mathcal{W}^{s,p}(\mathbb{R}^N)$ such that 
\[
\varphi(0)=0\qquad \mbox{ and }\qquad [\varphi]_{C^{0,\alpha_{s,p}}(\mathbb{R}^N)}=1,
\] 
there exists a sequence $\{\varphi_n\}_{n\in\mathbb{N}}\subseteq C^\infty_0(\mathbb{R}^N\setminus\{0\})$ such that
\begin{equation}
\label{approssima}
\lim_{n\to\infty} [\varphi_n-\varphi]_{W^{s,p}(\mathbb{R}^N)}=0.
\end{equation}
By applying Lemma \ref{lm:pratignano} to each function $\varphi_n-\varphi$, we have 
\[
\frac{|\varphi_n(x)-\varphi(x)-(\varphi_n(y)-\varphi(y))|}{|x-y|^{s-\frac{N}{p}}}\le  C\,[\varphi_n-\varphi]_{W^{s,p}(B_\frac{\delta}{2}(\frac{x+y}{2}))}\le C\,[\varphi_n-\varphi]_{W^{s,p}(\mathbb{R}^N)},
\]
which holds for every pair of distinct points $x,y\in\mathbb{R}^N$. By taking the supremum over $x$ and $y$ and using \eqref{approssima}, we get
\[
\lim_{n\to\infty} [\varphi_n-\varphi]_{C^{0,\alpha_{s,p}}(\mathbb{R}^N)}=0 \qquad \mbox{ and thus }\qquad \lim_{n\to\infty} [\varphi_n]_{C^{0,\alpha_{s,p}}(\mathbb{R}^N)}=[\varphi]_{C^{0,\alpha_{s,p}}(\mathbb{R}^N)}=1,
\]
as well. By using $\varphi_n/[\varphi_n]_{C^{0,\alpha_{s,p}}(\mathbb{R}^N)}^p$ in the definition of $\mathfrak{m}_{s,p}(\mathbb{R}^N)$, we thus get
\[
\mathfrak{m}_{s,p}(\mathbb{R}^N)\le \lim_{n\to\infty} \frac{[\varphi_n]_{W^{s,p}(\mathbb{R}^N)}^p}{[\varphi_n]_{C^{0,\alpha_{s,p}}(\mathbb{R}^N)}^p}=[\varphi]_{W^{s,p}(\mathbb{R}^N)}^p.
\]
Finally, by taking the infimum over the admissible $\varphi$ on the right-hand side, we obtain that
\[ 
\mathfrak{m}_{s,p}(\mathbb{R}^N) \le\inf_{\varphi\in \mathcal{W}^{s,p}(\mathbb{R}^N)} \Big\{[\varphi]_{W^{s,p}(\mathbb{R}^N)}^p\, : \, [\varphi]_{C^{0,\alpha_{s,p}}(\mathbb{R}^N)}=1,\, \varphi(0)=0\Big\},
\]
as well.
\vskip.2cm\noindent
We now show existence of a minimizer. For $s=1$ this is contained in \cite[Lemma 2.1]{HS}, for the case $0<s<1$ we can reproduce almost verbatim the same argument. Let $\{u_k\}_{k\in\mathbb{N}}\subseteq \mathcal{W}^{s,p}(\mathbb{R}^N)$ be a minimizing sequence for our problem.
For each $k\in\mathbb N$ we choose a pair of distinct points $x_k,\,y_k\in\mathbb{R}^N$ such that 
\begin{equation}
\label{sequence_prop}
    1=[u_k]_{C^{0,\alpha_{s,p}}(\mathbb{R}^N)}=\frac{u_k(x_k)-u_k(y_k)}{|x_k-y_k|^{\alpha_{s,p}}}.
\end{equation}
Such a pair of points exists by Theorem \ref{thm:maximize Holder ratio}. Without loss of generality, we assumed that $u_k(x_k)>u_k(y_k)$.
We also take $O_k$ an orthogonal transformation satisfying
\[
O_k(\mathbf{e}_N)=\frac{y_k-x_k}{|y_k-x_k|},
\]
where $\mathbf{e}_N=(0,\cdots,0,1)\in\mathbb{R}^N.$ For $k\in\mathbb N$ and $z\in\mathbb{R}^N$, we then define the new sequence
\begin{equation*}
U_k(z):=\frac{\Big(u_k\left(|x_k-y_k|\,O_k(z)+x_k\right)-u_k(x_k)\Big)}{|x_k-y_k|^{\alpha_{s,p}}}.
\end{equation*}
Then, by the invariances of both the H\"{o}lder and Sobolev-Slobodecki\u{\i} seminorms, we get
\[
 [U_k]_{C^{0,\alpha_{s,p}}(\mathbb{R}^N)}=1,\qquad \mbox{ for every } k\in\mathbb{N},
\]
and 
\[
\lim_{k\to\infty}[U_k]^p_{W^{s,p}(\mathbb{R}^N)}=\mathfrak{m}_{s,p}(\mathbb{R}^N),
\]
i.e. this is still a minimizing sequence. Moreover, by construction, we have
$U_k(0)=0$ and by \eqref{sequence_prop}
\[
U_k(\mathbf{e}_N)=\frac{u_k(x_k)-u_k(y_k)}{|x_k-y_k|^{\alpha_{s,p}}}=1.
\]
By applying the Arzel\`{a}-Ascoli theorem and a standard diagonal argument on an increasing sequence of centered balls, we get a subsequence $\{U_{k_n}\}_{n\in\mathbb{N}}$ which converges to a continuous limit function $U$, uniformly on compact sets.
Hence, by the properties of $U_{k}$ and the local uniform convergence, we obtain
\[
U(0)=0,\qquad U(\mathbf{e}_N)=1\qquad \mbox{ and }\qquad [U]_{C^{0,\alpha_{s,p}}(\mathbb{R}^N)}=1.
\]
Moreover, by Fatou's lemma, we have 
\[
[U]^p_{W^{s,p}(\mathbb{R}^N)}\leq\liminf_{k\to\infty}[U_k]^p_{W^{s,p}(\mathbb{R}^N)}=\mathfrak{m}_{s,p}(\mathbb{R}^N).
\]
Thus, $U\in\mathcal{W}^{s,p}(\mathbb{R}^N)$ is a minimizer for our problem. This completes the proof of the theorem.
\end{proof}

\begin{definition}
We will say that $u\in \mathcal{W}^{s,p}(\mathbb{R}^N)$ is an {\it extremal for $\mathfrak{m}_{s,p}(\mathbb{R}^N)$} if $u$ is non-constant and
\[
[u]^p_{W^{s,p}(\mathbb{R}^N)}=\mathfrak{m}_{s,p}(\mathbb{R}^N)\,[u]^p_{C^{0,\alpha_{s,p}}(\mathbb{R}^N)}.
\]
\end{definition}

The previous result entails that the elements of the space $\widetilde{W}^{s,p}_0(\Omega)$ enjoy a Morrey inequality, as well. We explicitly state the relevant result, since this will be needed in the sequel.
\begin{proposition}
\label{prop:embed_unif}
Let $N\ge 1$, $0<s<1$  and $1<p<\infty$ be such that $s\,p>N$. Let $\Omega\subseteq\mathbb{R}^N$ be an open set. Then we have the continuous embedding
\[
\widetilde{W}^{s,p}_0(\Omega)\hookrightarrow C^{0,\alpha_{s,p}}(\overline\Omega)\cap C_0(\overline\Omega).
\]
More precisely, there exists a constant $C=C(N,s,p)>0$ such that
\[
\|\varphi\|^p_{L^\infty(\Omega)}\le C\,\left(\|\varphi\|_{L^p(\Omega)}\right)^\frac{N}{s}\,\left([\varphi]_{W^{s,p}(\mathbb{R}^N)}\right)^{p-\frac{N}{s}},\qquad \mbox{ for every } \varphi\in \widetilde{W}^{s,p}_0(\Omega),
\]
and
\[
[\varphi]^p_{C^{0,\alpha_{s,p}}(\overline\Omega)}\le \frac{1}{\mathfrak{m}_{s,p}(\mathbb{R}^N)}\,[\varphi]_{W^{s,p}(\mathbb{R}^N)}^p,\qquad \mbox{ for every } \varphi\in \widetilde{W}^{s,p}_0(\Omega).
\]
Finally, for the constant $\lambda_{p,\infty}^s(\Omega)$ defined by \eqref{lambda_endpoint}, we have
\[
\lambda_{p,\infty}^s(\Omega)=\inf_{\varphi\in \widetilde{W}^{s,p}_0(\Omega)} \Big\{[\varphi]^p_{W^{s,p}(\mathbb{R}^N)}\, :\, \|\varphi\|_{L^\infty(\Omega)}=1\Big\}.
\]
\end{proposition}
\begin{proof}
We prove at first that $\widetilde W^{s,p}_0(\Omega)\subseteq C_0(\overline\Omega)$. Let $\varphi\in \widetilde{W}^{s,p}_0(\Omega)$, by definition there exists a sequence $\{\varphi_n\}_{n\in\mathbb{N}}\subseteq C^\infty_0(\Omega)$ such that 
\[
\lim_{n\to\infty} \|\varphi_n-\varphi\|_{W^{s,p}(\mathbb{R}^N)}=0.
\] 
In particular, $\{\varphi_n\}_{n\in\mathbb{N}}$ is a Cauchy sequence in $W^{s,p}(\mathbb{R}^N)$. By \cite[Lemma 2.2]{BPZ} and Morrey's inequality, we have for every $n,m\in\mathbb{N}$
\[
\|\varphi_n-\varphi_m\|_{L^\infty(\Omega)}\le C\,\left(\|\varphi_n-\varphi_m\|_{L^p(\Omega)}\right)^\frac{N}{s\,p}\,\left([\varphi_n-\varphi_m]_{C^{0,\alpha_{s,p}}(\mathbb{R}^N)}\right)^{1-\frac{N}{s\,p}},
\]
and
\[
[\varphi_n-\varphi_m]_{C^{0,\alpha_{s,p}}(\mathbb{R}^N)}\le \Big(\mathfrak{m}_{s,p}(\mathbb{R}^N)\Big)^{-\frac{1}{p}}\,[\varphi_n-\varphi_m]_{W^{s,p}(\mathbb{R}^N)}.
\]
Here the constant $C$ depends on $N,p$ and $s$, only\footnote{We remark that it stays bounded as $s\nearrow 1$, as well as $s\searrow N/p$ (see \cite[Remark 2.3]{BPZ}).}.   
Thus, $\{\varphi_n\}_{n\in\mathbb{N}}$ is a Cauchy sequence with respect to the sup norm, as well. Thus it converges uniformly to a limit function in $C_0(\Omega)$, which must coincide with $\varphi$. More precisely, it is a Cauchy sequence in the Banach space $C^{0,\alpha_{s,p}}(\overline\Omega)$, as well. Thus, we actually have $\varphi\in C^{0,\alpha_{s,p}}(\overline\Omega)\cap C_0(\overline\Omega)$.
\par
In particular, we have that $\widetilde{W}^{s,p}_0(\Omega)\subseteq \mathcal{W}^{s,p}(\mathbb{R}^N)$, thus we get the validity of Morrey's inequality for $\widetilde{W}^{s,p}_0(\Omega)$.
The validity of the claimed $L^\infty$ estimate follows in a standard way, by taking the limit as $n$ goes to $\infty$ in
\[
\|\varphi_n\|_{L^\infty(\Omega)}\le C\,\left(\|\varphi_n\|_{L^p(\Omega)}\right)^\frac{N}{s\,p}\,\left([\varphi_n]_{W^{s,p}(\mathbb{R}^N)}\right)^{1-\frac{N}{s\,p}}.
\] 
This estimate can be obtained as above, by combining \cite[Lemma 2.2]{BPZ} and Morrey's inequality for $C^\infty_0(\Omega)$ functions.
\par
As for the last statement, we observe at first that
\[
\lambda_{p,\infty}^s(\Omega)\ge \inf_{\varphi\in \widetilde{W}^{s,p}_0(\Omega)} \Big\{[\varphi]^p_{W^{s,p}(\mathbb{R}^N)}\, :\, \|\varphi\|_{L^\infty(\Omega)}=1\Big\}.
\]
In order to prove the reverse inequality, we take $\varphi\in \widetilde{W}^{s,p}_0(\Omega)$ with unit $L^\infty$ norm and take a sequence $\{\varphi_n\}_{n\in\mathbb{N}}\subseteq C^\infty_0(\Omega)$, converging to $\varphi$ both in $W^{s,p}(\mathbb{R}^N)$ and uniformly. Such a sequence exists, thanks to the first part of the proof. Thus we get
\[
\lambda_{p,\infty}^s(\Omega)\le \lim_{n\to\infty}\frac{[\varphi_n]^p_{W^{s,p}(\mathbb{R}^N)}}{\|\varphi_n\|^p_{L^\infty(\Omega)}}=[\varphi]^p_{W^{s,p}(\mathbb{R}^N)}.
\]
By arbitrariness of $\varphi$, we conclude.
\end{proof}
The use of the space $\mathcal{W}^{s,p}(\mathbb{R}^N)$ may look a bit akward. We now show that this actually coincides with a more standard definition of homogeneous space. The local case $s=1$ is included for completeness, so to show that our results are perfectly consistent with those of \cite{HS}. This is the content of the following
\begin{proposition}
\label{prop:uguaspazi}
Let $N\ge 1$, $0<s\le 1$ and $1<p<\infty$ be such that $s\,p>N$. Let us define
\[
\mathcal{D}^{s,p}(\mathbb{R}^N)=\Big\{\varphi\in L^1_{\rm loc}(\mathbb{R}^N)\, :\, [\varphi]_{W^{s,p}(\mathbb{R}^N)}<+\infty\Big\}.
\]
Then we have 
\[
\mathcal{D}^{s,p}(\mathbb{R}^N)=\mathcal{W}^{s,p}(\mathbb{R}^N),
\]
in the sense that for every $\varphi\in \mathcal{D}^{s,p}(\mathbb{R}^N)$, there exists $\widetilde{\varphi}\in \mathcal{W}^{s,p}(\mathbb{R}^N)$ such that $u=\widetilde{u}$ almost everywhere.
\par
 In particular, the fractional Morrey inequality holds in the space $\mathcal{D}^{s,p}(\mathbb{R}^N)$, as well, still with sharp constant given by $\mathfrak{m}_{s,p}(\mathbb{R}^N)$.
\end{proposition}
\begin{proof}
We first consider the case $0<s<1$.
Let $\varphi\in \mathcal{D}^{s,p}(\mathbb{R}^N)$ and consider $\varphi_{n,M}=\varphi_M\ast \rho_n$, where $\rho_n$ is the sequence of smoothing kernels as in the proof of Lemma \ref{lm:pratignano}. By construction, we clearly have
\[
\lim_{n\to\infty} \|\varphi_n-\varphi\|_{L^1(B_R(0))},\qquad \mbox{ for every }R>0,
\]
and thus 
\[
\lim_{n\to\infty} \varphi_n(x)=\varphi(x),\qquad \mbox{ for a.\,e. }x\in\mathbb{R}^N,
\]
possibly up to pass to a subsequence.
We also know that (see \cite[Lemma A.1]{BGCV})
\[
\lim_{n\to\infty} [\varphi_n-\varphi]_{W^{s,p}(\mathbb{R}^N)}=0,
\]
and $\{\varphi_n\}_{n\in\mathbb{N}}\subseteq C^\infty(\mathbb{R}^N)$.
By applying Lemma \ref{lm:pratignano} to each $\varphi_n$ and thanks to the arbitrariness of $R>0$, we get
\[
\frac{|\varphi_n(x)-\varphi_n(y)|}{|x-y|^{\alpha_{s,p}}}\le  4^\frac{4\,N+1}{p}\,  \left(\frac{2^{N+2\,p}}{\omega_N\,\Lambda_{s,p}(B_1(0))}\right)^\frac{1}{p}\,[\varphi_n]_{W^{s,p}(\mathbb{R}^N)},\quad \mbox{ for } x\not=y.
\]
By taking the limit as $n$ goes to $\infty$ and using the previous informations, we get that 
\[
\frac{|\varphi(x)-\varphi(y)|}{|x-y|^{\alpha_{s,p}}}\le C,\qquad \mbox{ for a.\,e.  }x,y\in\mathbb{R}^N,
\]
i.e. $\varphi$ coincides almost everywhere with a function in $\dot C^{0,\alpha_{s,p}}(\mathbb{R}^N)$. This gives the desired conclusion in the fractional case.
\par
In the local case $s=1$, we proceed in exactly the same way, by using the following Morrey--type estimate (see \cite[equation (6.1)]{HS})
\[
\frac{|\varphi_n(x)-\varphi_n(y)|}{|x-y|^{\alpha_{1,p}}}\le C_{N,p}\,\int_{B_\frac{\delta}{2}(\frac{x+y}{2})}|\nabla \varphi_n|^p\,dx,\qquad \mbox{ with } \delta=|x-y|>0,
\]
in place of Lemma \ref{lm:pratignano}.
\end{proof}

\begin{remark}
\label{rem:HS}
In the local case $s=1$, we can easily repeat the density argument of the first part of the proof of Theorem \ref{thm:existence} and obtain
\[
\mathfrak{m}_{1,p}(\mathbb{R}^N)=\inf_{\varphi\in \mathcal{W}^{1,p}(\mathbb{R}^N)} \Big\{[\varphi]_{W^{1,p}(\mathbb{R}^N)}^p\, : \, [\varphi]_{C^{0,\alpha_{1,p}}(\mathbb{R}^N)}=1\Big\},
\]
here as well. By using the previous result and \cite[Lemma 2.1]{HS}, we then obtain that this infimum is attained on $\mathcal{W}^{1,p}(\mathbb{R}^N)=D^{1,p}(\mathbb{R}^N)$, the latter being exactly the space used in \cite{HS}.
\end{remark}

\section{Some properties of extremals}
\label{sec:7}

\subsection{Basic properties}

\begin{lemma}
\label{lemma:propextrem}
For every $x_0\not=y_0\in\mathbb{R}^N$ and $a\not=b\in\mathbb{R}$, there exists an extremal $u\in \mathcal{W}^{s,p}(\mathbb{R}^N)$ such that
\[
u(x_0)=a,\qquad u(y_0)=b\qquad \mbox{ and }\qquad [u]_{C^{0,\alpha_{s,p}}(\mathbb{R}^N)}=\frac{|u(x_0)-u(y_0)|}{|x_0-y_0|^{\alpha_{s,p}}}.
\] 
\end{lemma}

\begin{proof}
From Theorem \ref{thm:existence}, we already have existence of an extremal $U\in\mathcal{W}^{s,p}(\mathbb{R}^N)$. Moreover, by Theorem \ref{thm:maximize Holder ratio}, there exists two points $\overline{x},\overline{y}$ such that
\[
[U]_{C^{0,\alpha_{s,p}}(\mathbb{R}^N)}=\frac{|U(\overline x)-U(\overline y)|}{|\overline x-\overline y|^{\alpha_{s,p}}}.
\]
Let $O$ be an orthogonal transformation of $\mathbb{R}^N$ such that 
\[
 O\left(\frac{y_0-x_0}{|y_0-x_0|}\right)=\frac{\overline{y}-\overline{x}}{|\overline{y}-\overline{x}|},
\] 
and define
\[
u(x)=A\,U\left(|\overline y-\overline x|\,O\left(\frac{x-x_0}{|y_0-x_0|}\right)+\overline x\right)+B,\qquad \mbox{ for every }x\in\mathbb{R}^N,
\]
where
\[
A=\frac{b-a}{{U(\overline y)-U(\overline x)}}\qquad \mbox{ and }\qquad B=\frac{a\,U(\overline y)-b\,U(\overline x)}{U(\overline y)-U(\overline x)}.
\]
Then, it is easy to observe that this function has the desired properties.
\end{proof}
Every extremal satisfies the relevant Euler-Lagrange equation. More precisely, if we recall the notation
\[
J_p(t)=|t|^{p-2}\,t,\qquad \mbox{ for every } t\in\mathbb{R},
\]
we have
\begin{proposition}
\label{prop:PDE}
Let $u\in \mathcal{W}^{s,p}(\mathbb{R}^N)$ be an extremal for $\mathfrak{m}_{s,p}(\mathbb{R}^N)$. If $x_0,y_0$ are two distinct points such that
\begin{equation}
\label{seminorm_equal}
[u]_{C^{0,\alpha_{s,p}}(\mathbb{R}^N)}=\frac{|u(x_0)-u(y_0)|}{|x_0-y_0|^{\alpha_{s,p}}},
\end{equation}
then $u$ is a weak solution of
\[
(-\Delta_p)^s u=\mathfrak{m}_{s,p}(\mathbb{R}^N)\,\left[\frac{J_p(u(x_0)-u(y_0))}{|x_0-y_0|^{s\,p-N}}\right]\,(\delta_{x_0}-\delta_{y_0}),\qquad \text{ in }\mathbb{R}^N.
\]
In other words, for every $\varphi\in \mathcal{W}^{s,p}(\mathbb{R}^N)$ we have 
\[
\iint_{\mathbb{R}^N\times\mathbb{R}^N} \frac{J_p(u(x)-u(y))\,(\varphi(x)-\varphi(y))}{|x-y|^{N+s\,p}}\,dx\,dy= \mathfrak{m}_{s,p}(\mathbb{R}^N)\,\frac{J_p(u(x_0)-u(y_0))\,\left(\varphi(x_0)-\varphi(y_0)\right)}{|x_0-y_0|^{s\,p-N}}.
\]

\end{proposition}

\begin{proof}
By Theorem \ref{thm:maximize Holder ratio} there exists a pair of distinct points $x_0,y_0$ such that \eqref{seminorm_equal} holds.
Thus, we get
\begin{equation}\label{extremal pde solution eq1}
\mathfrak{m}_{s,p}(\mathbb{R}^N)\, [u]^p_{C^{0,\alpha_{s,p}}(\mathbb{R}^N)}=\mathfrak{m}_{s,p}(\mathbb{R}^N)\,\frac{|u(x_0)-u(y_0)|^p}{|x_0-y_0|^{s\,p-N}}=[u]^p_{W^{s,p}(\mathbb{R}^N)}.
\end{equation}
For $t>0$, we consider the function
 $$
 v(x)=u(x)+t\,\varphi(x),\qquad \text{ where }\varphi\in \mathcal{W}^{s,p}(\mathbb{R}^N),
 $$
 which still belongs to $\mathcal{W}^{s,p}(\mathbb{R}^N)$. Therefore, by applying the fractional Morrey inequality for $v$ and using the convexity of the function $t\mapsto |t|^p$, we obtain
\[
\begin{split}
[u+t\,\varphi]_{W^{s,p}(\mathbb{R}^N)}^p&\geq\mathfrak{m}_{s,p}(\mathbb{R}^N)\,
\frac{|\left(u(x_0)-u(y_0)\right)+t\,\left(\varphi(x_0)-\varphi(y_0)\right)|^p}{|x_0-y_0|^{s\,p-N}}
\\
&\geq   \mathfrak{m}_{s,p}(\mathbb{R}^N)\,\frac{|u(x_0)-u(y_0)|^p}{|x_0-y_0|^{s\,p-N}}+\mathfrak{m}_{s,p}(\mathbb{R}^N)\,p\,t\,\frac{J_p(u(x_0)-u(y_0))\left(\varphi(x_0)-\varphi(y_0)\right)}{|x_0-y_0|^{s\,p-N}}.
\end{split}
\]
By combining this estimate with \eqref{extremal pde solution eq1}, we get for $t>0$
\[
\begin{split}
\frac{[u+t\,\varphi]^p_{W^{s,p}(\mathbb{R}^N)}-[u]^p_{W^{s,p}(\mathbb{R}^N)}}{t}
&\geq \mathfrak{m}_{s,p}(\mathbb{R}^N)\,p\,\frac{J_p(u(x_0)-u(y_0))\,\left(\varphi(x_0)-\varphi(y_0)\right)}{|x_0-y_0|^{s\,p-N}}.
\end{split}
\]
By taking the limit as $t$ goes to $0$ and recalling that $(-\Delta_p)^s$ in weak form coincides with first variation of the functional
\[
\varphi\mapsto \frac{1}{p}\,[\varphi]_{W^{s,p}(\mathbb{R}^N)}^p,
\]
we get that $u$ satisfies
\[
\iint_{\mathbb{R}^N\times\mathbb{R}^N} \frac{J_p(u(x)-u(y))\,(\varphi(x)-\varphi(y))}{|x-y|^{N+s\,p}}\,dx\,dy\ge \mathfrak{m}_{s,p}(\mathbb{R}^N)\,\frac{J_p(u(x_0)-u(y_0))\,\left(\varphi(x_0)-\varphi(y_0)\right)}{|x_0-y_0|^{s\,p-N}},
\]
for every $\varphi\in \mathcal{W}^{s,p}(\mathbb{R}^N)$. By using this inequality with $-\varphi$ in place of $\varphi$, we get the desired conclusion.
\end{proof}
\begin{remark}
From the previous result, we obtain in particular that if $u\in \mathcal{W}^{s,p}(\mathbb{R}^N)$ is an extremal, then the two distinct points $x_0,y_0$ obtained in Theorem \ref{thm:maximize Holder ratio} are actually unique. This easily follows from the Euler-Lagrange equation \eqref{prop:PDE}.
\end{remark}

\begin{theorem}
\label{thm:extremal_equiv}
Let $x_0,\,y_0\in\mathbb{R}^N$ be two distinct points. Then the following facts are equivalent: 
\begin{itemize}
\item[(i)] $u\in\mathcal{W}^{s,p}(\mathbb{R}^N)$ is an extremal for $\mathfrak{m}_{s,p}(\mathbb{R}^N)$ and
\[
[u]_{C^{0,\alpha_{s,p}}(\mathbb{R}^N)}=\frac{|u(x_0)-u(y_0)|}{|x_0-y_0|^{\alpha_{s,p}}};
\]
\vskip.2cm
\item[(ii)] $u\in \mathcal{W}^{s,p}(\mathbb{R}^N)$ weakly solves the following equation
\begin{equation*}
(-\Delta_p)^s u=c\,(\delta_{x_0}-\delta_{y_0})\text{ in }\mathbb{R}^N,
\end{equation*}
for some constant $c\not=0$;
\vskip.2cm
\item[(iii)] $u\in\mathcal{W}^{s,p}(\mathbb{R}^N)$ is non-constant and is the unique minimizer of the problem
\[
\inf_{\varphi\in \mathcal{W}^{s,p}(\mathbb{R}^N)}\Big\{[\varphi]^p_{W^{s,p}(\mathbb{R}^N)}\, :\, \varphi(x_0)=u(x_0)\, \mbox{ and } \varphi(y_0)=u(y_0)\Big\}.
\]
\end{itemize}
\end{theorem}  
\begin{proof}
The implication (i) $\Longrightarrow$ (ii) follows from Proposition \ref{prop:PDE}. 
\vskip.2cm\noindent
Let us suppose that (ii) holds and let $v\in\mathcal{W}^{s,p}(\mathbb{R}^N)$, with $u(x_0)=v(x_0)$ and $u(y_0)=v(y_0)$. Then, by convexity and the equation for $u$ (used with test function $\varphi=v-u$), we obtain
\[ 
\begin{split}
[v]^p_{W^{s,p}(\mathbb{R}^N)}&\ge [u]^p_{W^{s,p}(\mathbb{R}^N)}\\
&+p\,\iint_{\mathbb{R}^N\times\mathbb{R}^N}\frac{J_p(u(x)-u(y))\,\Big((v(x)-u(x))-(v(y)-u(y))\Big)}{|x-y|^{N+s\,p}}\,dx\,dy
\\
&=[u]^p_{W^{s,p}(\mathbb{R}^N)}+c\,p\Big((v(x_0)-u(x_0))-(v(y_0)-u(y_0))\Big)=[u]^p_{W^{s,p}(\mathbb{R}^N)}.
 \end{split}
\]
Therefore, property (iii) holds. 
\vskip.2cm\noindent
Finally, we prove the implication (iii) $\Longrightarrow$ (i). Let $u$ satisfies (iii), by Lemma \ref{lemma:propextrem} we know that there exists an extremal $w\in \mathcal{W}^{s,p}(\mathbb{R}^N)$ such that
\[
w(x_0)=u(x_0),\qquad w(y_0)=u(y_0)\qquad \mbox{ and }\qquad
[w]_{C^{0,\alpha_{s,p}}(\mathbb{R}^N)}=\frac{|w(x_0)-w(y_0)|}{|x_0-y_0|^{s-\frac{N}{p}}}.  
\]
Observe that the last properties imply in particular that  
\begin{equation}
\label{whold}
[w]_{C^{0,\alpha_{s,p}}(\mathbb{R}^N)}=\frac{|u(x_0)-u(y_0)|}{|x_0-y_0|^{s-\frac{N}{p}}}\le [u]_{C^{0,\alpha_{s,p}}(\mathbb{R}^N)}.
\end{equation}
By extremality of $u$, we have
\[
[w]^p_{W^{s,p}(\mathbb{R}^N)}\ge [u]_{W^{s,p}(\mathbb{R}^N)}^p.
\]
On the other hand, by extremality of $w$, we also get
\[
[u]_{W^{s,p}(\mathbb{R}^N)}^p\ge \mathfrak{m}_{s,p}(\mathbb{R}^N)\,[u]^p_{C^{0,\alpha_{s,p}}(\mathbb{R}^N)}=\frac{[w]_{W^{s,p}(\mathbb{R}^N)}^p}{[w]^p_{C^{0,\alpha_{s,p}}(\mathbb{R}^N)}}\,[u]^p_{C^{0,\alpha_{s,p}}(\mathbb{R}^N)}\ge [w]^p_{W^{s,p}(\mathbb{R}^N)}.
\]
In the last inequality, we used \eqref{whold}.
Thus the two seminorms must coincide and the previous chain of inequalities must be a chain of identities. This in particular shows that $u$ is an extremal for $\mathfrak{m}_{s,p}(\mathbb{R}^N)$.
 \end{proof}
\begin{corollary}[Uniqueness]
Suppose $u,v\in\mathcal{W}^{s,p}(\mathbb{R}^N)$ are two extremals such that
\[
[u]_{C^{0,\alpha_{s,p}}(\mathbb{R}^N)}=\frac{|u(x_0)-u(y_0)|}{|x_0-y_0|^{\alpha_{s,p}}}\qquad \mbox{ and }\qquad [v]_{C^{0,\alpha_{s,p}}(\mathbb{R}^N)}=\frac{|v(x_0)-v(y_0)|}{|x_0-y_0|^{\alpha_{s,p}}}.
\]
for some $x_0\not=y_0$.
If there exists a constant $C\not=0$ such that $u(x_0)=C\,v(x_0)$ and $u(y_0)=C\,v(y_0)$, then $u\equiv C\,v$.
\end{corollary}
\begin{proof}
We first observe that $u$ and $v$ are both non-constant, by definition of extremal.
Thanks to Theorem \ref{thm:extremal_equiv}, both $C\,v$ and $u$ solve the same problem
\[
\inf_{\varphi\in \mathcal{W}^{s,p}(\mathbb{R}^N)}\Big\{[\varphi]^p_{W^{s,p}(\mathbb{R}^N)}\, :\, \varphi(x_0)=C\,v(x_0) \mbox{ and } \varphi(y_0)=C\,v(y_0)\Big\}.
\]
Since this problem has a unique solution, we conclude.
\end{proof}

\subsection{Regularity: boundedness}We now prove a global $L^\infty$ bound, by using a curious Moser-type iteration, based on Hardy's inequality. This is the content of the following
\begin{proposition}
\label{prop:limita}
Let $N\ge 1$, $1< p<\infty$ and $0<s<1$ be such that $s\,p>N$. Let $\mu$ be a compactly supported finite measure. Let
$U\in \mathcal{W}^{s,p}(\mathbb{R}^N)$ be a weak solution of 
\[
(-\Delta_p)^s U=\mu,\qquad \mbox{ in }\mathbb{R}^N,
\]
that is 
\begin{equation}
\label{weak_mu}
\iint_{\mathbb{R}^N\times\mathbb{R}^N} \frac{J_p(U(x)-U(y))\,(\varphi(x)-\varphi(y))}{|x-y|^{N+s\,p}}\,dx\,dy=\int_{\mathbb{R}^N}\varphi\,d\mu,\quad \mbox{ for every } \varphi\in \mathcal{W}^{s,p}(\mathbb{R}^N).
\end{equation}
Then $U \in L^\infty(\mathbb{R}^N)$ and for every $z\in\mathbb{R}^N$ we have
\[
\|U-U(z)\|_{L^\infty(\mathbb{R}^N)}=\max_{K} |U-U(z)|,
\]
where $K\subseteq\mathbb{R}^N$ is the support of $\mu$. 
\end{proposition}
\begin{proof}
We fix $z\in\mathbb{R}^N$ and observe that $U-U(z)\in \mathcal{W}^{s,p}(\mathbb{R}^N)$ still solves the same equation. Thus,
without loss of generality, we can suppose that $U(z)=0$. 
We first observe that by proceeding as in \cite[Lemma 3.4]{BP}, we get that $|U|$ is a subsolution of a similar equation, i.e.
\[
\iint_{\mathbb{R}^N\times\mathbb{R}^N} \frac{J_p(|U(x)|-|U(y)|)\,(\varphi(x)-\varphi(y))}{|x-y|^{N+s\,p}}\,dx\,dy\le \int_{\mathbb{R}^N}\varphi\,d|\mu|,
\]
for every $\varphi\in \mathcal{W}^{s,p}(\mathbb{R}^N)$ such that $\varphi\ge 0$.
We now take $M>0$ and $\beta\ge 1$ and insert the test function
\[
\varphi=U_M^\beta:=\left(\min\{|U|,M\}\right)^\beta.
\]
By using \cite[Lemma A.2]{BP}, we get
\[
\frac{\beta\,p^p}{(p+\beta-1)^p}\,\left[U_M^\frac{\beta+p-1}{p}\right]^p_{W^{s,p}(\mathbb{R}^N)}\le \int_{\mathbb{R}^N} U_M^\beta\,d|\mu|=\int_K U_M^\beta\,d|\mu|\le |\mu|(\mathbb{R}^N)\,\max_K (U_M)^\beta.
\]
On the left-hand side, thanks to Proposition \ref{prop:homo} below, we can apply the fractional Hardy inequality for $\mathscr{D}^{s,p}_0(\mathbb{R}^N\setminus\{z\})$. Thus we get
\[
\mathfrak{h}_{s,p}(\mathbb{R}^N\setminus\{0\})\,\frac{\beta\,p^p}{(p+\beta-1)^p}\,\int_{\mathbb{R}^N} \frac{U_M^{\beta+p-1}}{|x-z|^{s\,p}}\,dx\le |\mu|(\mathbb{R}^N)\,\max_K (U_M)^\beta.
\]
We now raise both sided to the power $1/(\beta+p-1)$. This yields
\[
\left(\int_{\mathbb{R}^N} \frac{U_M^{\beta+p-1}}{|x-z|^{s\,p}}\,dx\right)^\frac{1}{\beta+p-1}\le \left(\frac{1}{\mathfrak{h}_{s,p}(\mathbb{R}^N\setminus\{0\})}\,\frac{(p+\beta-1)^p}{\beta\,p^p}\,|\mu|(\mathbb{R}^N)\,\max_K (U_M)^\beta\right)^\frac{1}{\beta+p-1}.
\]
This is valid for every $M>0$ and $\beta\ge 1$. We take the limit as $\beta$ goes to $\infty$, so to obtain
\[
\|U_M\|_{L^\infty(\mathbb{R}^N)}\le \max_K (U_M).
\]
Finally, since $U\in C(\mathbb{R}^N)$, it is bounded on the compact set $K$ and we have 
\[
\max_K (U_M)=\max_K |U|,
\]
for $M$ large enough. Thus, by finally taking the limit as $M$ goes to $\infty$, we conclude.
\end{proof}
By joining Propositions \ref{prop:PDE} and \ref{prop:limita}, we get the following
\begin{theorem}
\label{thm:bound}
Let $u\in \mathcal{W}^{s,p}(\mathbb{R}^N)$ be an extremal for $\mathfrak{m}_{s,p}(\mathbb{R}^N)$. Then we have $u\in L^\infty(\mathbb{R}^N)$ and 
\[
\min\{u(x_0),u(y_0)\}\le u(x)\le \max\{u(x_0),u(y_0)\},\qquad \mbox{ for every }x\in\mathbb{R}^N,
\]
where $x_0,\,y_0\in\mathbb{R}^N$ are the unique two distinct points such that
\[
[u]_{C^{0,\alpha_{s,p}}(\mathbb{R}^N)}=\frac{|u(x_0)-u(y_0)|}{|x_0-y_0|^{\alpha_{s,p}}}.
\]
\end{theorem}
\begin{proof}
Without loss of generality, we can suppose that $u(x_0)>u(y_0)$.
We can apply Proposition \ref{prop:limita} with $K=\{x_0,y_0\}$ and $z=y_0$. This gives for every $x\in\mathbb{R}^N$
\[
u(x)-u(y_0)\le |u(x)-u(y_0)|\le |u(x_0)-u(y_0)|=u(x_0)-u(y_0),
\]
that is $u(x)\le u(x_0)$ for every $x\in\mathbb{R}^N$. Similarly, with $z=x_0$ we have 
\[
u(x)-u(x_0)\ge -|u(x)-u(x_0)|\ge -|u(x_0)-u(y_0)|=u(y_0)-u(x_0),
\]
which shows that $y_0$ is a minimum point for $u$.
\end{proof}

\subsection{Regularity: improved H\"older continuity}

We will show that extremals for $\mathfrak{m}_{s,p}(\mathbb{R}^N)$ are more regular. This is a consequence of the following more general result.
\begin{proposition}
\label{prop:holder_reg}
Let $N\ge 1$, $1<p<\infty$ and $0<s<1$ be such that $s\,p>N$. Let $\mu$ be a compactly supported finite measure. Let
 $U\in \mathcal{W}^{s,p}(\mathbb{R}^N)$ be a weak solution of 
\[
(-\Delta_p)^s U=\mu,\qquad \mbox{ in }\mathbb{R}^N,
\]
i.e. $U$ satisfies \eqref{weak_mu}.
If we indicate by $K$ the support of $\mu$, we have 
\begin{equation}
\label{almostC1}
U\in C^{0,\alpha}_{\rm loc}(\mathbb{R}^N\setminus K),\qquad \mbox{ for every } \alpha_{s,p}\le \alpha<\min\left\{\frac{s\,p}{p-1},\,1\right\}.
\end{equation}
Moreover, we also have 
\begin{equation}
\label{sharp_delta}
U\in C^{0,\delta}_{\rm loc}(\mathbb{R}^N),\qquad \mbox{ for every } \alpha_{s,p}\le \delta<\frac{s\,p-N}{p-1},
\end{equation}
and for every $\alpha_{s,p}\le \delta<(s\,p-N)/(p-1)$ there exists a constant $C=C(N,s,p,\delta)>0$ such that
\[
[U]_{C^{0,\delta}(B_r(0))}\leq \frac{C}{r^\delta}\Bigg [\|U\|_{L^\infty(\mathbb{R}^N)}+r^\frac{s\,p-N}{p-1}\,\Big(|\mu|(\mathbb{R}^N)\Big)^\frac{1}{p-1}\Bigg ],
\]
holds for every $r>0$.
\end{proposition}
\begin{proof}
We will appeal to the regularity results for {\it local weak solutions} taken from\footnote{We observe that the results of \cite{GL} concerns the case $1<p<2$. Thus, they are needed to handle the case $N=1$ only, where it may happen that $1<p<2$ and $s\,p>1=N$. For $N\ge 2$, the condition $s\,p>N$ automatically implies that $p>2$.} \cite{BLS} and \cite{GL}. In order to do this, we need to assure at first that $U$ satisfies the following weighted integrability property
\[
\int_{\mathbb{R}^N} \frac{|U|^{p-1}}{(1+|x|)^{N+s\,p}}\,dx<+\infty.
\]
This follows, for example, from the fact that $U$ is globally bounded, thanks to Proposition \ref{prop:limita}.
\par
The first result \eqref{almostC1} immediately follows from \cite[Theorem 5.2]{BLS} and \cite[Theorem 1.1]{GL}, since $U$ is, in particular, a local weak solution of 
\[
(-\Delta_p)^s U=0,\qquad \mbox{ in }\mathbb{R}^N\setminus K.
\] 
As for \eqref{sharp_delta}, if the measure $\mu$ is absolutely continuous with respect to the $N-$dimensional Lebesgue measure, this is a direct consequence of \cite[Theorem 1.4]{BLS} and \cite[Theorem 1.2]{GL} with $q=1$ there.
When $\mu$ is a general measure, we need to go through a regularization procedure and get the desired result from the a priori estimates of \cite[Theorem 1.4]{BLS} and \cite[Theorem 1.2]{GL}. This is possible because $s\,p>N$, of course. We give the details, for completeness.
\vskip.2cm\noindent
{\bf Step 1: approximation}. We fix the ball $B_R(0)$ and choose $R$ large enough so that $K\Subset B_R(0)$. We now set 
\[
f_n=\mu\ast \rho_n,
\] 
where $\{\rho_n\}_{n\in\mathbb{N}}$ is the usual family of Friedrichs mollifiers. By construction, we have $f_n\in C^\infty_0(\mathbb{R}^N)$ and
\begin{equation}
\label{uniformeL1}
\|f_n\|_{L^1(\mathbb{R}^N)}\le |\mu|(\mathbb{R}^N)<+\infty,\qquad \mbox{ for every } n\in\mathbb{N}.
\end{equation}
Moreover, the sequence $\{f_n\}_{n\in\mathbb{N}}$ is $\ast-$weak converging to $\mu$, as $n$ goes to $\infty$ (see \cite[Theorem 2.2]{AFP}).
\par
We now consider the following nonlocal boundary value problem
\[
\left\{\begin{array}{rcll}
(-\Delta_p)^s u&=&f_n,& \mbox{ in } B_R(0),\\
u&=& U,& \mbox{ in }\mathbb{R}^N\setminus B_R(0).
\end{array}
\right.
\]
This admits a unique weak solution $u_n\in \mathcal{W}^{s,p}(\mathbb{R}^N)$ such that $u_n-U\in \widetilde{W}^{s,p}_0(B_R(0))$, i.e. $u_n$ satisfies
\[
\begin{split}
\iint_{\mathbb{R}^N\times\mathbb{R}^N} &\frac{J_p(u_n(x)-u_n(y))\,(\varphi(x)-\varphi(y))}{|x-y|^{N+s\,p}}\,dx\,dy=\int_{B_R(0)} f_n\,\varphi\,dx,\qquad \mbox{ for every } \varphi\in \widetilde{W}^{s,p}_0(B_R(0)).
\end{split}
\]
Existence of a solution can be established by variational methods. More precisely, $u_n$ is the unique solution of the following problem
\[
\inf_{\varphi\in \mathcal{W}^{s,p}(\mathbb{R}^N)} \left\{\frac{1}{p}\,[\varphi]^p_{W^{s,p}(\mathbb{R}^N)}-\int_{B_R(0)} f_n\,\varphi\,dx\, :\, \varphi-U\in \widetilde{W}^{s,p}_0(B_R(0))\right\}.
\]
We claim that $\{u_n\}_{n\in\mathbb{N}}$ converges to $U$ as $n$ goes to $\infty$, up to a subsequence.
By testing the weak formulation of the equation with $\varphi=u_n$, we get
\[
\begin{split}
[u_n]_{W^{s,p}(\mathbb{R}^N)}^p=\int_{B_R(0)} f_n\,u_n\,dx&\le \|f_n\|_{L^1(B_R(0))}\,\|u_n\|_{L^\infty(B_R(0))}\\
&\le \|f_n\|_{L^1(B_R(0))}\,\left(\|u_n-U\|_{L^\infty(B_R(0))}+\|U\|_{L^\infty(B_R(0))}\right).
\end{split}
\]
Observe that the last term is finite, since functions in $\mathcal{W}^{s,p}(\mathbb{R}^N)$ are locally bounded.
Moreover, since $u_n-U\in \widetilde{W}^{s,p}_0(B_R(0))$, from the definition of $\lambda_{p,\infty}^s(B_R(0))$ and Proposition \ref{prop:embed_unif}, we have 
\[
\begin{split}
\|u_n-U\|_{L^\infty(B_R(0))}&\le \left(\frac{1}{\lambda_{p,\infty}^s(B_R(0))}\right)^\frac{1}{p}\,[u_n-U]_{W^{s,p}(\mathbb{R}^N)}\\
&\le \left(\frac{1}{\lambda_{p,\infty}^s(B_R(0))}\right)^\frac{1}{p}\,\left([u_n]_{W^{s,p}(\mathbb{R}^N)}+[U]_{W^{s,p}(\mathbb{R}^N)}\right).\\
\end{split}
\]
By joining the previous two estimates and keeping \eqref{uniformeL1} into account, it is not difficult to see that there exists a constant $C=C(N,R,p,s,\mu)>0$ such that
\[
[u_n]_{W^{s,p}(\mathbb{R}^N)}^p\le C,\qquad \mbox{ for every }n\in\mathbb{N}.
\]
On the other hand, we can also bound uniformly the $L^p$ norms, by using the boundary datum and the fractional Poincar\'e inequality (see for example \cite[Lemma 2.4]{BLP}), i.\,e.
\[
\begin{split}
\|u_n\|_{L^p(B_R(0))}&\le \|u_n-U\|_{L^p(B_R(0))}+\|U\|_{L^p(B_R(0))}\le C\,R^s\,[u_n-U]_{W^{s,p}(\mathbb{R}^N)}+\|U\|_{L^p(B_R)},
\end{split}
\] 
for some $C=C(N,p,s)>0$ .
The last term in turn is uniformly bounded, thanks to what we said above. 
\par
In conclusion, $\{u_n-U\}_{n\in\mathbb{N}}\subseteq \widetilde{W}^{s,p}_0(B_R(0))$ is uniformly bounded. Thus, thanks to \cite[Theorem 2.7]{BLP}, up to a subsequence it converges weakly in $\widetilde{W}^{s,p}_0(B_R(0))$ and strongly in $L^p(B_R(0))$ to a function $\underline{u}\in \widetilde{W}^{s,p}_0(B_R(0))$. Such a convergence is uniform on $\overline{B_R(0)}$, as well (thanks to Proposition \ref{prop:embed_unif}, for example). 
We claim that 
\begin{equation}
\label{provalo!}
\underline{u}\equiv 0,
\end{equation}
this would give that $u_n-U$ actually converges to $0$, as $n$ goes to $\infty$. In order to prove \eqref{provalo!}, we test the minimality of $u_n$ against $U$. This yields
\begin{equation}
\label{protogamma}
\frac{1}{p}\,[u_n]^p_{W^{s,p}(\mathbb{R}^N)}-\int_{B_R(0)} f_n\,u_n\,dx\le \frac{1}{p}\,[U]^p_{W^{s,p}(\mathbb{R}^N)}-\int_{B_R(0)} f_n\,U\,dx.
\end{equation}
We now observe that 
\[
\lim_{n\to\infty} \int_{B_R(0)} f_n\,U\,dx=\int_{B_R(0)} U\,d\mu,
\]
and 
\[
\begin{split}
\lim_{n\to\infty}\int_{B_R(0)} f_n\,u_n\,dx&=\lim_{n\to\infty}\int_{B_R(0)} f_n\,(u_n-U)\,dx+\lim_{n\to\infty}\int_{B_R(0)} f_n\,U\,dx\\
&=\int_{B_R(0)} \underline{u}\,d\mu+\int_{B_R(0)} U\,d\mu,
\end{split}
\]
where we used the $\ast-$convergence of $f_n$ and the uniform convergence of $u_n-U$. Moreover, by weak convergence we have 
\[
\liminf_{n\to\infty} \frac{1}{p}\,[u_n]^p_{W^{s,p}(\mathbb{R}^N)}\ge \frac{1}{p}\,[\underline{u}+U]^p_{W^{s,p}(\mathbb{R}^N)}.
\]
Thus, by taking the limit as $n$ goes to $\infty$ in \eqref{protogamma}, we get
\begin{equation}
\label{sottominimo}
\frac{1}{p}\,[\underline{u}+U]^p_{W^{s,p}(\mathbb{R}^N)}-\int_{B_R(0)} (\underline{u}+U)\,d\mu\le \frac{1}{p}\,[U]^p_{W^{s,p}(\mathbb{R}^N)}-\int_{B_R(0)} U\,d\mu.
\end{equation}
It is only left to observe that $U$ is the {\it unique} solution of the strictly convex problem
\[
\inf_{\varphi\in \mathcal{W}^{s,p}(\mathbb{R}^N)} \left\{\frac{1}{p}\,[\varphi]^p_{W^{s,p}(\mathbb{R}^N)}-\int_{B_R(0)} \varphi\,d\mu\, :\, \varphi-U\in \widetilde{W}^{s,p}_0(B_R(0))\right\}.
\]
Since $\underline{u}+U$ is still admissible in this problem, from \eqref{sottominimo} we get \eqref{provalo!}.
\vskip.2cm\noindent
{\bf Step 2: uniform H\"older estimate}.
By using that $s\,p>N$, we can apply \cite[Theorem 1.4]{BLS} or \cite[Theorem 1.2]{GL} with $f=f_n$ and $q=1$, so to obtain that 
\[
u_n\in C^{0,\delta}_{\rm loc}(B_R(0)),\qquad \mbox{ for every } \alpha_{s,p}\le \delta<\frac{s\,p-N}{p-1}.
\]
Moreover, from \cite[Theorem 1.4]{BLS} and \cite[Theorem 1.2]{GL} we also have the following a priori estimate: for every $\alpha_{s,p}\le \delta<(s\,p-N)/(p-1)$ and every $r\le R/8$, there exists a constant $C=C(N,s,p,\delta)>0$ such that 
\begin{equation}
\label{stimona}
\begin{split}
[u_n]_{C^{0,\delta}(B_{r/8}(0))}\leq \frac{C}{r^\delta}&\,\Bigg [\|u_n\|_{L^\infty(B_{r}(0))}+\left(r^{s\,p}\,\int_{\mathbb{R}^N\setminus B_{r}(0)}\frac{|u_n|^{p-1}}{|y|^{N+s\,p}}\,  dy\right)^\frac{1}{p-1}\\
&+\left(r^{s\,p-N}\,\|f_n\|_{L^1(B_{r}(0))}\right)^\frac{1}{p-1}\Bigg ].
\end{split}
\end{equation}
For the first term on the right-hand side, we can use the uniform convergence, already obtained in {\bf Step 1}, so to get
\[
\lim_{n\to\infty}\|u_n\|_{L^\infty(B_{r}(0))}=\|U\|_{L^\infty(B_{r}(0))}.
\] 
For the ``tail'' term, by using again the uniform convergence on $\overline{B_R(0)}$ and the fact that $u_n\equiv U$ outside $B_R(0)$,  we have 
\[
\begin{split}
\lim_{n\to\infty}\int_{\mathbb{R}^N\setminus B_{r}(0)}\frac{|u_n|^{p-1}}{|y|^{N+s\,p}}\,  dy&=\lim_{n\to\infty}\int_{B_R(0)\setminus B_{r}(0)}\frac{|u_n|^{p-1}}{|y|^{N+s\,p}}\,  dy+\int_{\mathbb{R}^N\setminus B_{R}(0)}\frac{|U|^{p-1}}{|y|^{N+s\,p}}\,  dy\\
&=\int_{\mathbb{R}^N\setminus B_{r}(0)}\frac{|U|^{p-1}}{|y|^{N+s\,p}}\,  dy.
\end{split}
\]
Finally, we can use \eqref{uniformeL1} to bound the term in \eqref{stimona} containing $f_n$. By proceeding in this way and taking the limit as $n$ goes to $\infty$, we get
\[
\begin{split}
[U]_{C^{0,\delta}(B_{r/8}(0))}\leq \frac{C}{r^\delta}&\,\Bigg [\|U\|_{L^\infty(B_{r}(0))}+\left(r^{s\,p}\,\int_{\mathbb{R}^N\setminus B_{r}(0)}\frac{|U|^{p-1}}{|y|^{N+s\,p}}\,  dy\right)^\frac{1}{p-1}+r^\frac{s\,p-N}{p-1}\,\Big(|\mu|(\mathbb{R}^N)\Big)^\frac{1}{p-1}\Bigg ],
\end{split}
\]
for some $C=C(N,s,p,\delta)>0$.
This shows that $U\in C^{0,\delta}(B_{r/8}(0))$ for every $r\le R/8$ and every $R>0$ large enough. By arbitrariness of $r$ and $R$, we thus get the claimed conclusion. 
\par
Finally, by using that $U\in L^\infty(\mathbb{R}^N)$ thanks Proposition \ref{prop:limita}, we have 
\[
\begin{split}
r^{s\,p}\,\int_{\mathbb{R}^N\setminus B_{r}(0)}\frac{|U|^{p-1}}{|y|^{N+s\,p}}\,  dy&\le \|U\|^{p-1}_{L^\infty(\mathbb{R}^N\setminus B_r(0))}\,r^{s\,p}\,\int_{\mathbb{R}^N\setminus B_{r}(0)}\frac{dy}{|y|^{N+s\,p}}\\
&=\|U\|^{p-1}_{L^\infty(\mathbb{R}^N\setminus B_r(0))}\,\frac{N\,\omega_N}{s\,p}\le \omega_N\,\|U\|^{p-1}_{L^\infty(\mathbb{R}^N\setminus B_r(0))}.
\end{split}
\]
From the previous estimate, we can infer
\[
[U]_{C^{0,\delta}(B_{r/8}(0))}\leq \frac{C}{r^\delta}\Bigg [\|U\|_{L^\infty(\mathbb{R}^N)}+r^\frac{s\,p-N}{p-1}\,\Big(|\mu|(\mathbb{R}^N)\Big)^\frac{1}{p-1}\Bigg ],
\]
possibly for a different constant $C=C(N,s,p,\delta)>0$.
\end{proof}
By joining Propositions \ref{prop:PDE} and \ref{prop:holder_reg}, we immediately get the following
\begin{theorem}
\label{thm:holderreg}
Let $u\in \mathcal{W}^{s,p}(\mathbb{R}^N)$ be an extremal for $\mathfrak{m}_{s,p}(\mathbb{R}^N)$. Then we have
\[
u\in C^{0,\alpha}_{\rm loc}(\mathbb{R}^N\setminus \{x_0,y_0\}),\qquad \mbox{ for every } \alpha_{s,p}\le \alpha<\min\left\{\frac{s\,p}{p-1},\,1\right\},
\]
and also
\[
u\in C^{0,\delta}_{\rm loc}(\mathbb{R}^N),\qquad \mbox{ for every } \alpha_{s,p}\le \delta<\frac{s\,p-N}{p-1}.
\]
Here $x_0,\,y_0\in\mathbb{R}^N$ are the unique two distinct points such that
\[
[u]_{C^{0,\alpha_{s,p}}(\mathbb{R}^N)}=\frac{|u(x_0)-u(y_0)|}{|x_0-y_0|^{\alpha_{s,p}}}.
\]
Moreover, for every $\alpha_{s,p}\le \delta<(s\,p-N)/(p-1)$ there exists a constant $C=C(N,s,p,\delta)>0$ such that
\[
[u]_{C^{0,\delta}(B_r(0))}\leq \frac{C}{r^\delta}\,\left[\max\Big\{|u(x_0)|,\, |u(y_0)|\Big\}+r^\frac{s\,p-N}{p-1}\,\Big(\mathfrak{m}_{s,p}(\mathbb{R}^N)\Big)^\frac{1}{p-1}\,\frac{|u(x_0)-u(y_0)|}{|x_0-y_0|^\frac{s\,p-N}{p-1}}\right],
\]
holds for every $r>0$.
\end{theorem}

\section{Asymptotics for the sharp constant}
\label{sec:8}

\subsection{The case $s\searrow N/p$}
We start with the following lower bound, which is interesting in itself: it relates the sharp Morrey constant with the sharp Hardy one.
\begin{corollary}
\label{coro:mistery}
Let $N\ge 1$, $0<s<1$ and $1<p<\infty$ be such that $s\,p>N$. We have
\begin{equation}
\label{lowerdecay}
\mathfrak{m}_{s,p}(\mathbb{R}^N)\ge s\,C\,\frac{\mathfrak{h}_{s,p}(\mathbb{R}^N\setminus\{0\})}{s\,p-N},
\end{equation}
for a constant $C=C(N,p)>0$.
\end{corollary}
\begin{proof}
From \eqref{lam_cap} and Lemma \ref{lemma:point},
we obtain the following explicit lower bound
\begin{equation}
\label{loweralpha}
\Lambda_{s,p}(B_r(x_0))\ge \frac{N\,C\,s\,\omega_N}{r^{s\,p}}\,\frac{\mathfrak{h}_{s,p}(\mathbb{R}^N\setminus\{0\})}{s\,p-N}.
\end{equation}
By using this estimate with $x_0=0$ and $r=1$, in conjunction with Theorem \ref{thm:morrey}, we get the claimed result.
\end{proof}
The main result of this section dictates the exact decay rate to $0$ of the sharp Morrey constant, as we reach the borderline case $s\,p=N$.
\begin{proposition}
\label{prop:spN}
Let $1\le N<p<\infty$. We have 
\[
0<\liminf_{s\searrow \frac{N}{p}} \frac{\mathfrak{m}_{s,p}(\mathbb{R}^N)}{(s\,p-N)^{p-1}}\le \limsup_{s\searrow \frac{N}{p}} \frac{\mathfrak{m}_{s,p}(\mathbb{R}^N)}{(s\,p-N)^{p-1}}<+\infty.
\]
\end{proposition}
\begin{proof}
The $\liminf$ follows from \eqref{lowerdecay}, by keeping into account Lemma \ref{lemma:fastidio}, as well. 
\par
The $\limsup$ is in principle simpler, it is sufficient to make a clever choice of trial functions. We will use the idea of \cite[Appendix]{BB}, which was concerned with the case $N=1$, $p=2$ and $s\searrow 1/2$. We take the trial function
\[
\zeta(x)=\left(1-|x|^\frac{s\,p-N}{p-1}\right)_+,\qquad \mbox{ for } x\in\mathbb{R}^N,
\]
and observe that $\zeta\in \mathcal{W}^{s,p}(\mathbb{R}^N)$. Indeed, the fact that 
\[
[\zeta]_{W^{s,p}(\mathbb{R}^N)}<+\infty,
\]
follows from Lemma \ref{lm:gloria!}. Moreover, since $\zeta$ identically vanishes outside $B_1(0)$, we have
\[
\begin{split}
[\zeta]_{C^{0,\alpha_{s,p}}(\mathbb{R}^N)}=[\zeta]_{C^{0,\alpha_{s,p}}(\overline{B_1(0)})}=\sup_{x,y\in \overline{B_1(0)}} \frac{\left||x|^\frac{s\,p-N}{p-1}-|y|^\frac{s\,p-N}{p-1}\right|}{|x-y|^{s-\frac{N}{p}}}=1,
\end{split}
\]
thanks to the fact that
\[
\left||x|^\frac{s\,p-N}{p-1}-|y|^\frac{s\,p-N}{p-1}\right|\le \big||x|-|y|\big|^\frac{s\,p-N}{p-1}\qquad \mbox{ and }\qquad |x-y|^{s-\frac{N}{p}}\ge \big||x|-|y|\big|^{s-\frac{N}{p}},
\]
together with the observation that $s-p/N<(sp-N)/(p-1)<1$. Thus, by recalling Theorem \ref{thm:existence}, we get 
\[
\mathfrak{m}_{s,p}(\mathbb{R}^N)\le \frac{[\zeta]^p_{W^{s,p}(\mathbb{R}^N)}}{[\zeta]^p_{C^{0,\alpha_{s,p}}(\mathbb{R}^N)}}=[\zeta]^p_{W^{s,p}(\mathbb{R}^N)}.
\]
If we now apply the estimate of Lemma \ref{lm:gloria!}, we get the desired conclusion.
\end{proof}

\subsection{The case $p\nearrow \infty$}

We first need the following technical result. We still denote by $\Lambda_{s,p}$ the constant defined by \eqref{lambdone}.

\begin{lemma}
\label{lemma:lambda1}
For every $0<s<1\le N$, we have
\[
\lim_{p\to\infty} \Big(\Lambda_{s,p}(B_1(0))\Big)^\frac{1}{p}= 1.
\] 
\end{lemma}
\begin{proof}  
For every $\varepsilon>0$, we set
\[
f_\varepsilon(x)=\sqrt{\varepsilon^2+|x|^2}-\varepsilon,\qquad \mbox{ for } x\in \mathbb{R}^N.
\]
Then, the function
\[
F_{p,\varepsilon}(x)=\frac{f_\varepsilon(x)}{C_{N,p,\varepsilon}},\qquad \mbox{ with }\qquad C_{N,p,\varepsilon}=\left(N\,\omega_N\,\int_0^1 \big(\varepsilon+\varrho^2\big)^\frac{p}{2}\,\varrho^{N-1}\,d\varrho\right)^\frac{1}{p},
\]
is admissible for  $\Lambda_{s,p}(B_1(0))$, Thus, we get
\[ 
\Big(\Lambda_{s,p}(B_1(0))\Big)^{\frac{1}{p}} \leq [F_{p,\varepsilon}]_{W^{s,p}(B_1(0))}=\frac{1}{C_{N,p,\varepsilon}}\,[f_\varepsilon]_{W^{s,p}(B_1(0))}.
\]
Since $f_\varepsilon$ is a Lipschitz continuous function, by applying  \eqref{normegagliardo1}, we obtain that 
\[
\begin{split}
 \limsup_{p \to \infty}  \Big(\Lambda_{s,p}(B_1(0))\Big)^{\frac{1}{p}}&\leq  \limsup_{p\to \infty} \frac{1}{C_{N,p,\varepsilon}}\,[f_\varepsilon]_{W^{s,p}(B_1(0))}\\
 &=  \limsup_{p\to \infty} [f_\varepsilon]_{W^{s,p}(B_1(0))}= [f_\varepsilon]_{C^{0,s}(\overline{B_1(0))}}.
\end{split}
\]
Let us now take two distinct points $x,y\in \overline{B_1(0)}$, without loss of generality we can suppose that $|x|\ge |y|$. Thus, we get\footnote{We use that for every $t\ge \tau\ge 0$, the function
\[
g(\varepsilon)=\sqrt{\varepsilon^2+t^2}-\sqrt{\varepsilon^2+\tau^2},
\]
is monotone non-increasing. In particular, $g(\varepsilon)\le g(0)$.}
\[
|f_\varepsilon(x)-f_\varepsilon(y)|=\sqrt{\varepsilon^2+|x|^2}-\sqrt{\varepsilon^2+|y|^2}\le |x|-|y|=\Big||x|-|y|\Big|.
\]
In particular, this entails that
\[
\frac{|f_\varepsilon(x)-f_\varepsilon(y)|}{|x-y|^s}\le \frac{\Big||x|-|y|\Big|}{\Big||x|-|y|\Big|^s}=\Big||x|-|y|\Big|^{1-s}\le 1.
\]
The previous inequality implies that 
\begin{equation}
\label{eq:limsupmu}  \limsup_{p \to \infty}  \Big(\Lambda_{s,p}(B_1(0))\Big)^{\frac{1}{p}}\leq[f_\varepsilon]_{C^{0,s}(B_1(0)))}\le 1.
\end{equation}
Then, in order to conclude the proof, it is sufficient to show that 
\[
\liminf_{p \to \infty} \left(\Lambda_{s,p}(B_1)\right)^{\frac{1}{p}}\geq 1.
\]
For every  $p>1$, let  $u_{p} \in C^1(\overline{B_1(0)})$ be  such that 
\begin{equation}
\label{succmin} 
[ u_{p}]_{W^{s,p}(B_1(0))} \leq \Big(\Lambda_{s,p}(B_1(0))\Big)^{\frac{1}{p}}+\frac{1}{p},\qquad  \|u_{p}\|_{L^{p}(B_1(0))}=1, \qquad u_p(0)=0.
\end{equation}
We fix $1<q<\infty$ and $0<t<s$ such that $t\,q>N$, by \eqref{succmin} and Lemma \ref{lemma:embeddinglow} we get for every $p>q$
\begin{equation}
\label{eq:limitsemin}
[ u_{p}]_{W^{t,q}(B_1(0))}\le \left(\frac{N\,\omega_N\,|\Omega|}{s-t}\,\frac{p-q}{p\,q}\right)^\frac{p-q}{p\,q}\,\Big(\mathrm{diam}(\Omega)\Big)^{s-t}\, \left(\Big(\Lambda_{s,p}(B_1(0))\Big)^{\frac{1}{p}} +\frac{1}{p} \right).
 \end{equation}
Moreover, by using H\"older's inequality and the normalization on $u_p$, we have that
\begin{equation}
\label{eq:limitnormep}  
\|u_{p}\|_{L^{q}(B_1(0))}\leq (\omega_N)^{\frac 1 q -\frac 1 p}, \qquad \mbox{ for every  } p>q. 
 \end{equation}
By combining  \eqref{eq:limsupmu}, \eqref{eq:limitsemin} and \eqref{eq:limitnormep},  we get  that the family $\{ u_{p}\}_{p\geq q}$ is bounded in $W^{t,q}(B_1(0))$. 
Since $t\,q>N$, by \cite[Theorem 8.2]{DNPV}, we get that  the family $\{u_p\}_{p\geq q}$ is equicontinuous.  
Moreover, since $u_p(0)=0$, this implies that $\{u_p\}_{p\geq q}$ is equibounded, as well. 
By applying  the  Ascoli-Arzel\`a Theorem, we have that there exists an increasingly diverging sequence $\{p_n\}_{n\in\mathbb{N}}\subseteq (q,+\infty)$ and a continuous function $u_{\infty}$ on $\overline{B_1(0)}$, such that $u_{p_n}$ converges uniformly on $B_1(0)$ to $u_\infty$, as $n$ goes to $\infty$. Of course, we still have
\[
u_\infty(0)=0.
\]
In addition, by the uniform convergence, we have
\[
1=\liminf_{p\to \infty}\|u_{p_n}\|_{L^{p_n}(B_1(0))}\leq \limsup_{n\to \infty} \omega_N^\frac{1}{p_n}\,\|u_{p_n}\|_{L^\infty(B_1(0))}=\|u_{\infty}\|_{L^\infty(B_1(0))}.
\]
Thus, there exists $\overline{x}\in \overline{B_1(0)}$, such that
\[
u_\infty(\overline{x})\ge 1.
\] 
We are left with observing that the sequence   $\{u_{p_n}\}_{n\in\mathbb{N}}$ satisfies  the assumptions of Lemma \ref{lemma:gamma-conv} (in view of  \eqref{succmin} and \eqref{eq:limsupmu}). Thus, we actually have $u_{\infty}\in C^{0,s}(\overline{B_1(0)})$, together with the estimate 
\[
[u_{\infty}]_{C^{0,s}(B_1(0))}\leq \liminf_{n\to \infty }[u_{p_n}]_{W^{s,{p_n}}(B_1(0))} =\liminf_{n\to \infty}\left(\Lambda_{s,p_n}(B_1(0))\right)^{\frac{1}{p_n}}+\frac{1}{p_n}.
\]
On the other hand, by using the properties of $u_\infty$, we have
\[ 
1\leq  \frac{|u_{\infty}(\overline x)|}{|\overline x|^s}\leq  [u_{\infty}]_{C^{0,s}(B_1(0))}.
\]
The last two formulas in display, in conjunction with \eqref{eq:limsupmu}, show that
\[
\lim_{n\to \infty}\Big(\Lambda_{s,{p_n}}(B_1(0))\Big)^{\frac{1}{p_n}}=1.
\]
Since the above limit does not depend on the sequence $\{p_n\}_{n\in\mathbb{N}}$, we get the desired conclusion.
\end{proof}

We can now give the result about the asymptotic behaviour of the sharp constant, as $p$ diverges to $\infty$. This generalizes to the fractional case the asymptotic result of \cite[Corollary 4.3]{BPZ}.
\begin{proposition}
\label{prop:limitinfty}
Let $N\ge 1$ and $0<s<1$, then we have
\[
\lim_{p\to\infty} \Big(\mathfrak{m}_{s,p}(\mathbb{R}^N)\Big)^\frac{1}{p}=1.
\]
\end{proposition}
\begin{proof}
From Theorem \ref{thm:morrey}, we know that for every $p>N/s$ and every $T>1$, we have
\[
\left(\omega_N^\frac{1}{p}\,\frac{(T-1)^\frac{N}{p}}{(T-1)^s+T^s}\right)\,\Big(\Lambda_{s,p}(B_1(0))\Big)^\frac{1}{p}\le \Big(\mathfrak{m}_{s,p}(\mathbb{R}^N)\Big)^\frac{1}{p}.
\]
Thus, from Lemma \ref{lemma:lambda1} we get
\[
\liminf_{p\to\infty} \Big(\mathfrak{m}_{s,p}(\mathbb{R}^N)\Big)^\frac{1}{p}\ge \frac{1}{(T-1)^s+T^s}\,\lim_{p\to\infty}\Big(\Lambda_{s,p}(B_1(0))\Big)^\frac{1}{p}=\frac{1}{(T-1)^s+T^s},
\]
for every $T>1$. In particular, by arbitrariness of $T$ we get
\[
\liminf_{p\to\infty} \Big(\mathfrak{m}_{s,p}(\mathbb{R}^N)\Big)^\frac{1}{p}\ge 1.
\]
In order to prove the $\limsup$ inequality, we take the function
\[
\widetilde{f}_\varepsilon(x)=(1-f_\varepsilon(x))_+,\qquad \mbox{ where } f_\varepsilon(x)=\sqrt{\varepsilon^2+|x|^2}-\varepsilon.
\]
This is a compactly supported $1-$Lipschitz function, thus we have that $\widetilde{f}_\varepsilon\in \mathcal{W}^{s,p}(\mathbb{R}^N)$. This implies that
\[
\Big(\mathfrak{m}_{s,p}(\mathbb{R}^N)\Big)^\frac{1}{p}\le \frac{[\widetilde{f}_\varepsilon]_{W^{s,p}(\mathbb{R}^N)}}{[\widetilde{f}_\varepsilon]_{C^{0,\alpha_{s,p}}(\mathbb{R}^N)}}\le\frac{[\widetilde{f}_\varepsilon]_{W^{s,p}(\mathbb{R}^N)}}{|\widetilde{f}_\varepsilon(\mathbf{e}_1)-\widetilde{f}_\varepsilon(0)|}= \frac{[\widetilde{f}_\varepsilon]_{W^{s,p}(\mathbb{R}^N)}}{\sqrt{\varepsilon^2+1}-\varepsilon}.
\]
By taking the limit as $p$ goes to $\infty$ and using \eqref{normegagliardo1}, we get
\[
\limsup_{p\to\infty}\Big(\mathfrak{m}_{s,p}(\mathbb{R}^N)\Big)^\frac{1}{p}\le \frac{[\widetilde{f}_\varepsilon]_{C^{0,s}(\mathbb{R}^N)}}{{\sqrt{\varepsilon^2+1}-\varepsilon}}=\frac{[\widetilde{f}_\varepsilon]_{C^{0,s}(\overline{B_1(0)})}}{{\sqrt{\varepsilon^2+1}-\varepsilon}}=\frac{[f_\varepsilon]_{C^{0,s}(\overline{B_1(0)})}}{{\sqrt{\varepsilon^2+1}-\varepsilon}}.
\] 
In the proof of Lemma \ref{lemma:lambda1}, we have already shown that the last seminorm is smaller than or equal to $1$. This concludes the proof, by arbitrariness of $\varepsilon>0$.
\end{proof}

\subsection{The case $s\nearrow 1$}

We start by recalling that
\[
\begin{split}
\mathfrak{m}_{1,p}(\mathbb{R}^N)
&=\min_{\varphi\in  \mathcal{W}^{1,p}(\mathbb{R}^N)} \left\{\int_{\mathbb{R}^N} |\nabla \varphi|^p\,dx\, :\, [\varphi]_{C^{0,\alpha_{1,p}}(\mathbb{R}^N)}=1 \right\}\\
&=\min_{\varphi\in D^{1,p}(\mathbb{R}^N)} \left\{\int_{\mathbb{R}^N} |\nabla \varphi|^p\,dx\, :\, [\varphi]_{C^{0,\alpha_{1,p}}(\mathbb{R}^N)}=1 \right\},
\end{split}
\]
thanks to Remark \ref{rem:HS}.
The latter is exactly the problem studied in \cite{HS}.
We also define
\[
K_{p,N}=\frac{1}{p}\,\int_{\mathbb{S}^{N-1}} |\langle \omega,\mathbf{e}_1\rangle|\,d\mathcal{H}^{N-1}(\omega).
\]
\begin{theorem}  
\label{thm:limits1}
Let $1\le N<p<\infty$. Then 
\begin{equation}
\label{eq:sasym}
\lim_{s\nearrow 1}(1-s)\, \mathfrak{m}_{s,p}(\mathbb{R}^N)= K_{p,N}\, \mathfrak{m}_{1,p}(\mathbb{R}^N).
\end{equation}
Moroever, if for every $N/p<s<1$ we define $u_{s}\in \mathcal{W}^{s,p}(\mathbb{R}^N)$ to be the unique extremal for $\mathfrak{m}_{s,p}(\mathbb{R}^N)$ such that
\[
u_{s}(y_0)=0,\qquad u_{s}(x_0)=1,\qquad [u_{s}]_{C^{0,\alpha_{s,p}}(\mathbb{R}^N)}=\frac{u_s(x_0)-u_s(y_0)}{|x_0-y_0|^{\alpha_{s,p}}}=1,
\]
then, as $s$ goes to $1$, we get that  $\{u_{s}\}_{N/p<s<1}$ uniformly converges on compact sets to the unique extremal $u\in \mathcal{W}^{1,p}(\mathbb{R}^N)$ for $\mathfrak{m}_{1,p}(\mathbb{R}^N)$ 
satisfying 
\[
u(y_0)=0,\qquad u(x_0)=1,\qquad [u]_{C^{0,\alpha_{1,p}}(\mathbb{R}^N)}=\frac{u_s(x_0)-u_s(y_0)}{|x_0-y_0|^{\alpha_{1,p}}}=1.
\]
\end{theorem}
\begin{proof} 
Thanks to \cite[Corollary 4.3]{BPZ}, we have $\mathfrak{m}_{1,p}(\mathbb{R}^N)=\mathfrak{m}_{1,p}(B_1)$. Then, for every $\varepsilon>0$,    there exists   $u_{\varepsilon}\in C_0^{\infty}(B_1(0))$ such that 
\[
\mathfrak{m}_{1,p}(\mathbb{R}^N)+\varepsilon\geq \int_{B_1(0)} |\nabla u_{\varepsilon}|^p\,dx  \qquad \mbox{ and } \qquad  [u_{\varepsilon}]_{C^{0,\alpha_{1,p}}(\overline{B_1(0)})}=1. 
\]
Since $C^{0,\alpha_{1,p}}(\overline{B_1(0)})\subseteq C^{0,\alpha_{s,p}}(\overline{B_1(0)})$, by Remark \ref{below} and \cite[Corollary 3.20]{EE}, we get 
\[
\begin{split}
\limsup_{s\nearrow 1 }(1-s)\, \mathfrak{m}_{s,p}(\mathbb{R}^N)&\leq \limsup_{s\nearrow 1 }\frac{(1-s)\,[u_{\varepsilon}]^p _{W^{s,p}(\mathbb{R}^N)}   }{[u_{\varepsilon}]^p_{C^{0,\alpha_{s,p}}(\overline{B_1(0)})}}\\
&=\frac{K_{p,N}\,\|\nabla u_{\varepsilon}\|^p_{L^p(B_1)} }{ [u_{\varepsilon}]^p_{C^{0,\alpha_{1,p}}(\overline{B_1(0)})}}\leq K_{p,N}\,\Big(\mathfrak{m}_p(\mathbb{R}^N)+\varepsilon\Big).
\end{split}
\]
By arbitrariness of $\varepsilon>0$, the latter gives that
\begin{equation}
\label{limsupsto1}
 \limsup_{s\nearrow 1 }(1-s)\, \mathfrak{m}_{s,p}(\mathbb{R}^N)\leq K_{p,N}\,\mathfrak{m}_p(\mathbb{R}^N).
\end{equation}
In order to prove the $\liminf$ inequality, for every  $0<s<1$ such that $s\,p>N$, we take $u_{s}$  as in the statement.
By applying Theorem \ref{thm:bound}, we have that
\[
\|u_{s}\|_{L^\infty(\mathbb{R}^N)}=u_s(x_0)=1,\qquad \mbox{ for every } s>\frac{N}{p}.
\]
We now fix $N/p<\overline{s}<1$ and set $\alpha=\overline{s}-{N/p}$. We thus get that $\{u_{s}\}_{\overline{s}<s<1}\subseteq C^{0,\alpha}(\mathbb{R}^N)$ and 
\[
\|u_s\|_{C^{0,\alpha}(\mathbb{R}^N)}=\|u_s\|_{L^\infty(\mathbb{R}^N)}+[u_s]_{C^{0,\alpha}(\mathbb{R}^N)}\le C,\qquad \mbox{ for every } \overline{s}<s<1,
\] 
for a uniform constant $C>0$, thanks to the uniform $L^\infty$ bound and to \eqref{needed}. Hence this family is equibounded and equicontinuous. We now take an increasing sequence $\{s_k\}_{k\in\mathbb{N}}\subseteq(\overline{s},1)$ converging to $1$, such that
\[
\liminf_{s\nearrow 1}(1-s)\,\mathfrak{m}_{s,p}(\mathbb{R}^N)=\lim_{k\to\infty} (1-s_k)\,\mathfrak{m}_{s_k,p}(\mathbb{R}^N).
\]
By using the Ascoli-Arzel\`a Theorem and a standard diagonal argument, we can infer existence of a subsequence $\{s_{k_n}\}_{n\in\mathbb{N}}\subseteq \{s_k\}_{k\in\mathbb{N}}$, such that the sequence $\{u_{s_{k_n}}\}_{n\in\mathbb{N}}$ uniformly converges on compact sets to a function $u\in L^\infty(\mathbb{R}^N)$, as $n$ goes to $\infty$. 
\par
We claim that $u\in\mathcal{W}^{1,p}(\mathbb{R}^N)$. Let $x,\,y\in\mathbb{R}^N$ with $x\neq y$, then we have
\[
\frac{|u(x)-u(y)|}{|x-y|^{\alpha_{1,p}}}=\lim_{n\to \infty} \frac{|u_{s_{k_n}}(x)-u_{s_{k_n}}(y)|}{|x-y|^{\alpha_{s_{k_n},p}}}\leq 1, 
\]
which implies that $u\in \dot C^{0,\alpha_{1,p}}(\mathbb{R}^N)$. In addition, we have
\[
[u]_{C^{0,\alpha_{1,p}}(\mathbb{R}^N)} \geq \frac{|u(x_0)-u(y_0)|}{|x_0-y_0|^{\alpha_{1,p}}}=\lim_{n\to\infty}\frac{|u_{s_{k_n}}(x_0)-u_{s_{k_n}}(y_0)|}{|x_0-y_0|^{\alpha_{s_{k_n},p}}} =1,
\]
so that
\[
[u]_{C^{0,\alpha_{1,p}}(\mathbb{R}^N)}=1.
\] 
We still need to prove that $\nabla u\in L^p(\mathbb{R}^N)$.
By appealing Lemma \ref{lm:prin}, there exists a constant $C=C(N)>0$ such that
\[
\sup_{|h|>0}\int_{\mathbb{R}^N} \frac{|u_s(x+h)-u_s(x)|^p}{|h|^{s\,p}}\,dx\le C\,(1-s)\,[u_s]^p_{W^{s,p}(\mathbb{R}^N)}= C\,(1-s)\,\mathfrak{m}_{s,p}(\mathbb{R}^N).
\]
In particular, by Fatou's Lemma we get for every $|h|>0$
\[
\int_{\mathbb{R}^N} \frac{|u(x+h)-u(x)|^p}{|h|^{p}}\,dx\le \liminf_{n\to\infty}\int_{\mathbb{R}^N} \frac{|u_{s_{k_n}}(x+h)-u_{s_{k_n}}(x)|^p}{|h|^{s_{k_n}\,p}}\,dx\le \liminf_{s\nearrow 1}C\,(1-s)\,\mathfrak{m}_{s,p}(\mathbb{R}^N).
\]
By the classical characterization of Sobolev spaces in terms of finite differences, the last estimate, the arbitariness of $h$ and \eqref{limsupsto1} show that $\nabla u\in L^p(\mathbb{R}^N)$. Moreover,  for every $R>0$, thanks to  \cite[Theorem 8]{ponce} and to \eqref{limsupsto1},  we have that 
\[
 K_{p,N}\, \|\nabla u\| ^p_{L^p(B_R)} \leq \liminf_{n\to \infty}(1-s_{k_n})\, [u_{s_{k_n}}]^p_{W^{s_{k_n},p}(B_R)}\leq  \liminf_{s\nearrow 1}(1-s)\,\mathfrak{m}_{s,p}(\mathbb{R}^N)\le K_{p,N}\,\mathfrak{m}_p(\mathbb{R}^N).
\]
By the Monotone Convergence Theorem, the previous estimate holds by replacing the $L^p$ norm on $B_R$ with that on $\mathbb{R}^N$, on the first term in the left-hand side. By definition of $\mathfrak{m}_p(\mathbb{R}^N)$, this yields
\begin{equation}
\label{liminf est of msp}
\begin{split}
	K_{p,N}\, \mathfrak{m}_{1,p}(\mathbb{R}^N)&\leq K_{p,N}\, \|\nabla u\|^p_{L^p(\mathbb{R}^N)}\\
	 &\leq \liminf_{s\nearrow 1}(1-s)\,\mathfrak{m}_{s,p}(\mathbb{R}^N)\le K_{p,N}\,\mathfrak{m}_p(\mathbb{R}^N).
	\end{split}
\end{equation}
Combining \eqref{limsupsto1} and \eqref{liminf est of msp}, we get \eqref{eq:sasym} and that $u$ is an extremal for $\mathfrak{m}_p(\mathbb{R}^N)$, having the claimed normalization properties.  \par
Finally, we observe that by \cite[Corollary 3.2]{HS} such a minimizer is unique. Hence, we get convergence of the whole family $\{u_{s}\}_{N/p<s<1}$ to $u$, uniformly on compact sets.
\end{proof}

\appendix

\section{A family of test functions}
\label{app:A}
For every $N\ge 1$, $0<s<1$ and $1<p<\infty$ such that $s\,p>N$, we consider the function
\[
\zeta(x)=\left(1-|x|^\frac{s\,p-N}{p-1}\right)_+,\qquad \mbox{ for } x\in\mathbb{R}^N.
\]
With the aim to estimate the   Sobolev-Slobodecki\v{\i} seminorm $
[\zeta]^p_{W^{s,p}(\mathbb{R}^N)}$, we prove the following preliminary lemmas.
The first one is an extension of \cite[Lemma B.1]{BB}, the latter being concerned with the case $p=2$ only.
\begin{lemma}[One dimension]
Let $N=1$, $0<s<1$ and $1<p<\infty$ such that $s\,p>1$. There exists a constant $C=C(p)>0$  such that    
\begin{equation}
\label{estimate1d}
\begin{split}
[\zeta]^p_{W^{s,p}(\mathbb{R})}&\le  \frac{C}{s\,p-1}   \left(  \int_0^\frac{1}{2}\left|1-t^\frac{s\,p-1}{p-1}\right|^p\,\left(1+t^\frac{1-s\,p}{p-1}\right) \,dt +  \frac{1}{1-s}  \left(\frac{s\,p-1}{p-1}\right)^p\right)\\
&+\frac{4}{s\,p}\,\int_{-1}^1 \frac{\left(1-|t|^\frac{s\,p-1}{p-1}\right)^p}{(1-t)^{s\,p}}\,dt.
\end{split}
\end{equation}
The constant $C$ blows-up as $p\searrow 1$.
\end{lemma}
\begin{proof}
This can be proved as in \cite[Lemma B.1]{BB}, with minor modifications. Since $\zeta$ identically vanishes outside $(-1,1)$, we can write
\[
[\zeta]_{W^{s,p}(\mathbb{R})}^p=[\zeta]_{W^{s,p}((-1,1))}^p+2\,\iint_{(-1,1)\times(\mathbb{R}\setminus (-1,1))} \frac{|\zeta(x)|^p}{|x-y|^{1+s\,p}}\,dx\,dy.
\]
For $x\in (-1,1)$, we have
\[
\int_{\mathbb{R}\setminus(-1,1)} \frac{dy}{|x-y|^{1+s\,p}}=\frac{1}{s\,p}\,\frac{1}{(1-x)^{s\,p}}+\frac{1}{s\,p}\,\frac{1}{(1+x)^{s\,p}}.
\]
By using this fact, a change of variable and the definition of $\zeta$, we thus get
\[
[\zeta]_{W^{s,p}(\mathbb{R})}^p=[\zeta]_{W^{s,p}((-1,1)}^p+\frac{4}{s\,p}\,\int_{-1}^1 \frac{\left(1-|t|^\frac{s\,p-1}{p-1}\right)^p}{(1-t)^{s\,p}}\,dt.
\]
By using the symmetry of the set and of the integrand, we can estimate
\[
[\zeta]^p_{W^{s,p}((-1,1)}\ \le 4\,\iint_{(0,1)\times (0,1)} \frac{\Big||x|^\frac{s\,p-1}{p-1}-|y|^\frac{s\,p-1}{p-1}\Big|^p}{|x-y|^{1+s\,p}}\,dx\,dy=:\mathcal{I}.
\]
We apply \cite[Remark 4.2, formula (4.3)]{BBZ} with the choice $\beta=(s\,p-1)/(p-1)$ there, so the term $\mathcal{I}$ can be estimated as follows
\[
\mathcal{I}\le 4\,\iint_{(0,1)\times (0,1)} \frac{\Big||x|^\frac{s\,p-1}{p-1}-|y|^\frac{s\,p-1}{p-1}\Big|^p}{|x-y|^{1+s\,p}}\,dx\,dy\le \frac{4\,(p-1)}{s\,p-1}\,\int_0^1 \frac{\left|1-t^\frac{s\,p-1}{p-1}\right|^p}{|1-t|^{1+s\,p}}\,\left(1+t^\frac{1-s\,p}{p-1}\right)\,dt.
\]
In order to estimate the last integral, we break it into two pieces
\[
\begin{split}
\int_0^1 \frac{\left|1-t^\frac{s\,p-1}{p-1}\right|^p}{|1-t|^{1+s\,p}}\,\left(1+t^\frac{1-s\,p}{p-1}\right)\,dt&\le 2^{1+s\,p}\,\int_0^\frac{1}{2} \left|1-t^\frac{s\,p-1}{p-1}\right|^p\,\left(1+t^\frac{1-s\,p}{p-1}\right)\,dt\\
&+\left(1+2^\frac{s\,p-1}{p-1}\right)\,\int_\frac{1}{2}^1 \frac{\left|1-t^\frac{s\,p-1}{p-1}\right|^p}{|1-t|^{1+s\,p}}\,dt.
\end{split}
\]
In turn, the last integral can be estimated
by appealing to the following concavity inequality  
\begin{equation}
\label{conca}
a^\alpha-b^\alpha\le \alpha\,b^{\alpha-1}\,(a-b),\qquad \mbox{ for } 0<b\le a, \, 0<\alpha<1.
\end{equation} 
With the choice $\alpha=(s\,p-1)/(p-1)$, we get
\begin{equation}
\label{stimaI12}
\begin{split}
\int_\frac{1}{2}^1 \frac{\left|1-t^\frac{s\,p-1}{p-1}\right|^p}{|1-t|^{1+s\,p}}\,dt&\le  \left(\frac{s\,p-1}{p-1}\right)^p\,2^\frac{p\,(1-s)}{p-1}\,\int_\frac{1}{2}^1 (1-t)^{(1-\,s)\,p-1}\,dt\\
&=\frac{2^{\frac{p\,(2-p)}{p-1}\,(s-1)}}{p\,(1-s)}\, \left(\frac{s\,p-1}{p-1}\right)^p.
\end{split}
\end{equation} 
This is enough to conclude.
\end{proof}
In higher dimension we get a very similar estimate, but the computations are a bit more involved.
\begin{lemma}[Higher dimension]
Let $N\ge 2$, $0<s<1$ and $1<p<\infty$ such that $s\,p>N$. There exists a constant $C=C(N,p)>0$  such that    
\begin{equation}
\label{estimate2d}
\begin{split}
[\zeta]^p_{W^{s,p}(\mathbb{R}^N)}&\le  \frac{C}{s\,p-N}   \left(  \int_0^\frac{1}{2}\left|1-t^\frac{s\,p-N}{p-1}\right|^p\,\left(1+t^{\frac{N\,p-s\,p}{p-1}-1}\right) \,dt +  \frac{1}{1-s}  \left(\frac{s\,p-N}{p-1}\right)^p \right)\\
&+\frac{2\,(N\,\omega_N)^2}{s\,p}\,\int_{0}^1 \frac{\left(1-t^\frac{s\,p-N}{p-1}\right)^p}{(1-t)^{s\,p}}\,t^{N-1}\,dt.
\end{split}
\end{equation}
\end{lemma}
\begin{proof}
Since $\zeta$ identically vanishes outside $B_1(0)$, we first observe that we have
\[
[\zeta]_{W^{s,p}(\mathbb{R}^N)}^p=[\zeta]_{W^{s,p}(B_1(0))}^p+2\,\iint_{B_1(0)\times(\mathbb{R}^N\setminus B_1(0))} \frac{|\zeta(x)|^p}{|x-y|^{N+s\,p}}\,dx\,dy.
\]
Moreover, for every $x\in B_1(0)$, we have 
\[
\mathbb{R}^N\setminus B_1(0) \subseteq \mathbb{R}^N\setminus B_d(x),
\]
where we set for brevity $d=\mathrm{dist}(x,\partial B_1(0))=1-|x|$. This permits to infer that
\[
\begin{split}
[\zeta]_{W^{s,p}(\mathbb{R}^N)}^p&\le [\zeta]_{W^{s,p}(B_1(0))}^p+2\,\int_{B_1(0)} |\zeta(x)|^p\,\left(\int_{\mathbb{R}^N\setminus B_d(x)} \frac{dy}{|x-y|^{N+s\,p}}\right)\,dx\\
&=[\zeta]_{W^{s,p}(B_1(0))}^p+\frac{2\,N\,\omega_N}{s\,p}\,\int_{B_1(0)} \frac{|\zeta(x)|^p}{(1-|x|)^{s\,p}}\,dx\\
&=[\zeta]_{W^{s,p}(B_1(0))}^p+\frac{2\,(N\,\omega_N)^2}{s\,p}\,\int_{0}^1 \frac{\left(1-t^\frac{s\,p-N}{p-1}\right)^p}{(1-t)^{s\,p}}\,t^{N-1}\,dt.
\end{split}
\]
By using spherical coordinates and the fact that 
\[
|\zeta(x)-\zeta(y)|=\Big||x|^\frac{s\,p-N}{p-1}-|y|^\frac{s\,p-N}{p-1}\Big|,\qquad \mbox{ for }x,y\in B_1(0),
\] 
we can write
\[
[\zeta]^p_{W^{s,p}(B_1(0))}=\iint_{(0,1)\times(0,1)} \left|\varrho^\frac{s\,p-N}{p-1}-r^\frac{s\,p-N}{p-1}\right|^p\,\varrho^{N-1}\,r^{N-1}\,\Psi_{N,s\,p}(\varrho,r)\,d\varrho\,dr.
\]
Here the function $\Psi_{N,s\,p}$ is given by
\[
\Psi_{N,s\,p}(\varrho,r)=\int_{-1}^1 \frac{(1-t^2)^\frac{N-3}{2}}{\Big((\varrho-r)^2+2\,\varrho\,r\,(1-t)\Big)^{\frac{N+s\,p}{2}}}\,dt,
\]
and it has the following properties
\begin{equation}
\label{propp}
\Psi_{N,s\,p}(\varrho,r)=\Psi_{N,s\,p}(r,\varrho)\qquad \mbox{ and }\qquad \Psi_{N,s\,p}(\varrho,r)=\frac{1}{\varrho^{N+s\,p}}\,\Phi_{N,s\,p}\left(\frac{r}{\varrho}\right),
\end{equation}
where the function $\Phi_{N,s\,p}$ is the same as in \eqref{phistrange}.
We can rewrite
\[
\begin{split}
[\zeta]^p_{W^{s,p}(B_1(0))}&=\iint_{(0,1)\times(0,1)} \left|\varrho^\frac{s\,p-N}{p-1}-r^\frac{s\,p-N}{p-1}\right|^p\,\varrho^{N-1}\,r^{N-1}\,\Psi_{N,s\,p}(\varrho,r)\,d\varrho\,dr\\
&=\int_0^1 \varrho^{\frac{s\,p-N}{p-1}-2}\,\left(\int_0^1 \left|1-\left(\frac{r}{\varrho}\right)^\frac{s\,p-N}{p-1}\right|^p\,\left(\frac{r}{\varrho}\right)^{N-1}\,\Psi_{N,s\,p}\left(1,\frac{r}{\varrho}\right)\,dr\right)\,d\varrho.
\end{split}
\]
We pass to analyze the integral in $r$: we decompose it as follows
\[
\begin{split}
\int_0^1 \left|1-\left(\frac{r}{\varrho}\right)^\frac{s\,p-N}{p-1}\right|^p\,\left(\frac{r}{\varrho}\right)^{N-1}\,\Phi\left(1,\frac{r}{\varrho}\right)\,dr&=\int_0^\varrho \left|1-\left(\frac{r}{\varrho}\right)^\frac{s\,p-N}{p-1}\right|^p\,\left(\frac{r}{\varrho}\right)^{N-1}\,\Psi_{N,s\,p}\left(1,\frac{r}{\varrho}\right)\,dr\\
&+\int_\varrho^1 \left|1-\left(\frac{r}{\varrho}\right)^\frac{s\,p-N}{p-1}\right|^p\,\left(\frac{r}{\varrho}\right)^{N-1}\,\Psi_{N,s\,p}\left(1,\frac{r}{\varrho}\right)\,dr
\end{split}
\]
and then we make the change of variable $r/\varrho=t$ in the first integral and $r/\varrho=1/t$ is the second one. This gives
\[
\begin{split}
\int_0^1 \left|1-\left(\frac{r}{\varrho}\right)^\frac{s\,p-N}{p-1}\right|^p\,\left(\frac{r}{\varrho}\right)^{N-1}\,\Psi_{N,s\,p}\left(1,\frac{r}{\varrho}\right)\,dr&=\varrho\,\int_0^1 \left|1-t^\frac{s\,p-N}{p-1}\right|^p\,t^{N-1}\,\Psi_{N,s\,p}\left(1,t\right)\,dt\\
&+\varrho\,\int_\varrho^1 \frac{\left|1-t^\frac{s\,p-N}{p-1}\right|^p}{t^{\frac{s\,p-N}{p-1}+N+1}}\,\Psi_{N,s\,p}\left(1,\frac{1}{t}\right)\,dt.
\end{split}
\] 
In the last integral, we can use \eqref{propp} to infer that 
\[
\Psi_{N,s\,p}\left(1,\frac{1}{t}\right)=t^{N+s\,p}\,\Phi_{N,s\,p}(t).
\]
Thus, we finally get
\[
\begin{split}
\int_0^1 \left|1-\left(\frac{r}{\varrho}\right)^\frac{s\,p-N}{p-1}\right|^p\,\left(\frac{r}{\varrho}\right)^{N-1}\,\Psi_{N,s\,p}\left(1,\frac{r}{\varrho}\right)\,dr&=\varrho\,\int_0^1 \left|1-t^\frac{s\,p-N}{p-1}\right|^p\,t^{N-1}\,\Phi_{N,s\,p}(t)\,dt\\
&+\varrho\,\int_\varrho^1 \frac{\left|1-t^\frac{s\,p-N}{p-1}\right|^p}{t^{\frac{s\,p-N\,p}{p-1}+1}}\,\Phi_{N,s\,p}(t)\,dt\\
&\le \varrho\,\int_0^1 \left|1-t^\frac{s\,p-N}{p-1}\right|^p\,\left(t^{N-1}+t^{\frac{N\,p-s\,p}{p-1}-1}\right)\,\Phi_{N,s\,p}(t)\,dt .
\end{split}
\]
We notice that 
\[
\frac{s\,p-N}{p-1}-1>-1,
\]
hence, from the previous estimates we obtain 
\[
\begin{split}
[\zeta]^p_{W^{s,p}(B_1(0))}&\le \left(\int_0^1\varrho^{\frac{s\,p-N}{p-1}-1}\,d\varrho\right)\,\int_0^1 \left|1-t^\frac{s\,p-N}{p-1}\right|^p\,\left(t^{N-1}+t^{\frac{N\,p-s\,p}{p-1}-1}\right)\,\Phi_{N,s\,p}(t)\,dt\\
&= \frac{p-1}{s\,p-N}\,\int_0^1 \left|1-t^\frac{s\,p-N}{p-1}\right|^p\,\left(t^{N-1}+t^{\frac{N\,p-s\,p}{p-1}-1}\right)\,\Phi_{N,s\,p}(t)\,dt\\
&\leq  \frac{p-1}{s\,p-N}\,\int_0^1 \left|1-t^\frac{s\,p-N}{p-1}\right|^p\,\left(1+t^{\frac{N\,p-s\,p}{p-1}-1}\right)\,\Phi_{N,s\,p}(t)\,dt=:\frac{p-1}{s\,p-N}\,\mathcal I.
\end{split}
\]
As in the one-dimensional case, we can break $\mathcal{I}$ in two pieces
\[
\mathcal I_1=\int_0^{1/2} \left|1-t^\frac{s\,p-N}{p-1}\right|^p\,\left(1+t^{\frac{N\,p-s\,p}{p-1}-1}\right)\,\Phi_{N,s\,p}(t)\,dt,
\]
and
\[
 \mathcal{I}_2:=\int_{1/2}^1 \left|1-t^\frac{s\,p-N}{p-1}\right|^p\,\left(1+t^{\frac{N\,p-s\,p}{p-1}-1}\right)\,\Phi_{N,s\,p}(t)\,dt.
\]
For $0<t\leq 1/2$, by recalling the definition \eqref{phistrange} it holds
\[
\Phi_{N,s\,p}(t) \leq \frac{1}{(1-t)^{N+s\,p}}\,  \int_{-1}^1 {(1-\tau^2)^\frac{N-3}{2}}\,d\tau\leq C\, 2^{N+p},
\]	
with $C=C(N)>0$ (observe that the last integral is converging for $N\ge 2$). This yields 
\[
\mathcal I_1\leq  C\, \int_0^{1/2} \left|1-t^\frac{s\,p-N}{p-1}\right|^p\,\left(1+t^{\frac{N\,p-s\,p}{p-1}-1}\right)\,dt,
\]
possibly for a different $C=C(N,p)>0$.
 In the case $1/2 \leq t<1$, by using the change of variable 
\[
z=\displaystyle \frac{2t}{(1-t)^2}(1-\tau),
\] 
it is not difficult to show that 
\[
\Phi_{N,s\,p}(t) \le C\,(1-t)^{-(s\,p+1)},  
\]
for a dimensional constant $C>0$ (see \cite[formula (A.8)]{BMS}). 
By using this estimate and again \eqref{conca}, we obtain that 
\[  
\begin{split} 
\mathcal I_2&\leq C\,\int_{1/2}^1 \frac{\left|1-t^\frac{s\,p-N}{p-1}\right|^p}{(1-t)^{1+s\,p}}\,\left(1+t^{\frac{N\,p-s\,p}{p-1}-1}\right)\,\Phi_{N,s\,p}(t)\,dt\\
&\leq  C\,  \left(\frac{s\,p-N}{p-1}\right)^p  \int_{1/2}^1 \left(1-t\right)^{(1-s)\,p-1}\,  \left( t^{\left(\frac{s\,p-N}{p-1}-1\right)\,p}  +t^{s\,p-p-1}\right)\,dt\\
&\leq   C\,  \left( 2^{\left(\frac{N-s\,p}{p-1}+1\right)\,p}+ 2^{-s\,p+p+1}\right)\,\left(\frac{s\,p-N}{p-1}\right)^p  \int_{1/2}^1 \left(1-t\right)^{(1-s)\,p-1}\,dt.
\end{split} 
\]
By computing the last integral and combining the estimates given for  $\mathcal I_1$ and $\mathcal I_2$, we finally obtain the desired conclusion.
\end{proof}
Finally, the following estimate is the main point of this section.
\begin{lemma}
\label{lm:gloria!}
For $N\ge 1$, $0<s<1$ and $1<p<\infty$ such that $s\,p>N$, we have 
\[
[\zeta]^p_{W^{s,p}(\mathbb{R}^N)}\le  C\,\frac{(s\,p-N)^{p-1}}{1-s},\qquad \mbox{ for every } \frac{N}{p}<s<1.
\]
The constant $C=C(N,p)>0$ does not depend on $s$.
\end{lemma}
\begin{proof}
Let us give the proof for $N\ge 2$, the case $N=1$ being exactly the same, with \eqref{estimate1d} in place of \eqref{estimate2d}. We proceed as in the proof of \cite[Lemma B.1]{BB}: we observe at first that 
\[
\frac{N\,p-s\,p}{p-1}-1>0,\qquad \mbox{ for }N\ge 2,
\]
thus we get
\[
\int_0^\frac{1}{2}\left|1-t^\frac{s\,p-N}{p-1}\right|^p\,\left(1+t^{\frac{N\,p-s\,p}{p-1}-1}\right) \,dt\le 2\,\int_0^\frac{1}{2}\left|1-t^\frac{s\,p-N}{p-1}\right|^p\, \,dt.
\]
In order to estimate the last integral,  we define
\[
h_\tau(s)=\tau^\frac{s\,p-N}{p-1},\qquad \mbox{ for } \tau>0,\, s>\frac{N}{p}.
\]
Then we have
\[
\left|h_\tau(s)-h_\tau\left(\frac{N}{p}\right)\right|=\left|\int^s_\frac{N}{p} h'_\tau(t)\,dt\right|.
\]
Thus, for $0<\tau\le 1/2$ we get
\[
\left|1-\tau^\frac{s\,p-N}{p-1}\right|=\frac{p}{p-1}\,|\log \tau|\,\left|\int_\frac{N}{p}^s \tau^\frac{t\,p-N}{p-1}\,dt\right|\le \frac{p}{p-1}\,(-\log \tau)\,\left(s-\frac{N}{p}\right)=(-\log \tau)\,\frac{s\,p-N}{p-1}.
\]
We then obtain for $N/p<s<1$ and $N\ge 2$
\begin{equation}
\label{I11low}
\int_0^\frac{1}{2}\left|1-t^\frac{s\,p-N}{p-1}\right|^p\, \,dt\le \left(\frac{s\,p-N}{p-1}\right)^p\,\int_0^\frac{1}{2} (-\log \tau)^p\,d\tau,
\end{equation}
and this gives the desired estimate, since the last integral is finite and independent of $s$. Thus, from \eqref{estimate2d} we get
\[
\begin{split}
[\zeta]^p_{W^{s,p}(\mathbb{R}^N)}&\le  \frac{C}{s\,p-N}   \left(\left(\frac{s\,p-N}{p-1}\right)^p +  \frac{1}{1-s}  \left(\frac{s\,p-N}{p-1}\right)^p \right)\\
&+\frac{2\,(N\,\omega_N)^2}{s\,p}\,\int_{0}^1 \frac{\left(1-t^\frac{s\,p-N}{p-1}\right)^p}{(1-t)^{s\,p}}\,t^{N-1}\,dt,
\end{split}
\]
possibly for a different constant $C=C(N,p)>0$. The last integral can be handled in a similar manner: we have
\[
\begin{split}
\int_{0}^1 \frac{\left(1-t^\frac{s\,p-N}{p-1}\right)^p}{(1-t)^{s\,p}}\,t^{N-1}\,dt&\le \int_{0}^1 \frac{\left(1-t^\frac{s\,p-N}{p-1}\right)^p}{(1-t)^{s\,p}}\,dt\\
&=\int_{0}^\frac{1}{2} \frac{\left(1-t^\frac{s\,p-N}{p-1}\right)^p}{(1-t)^{s\,p}}\,dt+\int_\frac{1}{2}^1 \frac{\left(1-t^\frac{s\,p-N}{p-1}\right)^p}{(1-t)^{s\,p}}\,dt\\
&\le 2^p\,\int_{0}^\frac{1}{2} \left(1-t^\frac{s\,p-N}{p-1}\right)^p\,dt+\int_\frac{1}{2}^1 \frac{\left(1-t^\frac{s\,p-N}{p-1}\right)^p}{(1-t)^{s\,p}}\,dt.
\end{split}
\]
The first integral on the right-hand side has been already estimated in \eqref{I11low}. The second one can be estimated as in \eqref{stimaI12}, by relying on \eqref{conca}. We leave the details to the reader.
\end{proof}

\section{Homogeneous space on the punctured space}
\label{app:B}

We still consider the space $\mathcal{W}^{s,p}(\mathbb{R}^N)$ defined by \eqref{goodspace}. For an open set $\Omega\subseteq\mathbb{R}^N$, we indicate by $\mathscr{D}^{s,p}_0(\Omega)$ the completion of $C^\infty_0(\Omega)$, with respect to the norm
\[
\varphi\mapsto [\varphi]_{W^{s,p}(\mathbb{R}^N)},\qquad\mbox{ for every }\varphi\in C^\infty_0(\Omega).
\]
For a characterization of this space when $\Omega=\mathbb{R}^N$, we refer to \cite{BGCV, MPS}. In the case $s\,p>N$, when $\Omega$ coincides with the whole space minus a finite number of points, we have the following
\begin{proposition}
\label{prop:homo}
Let $N\ge 1$, $0<s\le 1$ and $1<p<\infty$  be such that $s\,p>N$. For every $k\in\mathbb{N}$ and every $x_1,\dots,x_k\in\mathbb{R}^N$, we have 
\[
\mathscr{D}^{s,p}_0(\mathbb{R}^N\setminus\{x_1,\dots,x_k\})=\Big\{u\in \mathcal{W}^{s,p}(\mathbb{R}^N)\, :\, u(x_1)=\dots=u(x_k)=0\Big\}.
\]
\end{proposition}
\begin{proof}
We will prove the result for the fractional case $0<s<1$ only. The limit case $s=1$ is actually simpler and can be proved by repeating the same arguments.
We also notice that it is sufficient to prove the result for $k=1$, in the general case the proof is exactly the same.
Finally, without loss of generality we can suppose that $x_1$ coincides with the origin.
\par
We first observe that if $u\in \mathcal{W}^{s,p}(\mathbb{R}^N)$ is such that $u(0)=0$, then by Morrey's inequality we have
\begin{equation}
\label{growth}
|u(x)|^p= |u(x)-u(0)|^p\le \frac{[u]^p_{W^{s,p}(\mathbb{R}^N)}}{\mathfrak{m}_{s,p}(\mathbb{R}^N)}\,|x|^{s\,p-N},\qquad \mbox{ for every } x\in \mathbb{R}^N,
\end{equation}
We take a cut-off function $\eta\in C^\infty(\mathbb{R}^N)$ such that
\[
0\le \eta\le 1,\qquad \eta\equiv 0 \mbox{ in } B_1(0),\qquad \eta \equiv 1 \mbox{ in } \mathbb{R}^N\setminus B_2(0),\qquad \|\nabla \eta\|_{L^\infty}\le C_N.
\]
{\bf Truncation at the origin}.
For every $n\in\mathbb{N}\setminus\{0\}$ we define 
\[
\eta_n(x):=\eta(n\,x),\qquad \mbox{ for } x\in \mathbb{R}^N.
\]
We wish to prove at first that there exists a constant $C=C(N,s,p)>0$ such that 
\begin{equation}
\label{bound}
[u\,\eta_n]_{W^{s,p}(\mathbb{R}^N)}^p\le C\,[u]^p_{W^{s,p}(\mathbb{R}^N)},\qquad \mbox{ for every } n\in\mathbb{N}.
\end{equation}
We start by decomposing the seminorm as follows
\[
\begin{split}
[u\,\eta_n]_{W^{s,p}(\mathbb{R}^N)}^p&=\iint_{B_\frac{3}{n}(0)\times B_\frac{3}{n}(0)} \frac{|u(x)\,\eta_n(x)-u(y)\,\eta_n(y)|^p}{|x-y|^{N+s\,p}}\,dx\,dy\\
&+2\,\iint_{B_\frac{3}{n}(0)\times (\mathbb{R}^N\setminus B_\frac{3}{n}(0))} \frac{|u(x)\,\eta_n(x)-u(y)|^p}{|x-y|^{N+s\,p}}\,dx\,dy\\
&+\iint_{(\mathbb{R}^N\setminus B_\frac{3}{n}(0))\times (\mathbb{R}^N\setminus B_\frac{3}{n}(0))} \frac{|u(x)-u(y)|^p}{|x-y|^{N+s\,p}}\,dx\,dy:=\mathcal{J}_1+2\,\mathcal{J}_2+\mathcal{J}_3.
\end{split}
\]
It is easily seen that $\mathcal{J}_3$ is uniformly bounded by the seminorm of $u$ on the whole $\mathbb{R}^N$.
For $\mathcal{J}_1$, by the triangle inequality and the fact that $0\le \eta_n\le 1$, we have 
\[
\begin{split}
\mathcal{J}_1&\le 2^{p-1}\,\iint_{B_\frac{3}{n}(0)\times B_\frac{3}{n}(0)} \frac{|u(x)-u(y)|^p\,|\eta(x)|^p}{|x-y|^{N+s\,p}}\,dx\,dy\\
&+2^{p-1}\,\iint_{B_\frac{3}{n}(0)\times B_\frac{3}{n}(0)} \frac{|\eta_n(x)-\eta_n(y)|^p\,|u(y)|^p}{|x-y|^{N+s\,p}}\,dx\,dy\\
&\le 2^{p-1}\,[u]^p_{W^{s,p}(B_\frac{2}{n}(0))}+2^{p-1}\,\frac{[u]^p_{W^{s,p}(\mathbb{R}^N)}}{\mathfrak{m}_{s,p}(\mathbb{R}^N)}\,\left(\frac{3}{n}\right)^{N-s\,p}\,\iint_{B_\frac{3}{n}(0)\times B_\frac{3}{n}(0)} \frac{|\eta_n(x)-\eta_n(y)|^p}{|x-y|^{N+s\,p}}\,dx\,dy,
\end{split}
\]
where we also used \eqref{growth}. By recalling the construction of $\eta_n$, a change of variable finally leads to
\[
\mathcal{J}_1\le C\,[u]^p_{W^{s,p}(\mathbb{R}^N)}\,\left(1 +[\eta]^p_{W^{s,p}(B_3(0))}\right),
\]
for some $C=C(N,s,p)>0$. In turn, by using the Lipschitz character of $\eta$, we have
\[
[\eta]^p_{W^{s,p}(B_3(0))}\le \|\nabla \eta\|_{L^\infty}^p\,\iint_{B_3(0)\times B_3(0)} |x-y|^{p\,(1-s)-N}\,dx\,dy\le C.
\]
This permits to conclude that 
\[
\mathcal{J}_1\le C\,[u]^p_{W^{s,p}(\mathbb{R}^N)},
\]
for some $C=C(N,s,p)>0$.
For $\mathcal{J}_2$, we recall that $\eta_n\equiv 1$ outside $B_{2/n}(0)$ and that $0\le \eta_n\le 1$, thus we have
\[
\begin{split}
\mathcal{J}_2&=\iint_{B_\frac{2}{n}(0)\times (\mathbb{R}^N\setminus B_\frac{3}{n}(0))} \frac{|u(x)\,\eta_n(x)-u(y)|^p}{|x-y|^{N+s\,p}}\,dx\,dy\\
&+\iint_{(B_\frac{3}{n}(0)\setminus B_\frac{2}{n}(0))\times (\mathbb{R}^N\setminus B_\frac{3}{n}(0))} \frac{|u(x)-u(y)|^p}{|x-y|^{N+s\,p}}\,dx\,dy\\
&\le 2^{p-1}\,\iint_{B_\frac{2}{n}(0)\times (\mathbb{R}^N\setminus B_\frac{3}{n}(0))} \frac{|u(x)-u(y)|^p}{|x-y|^{N+s\,p}}\,dx\,dy\\
&+2^{p-1}\,\iint_{B_\frac{2}{n}(0)\times (\mathbb{R}^N\setminus B_\frac{3}{n}(0))} \frac{|\eta_n(x)-1|^p\,|u(y)|^p}{|x-y|^{N+s\,p}}\,dx\,dy\\
&+[u]_{W^{s,p}(\mathbb{R}^N)}^p\\
&\le C\,[u]_{W^{s,p}(\mathbb{R}^N)}^p+2^{p-1}\,n^{s\,p-N}\,\iint_{B_2(0)\times (\mathbb{R}^N\setminus B_3(0))} \frac{|\eta(x)-1|^p\,|u(y/n)|^p}{|x-y|^{N+s\,p}}\,dx\,dy.
\end{split}
\]
In the last estimate we also used a change of variable. We now observe that for $y\in \mathbb{R}^N\setminus B_3(0)$ and $x\in B_2(0)$, we have
\[
|x-y|\ge |y|-|x|\ge |y|-2\ge \frac{1}{3}\,|y|\ge \frac{1}{6}\,(1+|y|).
\]
Thus we get
\[
\begin{split}
\iint_{B_2(0)\times (\mathbb{R}^N\setminus B_3(0))} \frac{|\eta(x)-1|^p\,|u(y/n)|^p}{|x-y|^{N+s\,p}}\,dx\,dy&\le C\,\iint_{B_2(0)\times (\mathbb{R}^N\setminus B_3(0))} \frac{|u(y/n)|^p}{(1+|y|)^{N+s\,p}}\,dx\,dy\\
&=C\,N\,\omega_N\,2^N\,\int_{\mathbb{R}^N\setminus B_3(0)} \frac{|u(y/n)|^p}{(1+|y|)^{N+s\,p}}\,dy.
\end{split}
\]
By using again \eqref{growth}, we have 
\[
\iint_{B_2(0)\times (\mathbb{R}^N\setminus B_3(0))} \frac{|\eta(x)-1|^p\,|u(y/n)|^p}{|x-y|^{N+s\,p}}\,dx\,dy\le C\,[u]^p_{W^{s,p}(\mathbb{R}^N)}\,n^{N-s\,p}\,\int_{\mathbb{R}^N\setminus B_3(0)} \frac{|y|^{s\,p-N}}{(1+|y|)^{N+s\,p}}\,dy,
\]
and the last integral is finite. By collecting the previous estimates, we finally obtain that $\mathcal{J}_2$ is uniformly bounded, as well. Thus, \eqref{bound} holds true.
\vskip.2cm\noindent
{\bf Truncation at infinity}. For every $n\in\mathbb{N}\setminus\{0\}$, we now set
\[
\psi_n(x):=\eta\left(\frac{x}{n}\right),\qquad \mbox{ for }x\in\mathbb{R}^N.
\]
By applying \cite[Lemma B.2]{BGCV} to the function $u\,\eta_n$, we get
\begin{equation}
\label{bound2}
[u\,\eta_n\,\psi_n]_{W^{s,p}(\mathbb{R}^N)}\le C\,[u\,\eta_n]_{W^{s,p}(\mathbb{R}^N)},\qquad \mbox{ for every } n\in\mathbb{N},
\end{equation}
for some $C=C(N,s,p)>0$.
\vskip.2cm\noindent
{\bf Rough approximation}. We now set 
\[
F_n(x,y)=\frac{u(x)\,\eta_n(x)\,(1-\psi_n(x))-u(y)\,\eta_n(y)\,(1-\psi_n(y))}{|x-y|^{\frac{N}{p}+s}},\qquad \mbox{ for a.\,e. } (x,y)\in\mathbb{R}^N\setminus\mathbb{R}^N.
\]
We observe that this sequence is bounded in $L^p(\mathbb{R}^N\times\mathbb{R}^N)$, since 
\[
\begin{split}
\|F_n\|_{L^p(\mathbb{R}^N\times\mathbb{R}^N)}&= [u\,\eta_n\,(1-\psi_n)]_{W^{s,p}(\mathbb{R}^N)}\\
&\le [u\,\eta_n]_{W^{s,p}(\mathbb{R}^N)}+[u\,\eta_n\,\psi_n]_{W^{s,p}(\mathbb{R}^N)}\le C\,[u]_{W^{s,p}(\mathbb{R}^N)},\qquad \mbox{ for every } n\in\mathbb{N},
\end{split}
\]
thanks to \eqref{bound2} and \eqref{bound}.
Thus, up to a subsequence, it converges weakly in $L^p(\mathbb{R}^N\times\mathbb{R}^N)$. Such a limit must coincide with the pointwise limit, which by construction is given by the function
\[
F(x,y)=\frac{u(x)-u(y)}{|x-y|^{\frac{N}{p}+s}},\qquad \mbox{ for a.\,e. } (x,y)\in\mathbb{R}^N\times\mathbb{R}^N.
\] 
We can now appeal to Mazur's Lemma and infer existence of a sequence of convex combinations of $F_n$ which strongly converges to $F$ in $L^p(\mathbb{R}^N\times\mathbb{R}^N)$. Thanks to the peculiar form of $F_n$ and the properties of both $\psi_n$ and $\eta_n$, this implies that there exists a new sequence
\[
\phi_n\in C^\infty_0(B_{2n}(0)\setminus B_\frac{1}{n}(0)),
\] 
such that 
\[
\lim_{n\to\infty}\left\|\frac{u(x)\,\phi_n(x)-u(y)\,\phi_n(y)}{|x-y|^{\frac{N}{p}+s}}-\frac{u(x)-u(y)}{|x-y|^{\frac{N}{p}+s}}\right\|_{L^p(\mathbb{R}^N\times\mathbb{R}^N)}=0.
\]
This is the same as
\[
\lim_{n\to\infty} [u\,\phi_n-u]_{W^{s,p}(\mathbb{R}^N)}=0.
\]
Thus, $u$ is the strong limit of a sequence of $W^{s,p}(\mathbb{R}^N)$ function, with compact support which is distant from the origin.
\vskip.2cm\noindent
{\bf Smooth approximation}. Finally, from the previous point we know that there exists a sequence $\{u_n\}_{n\in\mathbb{N}}\subseteq W^{s,p}(\mathbb{R}^N)$ approximating $u$ in the Sobolev-Slobodecki\u{\i} seminorm, such that 
\[
u_n\equiv 0\quad \mbox{ in } B_\frac{1}{n}(0)\cup (\mathbb{R}^N\setminus B_{2n}(0)).
\]
We take $\{\rho_k\}_{k\in\mathbb{N}}$ the usual family of Friedrichs mollifiers and set 
\[
u_{n,k}=u_n\ast \rho_k.
\]
Thanks to the properties of convolutions, for every $k\in\mathbb{N}$ large enough we have that $u_{n,k}\in C^\infty_0(\mathbb{R}^N\setminus\{0\})$ and 
\[
\lim_{k\to\infty} [u_{n,k}-u_n]_{W^{s,p}(\mathbb{R}^N)}=0.
\]
In particular, for every $n\in\mathbb{N}\setminus\{0\}$, we can choose $k_n\in \mathbb{N}$ such that 
\[
[u_{n,k}-u_n]_{W^{s,p}(\mathbb{R}^N)}\le \frac{1}{n},\qquad \mbox{ for every } k\ge k_n.
\]
We thus have 
\[
[u_{n,k_n}-u]_{W^{s,p}(\mathbb{R}^N)}\le [u_{n,k_n}-u_{n}]_{W^{s,p}(\mathbb{R}^N)}+[u_{n}-u]_{W^{s,p}(\mathbb{R}^N)}\le \frac{1}{n}+[u_{n}-u]_{W^{s,p}(\mathbb{R}^N)}.
\]
This finally shows that $u$ is the strong limit of the sequence $\{u_{n,k_n}\}_{n\in\mathbb{N}}\subseteq C^\infty_0(\mathbb{R}^N\setminus\{0\})$. This concludes the proof.
\end{proof}
Finally, we extend to the space
\[
\mathcal{D}^{s,p}(\mathbb{R}^N)=\Big\{\varphi\in L^1_{\rm loc}(\mathbb{R}^N)\, :\, [\varphi]_{W^{s,p}(\mathbb{R}^N)}<+\infty\Big\},
\]
a simple yet useful result. The main focus is on dependence on $s$ of the constant appearing in the estimate below.
\begin{lemma}
\label{lm:prin}
Let $N\ge 1$, $0<s< 1$ and $1<p<\infty$  be such that $s\,p>N$. There exists a constant $C=C(N)>0$ such that
\[
\sup_{|h|>0} \int_{\mathbb{R}^N} \left|\frac{\varphi(x+h)-\varphi(x)}{|h|^s}\right|^p\,dx\le C\,s\,(1-s)\,[\varphi]^p_{W^{s,p}(\mathbb{R}^N)},\qquad \mbox{ for every } \varphi\in \mathcal{D}^{s,p}(\mathbb{R}^N).
\]
\end{lemma}
\begin{proof}
From Proposition \ref{prop:uguaspazi}, we know that $\mathcal{D}^{s,p}(\mathbb{R}^N)=\mathcal{W}^{s,p}(\mathbb{R}^N)$. Thus, for every $\varphi\in \mathcal{D}^{s,p}(\mathbb{R}^N)$, we have
\[
\widehat{\varphi}=\varphi-\varphi(0)\in \mathscr{D}^{s,p}_0(\mathbb{R}^N\setminus\{0\}),
\]
thanks to Proposition \ref{prop:homo}. 
By definition of $\mathscr{D}^{s,p}_0(\mathbb{R}^N\setminus\{0\})$, there exists a sequence $\{\varphi_n\}_{n\in\mathbb{N}}\subseteq C^\infty_0(\mathbb{R}^N\setminus\{0\})$ such that
\[
\lim_{n\to\infty} [\varphi_n-\widehat\varphi]_{W^{s,p}(\mathbb{R}^N)}=0.
\]
By applying Hardy's inequality for $\mathscr{D}^{s,p}_0(\mathbb{R}^N\setminus\{0\})$ to each function $\varphi_n-\widehat{\varphi}$, we have
\[
\mathfrak{h}_{s,p}(\mathbb{R}^N\setminus\{0\})\,\int_{\mathbb{R}^N} \frac{|\varphi_n-\widehat{\varphi}|^p}{|x|^{s\,p}}\,dx\le [\varphi_n-\widehat{\varphi}]_{W^{s,p}(\mathbb{R}^N)}.
\]
The last two equations show that $\{\varphi_n-\widehat\varphi\}_{n\in\mathbb{N}}$ converges to $0$ in the weigthed space $L^p(\mathbb{R}^N;|x|^{-s\,p})$. Thus, we can further suppose that
\[
\lim_{n\to\infty} |\varphi_n(x)-\widehat{\varphi}(x)|=0,\qquad \mbox{ for a.\,e. }x\in\mathbb{R}^N,
\]
up to a subsequence. We now use that
\[
\sup_{|h|>0} \int_{\mathbb{R}^N} \left|\frac{\varphi_n(x+h)-\varphi_n(x)}{|h|^s}\right|^p\,dx\le C\,s\,(1-s)\,[\varphi_n]^p_{W^{s,p}(\mathbb{R}^N)},\qquad \mbox{ for every } n\in\mathbb{N},
\]
thanks to \cite[Theorem 1.1]{dTGCV} (see also \cite[Lemma A.1]{BLP}). By taking the limit as $n$ goes to $\infty$ and using Fatou's Lemma in the left-hand side, we get the desired conclusion.
\end{proof}

\end{document}